\documentclass{article}
\usepackage{booktabs}
\usepackage{etoolbox}
\usepackage{tikz}
\usetikzlibrary{calc}
\patchcmd{\thebibliography}
  {\settowidth}
  {\setlength{\itemsep}{-3pt }\settowidth}
  {}{}
\apptocmd{\thebibliography}
  {\small}
  {}{}

\usepackage{multirow}
\usepackage[section]{algorithm}
\usepackage{algpseudocode}  
\usepackage[utf8]{inputenc}
\usepackage{ragged2e}
\usepackage{enumitem}
\newlist{steps}{enumerate}{1}
\setlist[steps, 1]{label = Step \arabic*:}
\usepackage[english]{babel}
\usepackage[T1]{fontenc}
\usepackage{mathtools}

\usepackage{xcolor}
\usepackage{caption}
\usepackage[a4paper,top=3cm,bottom=2cm,left=3cm,right=3cm,marginparwidth=1.75cm]{geometry}
\usepackage{mathrsfs, amsthm}
\usepackage{scalerel}
\allowdisplaybreaks

\usepackage{enumitem}
\usepackage[T1]{fontenc}
\setlist[steps, 1]{label = Step \arabic*:}
\usepackage{changepage}

\usepackage{comment}
\usepackage{amsmath,amssymb}
\usepackage{graphicx}
\usepackage[colorlinks=true, allcolors=blue]{hyperref}
\usepackage[english]{babel}
\usepackage{bbm}

\newtheorem{theorem}{Theorem}[section]

\newtheorem{lemma}{Lemma}[section]
\newtheorem{proposition}{Proposition}[section]
\newtheorem{assumption}{Assumption}[section]

\theoremstyle{definition}
\newtheorem{example}{Example}[section]
\newtheorem{experiment}{Experiment}[section]

\counterwithin{figure}{section}
\counterwithin{table}{section}

\newtheorem{remark}{Remark}[section]

\numberwithin{equation}{section}
\DeclareMathOperator{\dist}{dist}

\def\E{\mathbb{E}} 
\def\R{\mathbb{R}}

\title{Numerical integrators for confined Langevin dynamics}
\author{B. Leimkuhler\thanks{%
School of Mathematics and the Maxwell Institute for Mathematical Sciences, University of Edinburgh, UK; B.Leimkuhler@ed.ac.uk} 
\and A. Sharma\thanks{Department of Mathematical Sciences, Chalmers University of Technology and the University of Gothenburg, Sweden; akashs@chalmers.se} 
\and M.V. Tretyakov\thanks{%
School of Mathematical Sciences, University of Nottingham,
 UK; Michael.Tretyakov@nottingham.ac.uk}}

\begin{document}

\maketitle

\begin{abstract}
 We derive and analyze  numerical methods for  underdamped (kinetic) Langevin dynamics in a domain with elastic reflection at the boundary. 
 First-order approximations are based on an Euler-type scheme incorporating collision-handling at the boundary. To achieve second order,  composition schemes are derived based on decomposition of the generator into collisional drift,  impulse, and stochastic momentum evolution.  In a deterministic setting, this approach would typically lead to first-order approximation, even in symmetric compositions, but we find that the stochastic  method can provide second-order weak approximation  with a single gradient evaluation, both at finite times and in the ergodic limit.   We provide  analysis of this observation, as well as numerical demonstration, and we compare and contrast the performance of different variants of the integration method using model problems.
 
       \medskip
       

     
\end{abstract}

\section{Introduction}

There is great current demand for reliable numerical methods for solving stochastic differential equations (SDEs). In part this can be attributed to their role in large scale statistical computation, where long stochastic paths are often used to calculate ergodic averages.   Many standard algorithms for SDEs are formulated on unbounded domains, e.g. Euclidean space, but in practice are subject to inequality constraints (e.g. positivity relations) or are employed with more complicated domain restrictions, often implemented in an ad hoc manner.  In this article, we explore the rational design of numerical algorithms for accurate long-term simulation of Langevin SDEs with domain restriction, including sampling from stationary measures with compact support.

Let $G \subset \mathbb{R}^{d}$ be a bounded domain with sufficiently smooth boundary $\partial G$, $n_{G}(q)$ be the outward normal unit vector at $q \in \partial G$, and $I_{\partial G}(q) $ be the indicator function of $q \in \partial G$. Let $(x \cdot y)$ denote the scalar product between two vectors $x$, $y \in \mathbb{R}^d$.  Let $\sigma$ be a positive constant, $b: \bar{G} \times \mathbb{R}^{d} \rightarrow \mathbb{R}^{d}$, and
consider the confined Langevin dynamics (CLD) governed by the  stochastic differential equations (SDEs):
\begin{equation}
\tag{CLD}
\begin{aligned}
    Q(t) = Q(0) + &\int_{0}^{t}P(s)ds,\;\;\;\;
    P(t) = P(0) +\int_{0}^{t}b(Q(s), P(s))ds  + \sigma W(t) + \mathcal{R}(t),\\
    \mathcal{R}(t) &= -\sum\limits_{0<s\leq t}2(P({s^{-}})\cdot n_{G}(Q(s)))n_{G}(Q(s))I_{ \partial G}(Q(s)), \;\;\; t \ge 0,
\end{aligned}\label{cld}
\end{equation}
where $W(t)$ is a $d$-dimensional standard Wiener process.

The process $Q(t)$ does not leave the set $\bar{G}$ and it elastically reflects from the boundary $\partial G$ \cite{a19,a14,a13} which is achieved due to the c\`{a}dl\`{a}g process $\mathcal{R}(t)$. At the collision of the position $Q(t)$ with $\partial G$,  $\mathcal{R}(t)$ acts on the momentum $P(t)$ to reverse the direction of its component orthogonal to $\partial G$ keeping the tangential component unchanged. The summation in $\mathcal{R}(t)$ is understood as a sum of jumps (in other words, discontinuities) in the momentum $P(s)$, which occur when $Q(s)$ hits the boundary at random times within $[0,t]$. More precisely, let $\chi$ be the number of collisions of a trajectory $Q(s)$, $s \in [0,t]$, with points on the boundary $\partial G$,  and let $\tau_i$, $i=1,\ldots, \chi$, be the corresponding collision times (i.e., such random times that $Q(\tau_i) \in \partial G$), then  
\[
\mathcal{R}(t) = -\sum\limits_{i=1}^{\chi}2(P({\tau_i^{-}})\cdot n_{G}(Q(\tau_i)))n_{G}(Q(\tau_i))I_{ \partial G}(Q(\tau_i)),
\]
where, as usual, $P({\tau_i^{-}})$ denotes the value of the momentum before the jump at $\tau_i$. 


Let $ b(q,p) := -\nabla_q U(q) - \gamma p$  and  $\sigma = \sqrt{\frac{2\gamma}{\beta}}$,   then consider the particular case of  confined Langevin dynamics (\ref{cld}) defined in the set $\bar G\times \mathbb{R}^{d}$:
\begin{equation}
\tag{ECLD}
\begin{aligned}
    Q(t) &= Q(0) + \int_{0}^{t}P(s)ds, \\ 
    P(t)  &= P(0) -\int_{0}^{t}\nabla_q U(Q(s))ds - \gamma\int_{0}^{t} P(s)ds + \sqrt{\frac{2\gamma}{\beta}}W(t) + \mathcal{R}(t),\\
    \mathcal{R}(t) &= -\sum\limits_{0<s\leq t}2\big(P(s^{-})\cdot n_{G}(Q(s))\big)n_{G}(Q(s))I_{ \partial G}(Q(s)), 
\end{aligned}
\label{ecld}
\end{equation}
where $\beta >0$ is interpreted as the inverse temperature in molecular dynamics and $\gamma >0$ is a friction parameter.  Eq. (\ref{ecld}) is ergodic with the invariant density $\rho(q,p) = \frac{1}{Z}e^{-\beta(U(q) + |p|^{2}/2)}$, $(q,p) \in \bar G \times \mathbb{R}^{d}$, under suitable assumptions on the potential $U(q)$ and the domain $G$ (see the corresponding discussion in Section~\ref{sec:erg}).

\subsection{Motivation and related work}
The {\em central theme} of this paper is the efficient weak-sense numerical approximation of (\ref{cld})  including the computation of ergodic limits associated with (\ref{ecld}) which is motivated by applications such as those discussed below.

\paragraph{Sampling}  Sampling from a desired distribution is of prime importance for computational statistics and various machine learning applications. Domain restriction arises, for example, in truncated data problems for survival time studies \cite{survanalysis}, ordinal data models \cite{odm}, constrained Lasso and Ridge regression \cite{constlassoridge}, Latent Dirichlet Allocation \cite{lda}, non-negative matrix factorization \cite{matrixfactor}, and stochastic optimal control \cite{Huynh_Karaman_Frazzoli_2016}. 
In \cite{lst23} (see also references therein) the task of sampling from a given distribution with compact support is addressed using reflected Brownian dynamics (in other words, \textit{overdamped} Langevin dynamics with reflection), where the dynamics are confined within $\bar{G}$ thanks to making use of the local time. In the setting of sampling in Euclidean space, the advantages of  underdamped Langevin equations vs  overdamped Langevin dynamics are well known (see e.g. \cite{le_ma_st2016,Ruslan15}) and these benefits can be expected to carry over to the confined-space setting considered in this article.   

\paragraph{Molecular dynamics} To study molecular properties of a given substance at constant temperature, Langevin dynamics is often employed in which positions $q$ and momenta $p$ evolve under forcing due to a potential energy function $U(q)$ together with the stochastic and dissipative forces associated to interactions with a thermal bath. These properties can be expressed as ergodic limits (see \cite{a11,leimkuhler_mathews_15} for more details). A steep (often singular) penalty term may be included in the potential energy  to  prevent particles from leaving the confined domain. The barrier leads to the  drift $- \nabla_q U(q)$ in the Langevin equations increasing rapidly near the boundary of the confined space which can negatively affect stability of numerical methods, or require the use of a variable stepsize \cite{Le_et_al_2024}. Confined Langevin dynamics (\ref{ecld}) serves the purpose of simulating molecular configurations in a restricted region, with atoms undergoing perfectly elastic collisions at the walls, without causing stability problems or requiring stepsize restriction, in contrast to the case of using a steep penalty. Furthermore, confined Langevin dynamics models a natural physical behavior of particles (elastic collisions) at the boundary. In the kinetic theory of gases, for example, specular reflection models are used to describe reflection of gas particles at a totally elastic solid surface \cite{a8}. 

\paragraph{Computational fluid dynamics} The SDEs (\ref{cld})  find applications in modelling discrete particles immersed in turbulent flows inside a confined domain, where colloidal particles are carried by the fluid flow, experiencing a drag force and collisions with the domain boundary. Numerical experiments using models of (\ref{cld}) type allow to study, e.g., convection, polydispersity etc. related to mesoscopic particles suspended in turbulent flow (see e.g., \cite{wall_coll_2019,a27, bossy23} for more details).  

\paragraph{Optimization} Minimizing a loss/error function is at the core of machine learning. In many cases,  constraints are an essential part of model specification, for example reflecting the positivity of certain model parameters or interval restrictions. Moreover, constraining the parameter set can result in improvement of the performance of the model due to the regularization effects of constraints, e.g. \cite{Sangalli2021ConstrainedOT, leimkuhler_constraint_net}. One can exploit  (\ref{ecld}) to minimize a function $U(q)$, where $q$ satisfies constraints represented by domain $G$. 

\paragraph{Specular boundary value problems} Via the Feynman-Kac formula \cite{a13,a10}, the SDEs (\ref{cld}) provide a probabilistic representation of the  specular boundary value problem for linear parabolic partial differential equations (PDEs) (see (\ref{eq2.1})-(\ref{eq2.3})). Solving such PDEs in high dimensions using deterministic methods is computationally infeasible. The proposed numerical methods for (\ref{cld}) given in this paper  may be used together with the Monte-Carlo technique to solve such high-dimensional problems.

The above applications strongly motivate the construction and study of numerical methods for (\ref{cld}). Despite this,  very limited attention has been paid so far to numerical integration of confined Langevin dynamics. In this respect, in \cite{a10} the authors investigated a numerical scheme when $G=\mathbb{R}^{d-1} \times(0,\infty)$ and proved its first order weak convergence. 

\subsection{Contributions and overview} In this paper, we derive first-order weak methods to numerically approximate (\ref{cld}) as well as second-order schemes in the case of (\ref{ecld}). We establish the order of convergence of first-order schemes in both the finite and infinite time settings. The major obstacle is to obtain an estimate related to the number of collisions of the walker with the boundary (see Lemmas~\ref{bl} and~\ref{ebl}).  

To construct splitting schemes, we decompose the generator of the Markov process $(Q(t),P(t))$ into components representing collisional drift (A$_c$), impulse (B), and stochastic momentum evolution (O).  (For details of these maps, see Section 3.) We highlight the following fact about second-order symmetric integrators for (\ref{ecld}). In the deterministic setting, the discussed splitting approach would typically lead to first-order approximation, even in symmetric compositions \cite{BrianBen2000,leimkuhler_mathews_15}. Stochastic numerics experience \cite{milstein2021stochastic} then suggests that finite-time convergence of  stochastic generalizations of such deterministic methods to (\ref{ecld}) should be of weak order one as well. However, it turns out  that the stochastic  method with a single gradient evaluation has second-order convergence (see Table~\ref{table_1.1}). Simply put, this counterintuitive  result is due to the average collision time (within a step of the numerical method in which a collision is encountered) being approximately at the midpoint of the step (see the discussion in Section~\ref{sec:2ndorder}).   
We also make the observation (see Appendix~\ref{sec:colisHD}) that if we randomize the initial condition of otherwise deterministic Hamiltonian dynamics with elastic collisions, we see the same effect as for Langevin equations: geometric integrators such as [BA$_c$B] become 2nd order.  

\begin{table}[htbp] 
\centering
\caption{Note that [OBA$_c$BO] degenerates to the Störmer-Verlet method [BA$_c$B] in the absence of the noise and damping, and A$_c$ becomes A in the absence of collisions. }
\begin{tabular}{c|c|c}\label{table_1.1}
  Method & Without collision &  With collision  \\
\hline\rule{0pt}{5pt}
Störmer-Verlet &  $\mathcal{O}(h^2)$  &  $\mathcal{O}(h)$ \\
\hline
OBA$_c$BO & $\mathcal{O}(h^2)$ & $\mathcal{O}(h^2)$ \\
\end{tabular}
\end{table}

The second-order convergence of the symmetric splitting scheme for  \eqref{ecld} is interesting from another perspective as well. Approximating SDEs in bounded domains typically requires a special construction near the boundary to preserve the convergence order achieved in the whole of $\mathbb{R}^d$. The canonical examples are the projection and penalty Euler methods for simulating reflected SDEs, which exhibit only half-order weak convergence -- lower than the $\mathbb{R}^d$ counterpart.  Therefore the Euler scheme requires specialized procedures on or near the boundary to preserve  first order convergence, for example see \cite{1, 64, 2, lst23, sharma2025stickySDE, sharma2022random}.  In contrast, straightforward specular reflection of momentum in approximating Markov chain provides weak first-order  in the case of Euler-type schemes (see Subsection~\ref{sec:1st}) with second-order convergence achievable for splitting schemes (see Subsection~\ref{sec:2nd}), respectively, for \eqref{ecld}. 

In Section~\ref{cld_sec3} we present the numerical integrators and also state theoretical results on error bounds for the methods, both when the SDEs are integrated over a finite time interval and when associated ergodic limits are approximated. Proofs of the error bounds are given in Section~\ref{sec:proofs} with necessary background material for (\ref{cld}) and (\ref{ecld}) provided in Section~\ref{sec:pre}. The main difficulty in proving convergence arises due to the need to quantify the number of steps with collision. In Section~\ref{sec:tests} we present the results of numerical tests which confirm our theoretical predictions and compare the methods from  Section~\ref{cld_sec3}. We also include experiments with a restricted Neal's funnel and Bayesian inference for the SIR model of disease spread. In Section~\ref{sec:concl} we summarize this work and suggest topics of interest for further research related to confined Langevin dynamics.

\section{Preliminaries}\label{sec:pre}

\textbf{Notation. }
For the boundary sets we write:
    $\Sigma^{+} = \{  (q,p) \in \partial G \times \mathbb{R}^{d}; (p\cdot n_{G}(q)) >0\}$ and $\Sigma_{t}^{+} = (0,t)\times \Sigma^{+}$. The space $C^{l_{1},l_{2},l_{3}}([0,T]\times A\times \mathbb{R}^{d})$ denotes the space of  functions whose continuous derivatives exist with respect to $t \in [0,T] $, $q \in A$ and $p\in \mathbb{R}^{d}$ up to order $l_{1}$, $l_{2}$ and $l_{3}$, respectively; $C^{l_{1},l_{2}}(A\times \mathbb{R}^{d})$  has  the same meaning but with respect to $q$ and $p$;  $ C_{b}^{l_1,l_2}(A\times \mathbb{R}^{d})$  is a space of functions in  $C^{l_{1},l_{2}}(A\times \mathbb{R}^{d})$ with bounded  derivatives up to order $l_1$ in $q$ and $l_2$ in $p$; and  $ C_{c}^{l_1,l_2}(A\times \mathbb{R}^{d})$ are  functions in  $C^{l_{1},l_{2}}(A\times \mathbb{R}^{d})$ with compact support. 
As usual, $|\cdot |$ denotes the $L^2$-norm of a vector in $\mathbb{R}^{d}$.
For a differentiable vector-valued function  $F(x)$ from $\bar{G}\subset \mathbb{R}^{d}$ or $\mathbb{R}^{d}$ to $\mathbb{R}^{d}$, $\textbf{J}_{F}(x)$ denotes the Jacobian matrix of $F$ evaluated at $x$.

\medskip

To prove the first-order weak-sense convergence of the schemes proposed in Section~\ref{sec:1st}, we need assumptions on the boundary $\partial G$ of $G$ and the drift coefficient $b(q,p)$.
\begin{assumption}\label{cld_as:1}
We assume the boundary $\partial G$ of set $G$ belongs to $C^{4}$.
\end{assumption}  
\begin{assumption}\label{as:2}
The coefficient $b(q,p)$ is a $C^{2,2}(\bar{G} \times \mathbb{R}^{d})$ function which grows at most linearly in $p$ at infinity, i.e. there exists a constant $K > 0$ such that,
\begin{align}\label{linegb}
    |b(q,p)| \leq K(1+|p|),
\end{align}
for all $p \in \mathbb{R}^{d}$ and $q \in \bar{G}$.
\end{assumption}

In \cite{a19}, existence of weak solutions and uniqueness in law of (\ref{cld}) in $G=(0,\infty)\times \mathbb{R}^{d-1}$ with multiplicative noise and  $b(q,p) = b(q) - \gamma p$, $\gamma > 0$,  were established under boundedness assumptions on $b(q)$ and  the diffusion coefficient $\sigma(q)$. 
In \cite{a14}, existence of weak solution of (\ref{cld}) in $G=(0,\infty)\times \mathbb{R}^{d-1}$  with bounded drift coefficient $b$ was shown. 
This  well-posedness result was extended in \cite{a13} to the case of $\partial G \in C^{3}$ being a compact submanifold of $\mathbb{R}^{d}$, by establishing the existence of a weak solution and its pathwise uniqueness for $b(q,p) \equiv 0$, and then, by using  Girsanov's theorem, for any bounded $b(q,p)$. We note that Girsanov's theorem can be applied if $b(q,p)$ has linear growth in $p$ (see \cite[Proposition 5.3.6]{a20}).  Pathwise uniqueness of the solution of (\ref{cld}) with $b(q,p) \equiv 0$ is proved in \cite{a18} when $ G$ is a convex polytope.

\subsection{Kolmogorov equations for (\ref{cld})}

Consider the backward Kolmogorov equation
\begin{align}
    \frac{\partial u}{\partial t}  + \frac{\sigma^{2}}{2}\Delta_{p}u + (p\cdot\nabla_{q}u) + (b(q,p)\cdot \nabla_{p}u)   = 0, \;\;\;\;(t,q,p) \in (0,T)\times G \times \mathbb{R}^{d}, \label{eq2.1} 
\end{align}
with terminal condition
\begin{equation}
 u(T,q,p) = \varphi(q,p),\;\;\;\; (q,p) \in G\times \mathbb{R}^{d},\label{eq2.2}
\end{equation}
and specular boundary condition
\begin{equation}
  u(t,q,p) = u(t,q,p-2(p\cdot n_{G})n_{G}),\;\;\;\; (t,q,p) \in \Sigma_{T}^{+},  \label{eq2.3}
\end{equation}
where $\Delta_{p}$ is the Laplacian with respect to $p$, $\nabla_{q}$ and $\nabla_{p}$ are gradients with respect to $q$ and $p$, respectively, and $n_{G}:= n_{G}(q)$. 

The solution of (\ref{eq2.1})-(\ref{eq2.3}) has the probabilistic representation  \cite{a13,a10}:
\begin{equation}
    u(t_{0},q_{0},p_{0}) = \mathbb{E}\varphi(Q_{t_0,q_0,p_0}(T),P_{t_0,q_0,p_0}(T)), 
\end{equation}
where $(Q_{t_0,q_0,p_0}(t),P_{t_0,q_0,p_0}(t))$, $t \ge t_0$, solves (\ref{cld}) with the starting point $(t_{0},q_{0},p_{0})$.

Below we give three examples of the specular boundary value problem (\ref{eq2.1})-(\ref{eq2.3}), for which we can obtain solutions in explicit form. 
\begin{example}
Take $G =(0,\infty)$ and $\partial G = \{ 0\}$, $ b(q,p) = \Big(\frac{1}{q+c} + p\Big)$, $\sigma  =1$, $\varphi(q,p) = (q+c)^{2}e^{-p^{2}}$, and boundary condition $u(t,q,p) = u(t,q,-p)$ at $\Sigma_{T}^{+}$, where $c \ge 0$. The solution is given by $u(t,q,p) = (q+c)^{2}e^{-p^{2}}e^{-(T-t)}$.
\end{example}

\begin{example}\label{example2.2}
Let $G \subset \mathbb{R}^{d}$, $U(q) \in C^{1}(\bar{G})$ and $\alpha,\; \beta  > 0$. Take $b(q,p) = -\nabla_{q}U(q) + \alpha p$, $\sigma = \sqrt{2\alpha/\beta}$, $\varphi = e^{-\beta\big(|p|^{2}/2 + U(q)\big)} $ and boundary condition is $u(t,q,p) = u(t,q,p-2(p\cdot n_{G})n_{G}), (t,q,p) \in \Sigma_{T}^{+}$. The solution is $u(t,q,p)=\exp{\Big(-\beta\big(|p|^{2}/2 + U(q)\big)-\alpha d(T-t)\Big)}$.
\end{example}

\begin{example}
Let $G \subset \mathbb{R}^{d}$, $b(q,p) = -\frac{q\cos{|q|^{2}}+2(d+2)p}{2|p|^{2} + 1} $, $\sigma=\sqrt{2}$, and $\varphi(q,p) = |p|^{4} + |p|^{2} + \sin{|q|^{2}}$. The boundary condition is $u(t,q,p) = u(t,q,p-2(p\cdot n_{G})n_{G}),\; (t,q,p) \in \Sigma_{T}^{+}$. Then, the solution is $u(t,q,p)=|p|^{4} + |p|^{2} + \sin{|q|^{2}} + 2d(T-t)$.
\end{example}

In the above examples, the condition $u(t, q, p ) = u(t,q, p - 2 (\tilde{n}(q) \cdot p))  $, where $\tilde{n}$ is smooth extension of $n$, is satisfied in the entire domain $\bar G$. We present an example below for which this is not the case. 

\begin{example}
Take $G = (0, \infty)$ with $\partial G = \{ 0\}$.  Consider the PDE:
\begin{align}
    \frac{\partial}{\partial t}u + p \frac{\partial u}{\partial q} - q \frac{\partial u}{\partial p} - p \frac{\partial u}{\partial p} + \frac{\partial^2 u}{\partial p^2} = 0, \quad (t, q, p) \in (0,T) \times G \times \mathbb{R},
\end{align}
with 
\begin{align}
    u(T,q,p) = q^2-1,\quad   (q, p) \in \times G \times \mathbb{R},
\end{align}
and
\begin{align}
 u(t, 0, p) = u(t, 0, - p), \quad     (t, p) \in (0,T)\times \mathbb{R}.
\end{align}
It can be verified that 
\begin{align*}
    u(t,q,p) &= \frac{2}{3}(q^2 + p^2 + qp - 2)e^{-(T-t)}    \\ 
    & \quad + \frac{1}{3} (q^2 - 2p^2 - 2 q p + 1)e^{-(T-t)}\cos(\sqrt{3}(T-t)) \\ & \quad + \frac{1}{\sqrt{3}}( 2 qp  + q^2 - 1) e^{-(T-t)}\sin(\sqrt{3}(T-t)).
\end{align*}
\end{example}

We make the following assumptions on the terminal function $\varphi$ and the solution of (\ref{eq2.1})-(\ref{eq2.3}).
\begin{assumption}\label{cld_as:3}
$\varphi (q,p) \in C^{4 ,4 }(\bar{G}\times \mathbb{R}^{d})$.
\end{assumption}
\begin{assumption}\label{cld_as:4a}
Let $G$ be a bounded domain. The solution of (\ref{eq2.1})-(\ref{eq2.3}) exists and  belongs to $C^{2,2,4}([0,T)\times G\times \mathbb{R}^{d})$ and $C^{0,1,0}([0,T]\times \bar{G} \times \mathbb{R}^{d})$, 
satisfying
\begin{align}\label{cld_der_bound}
   \sup_{(t,q) \in[0, T) \times G}\bigg( \sum\limits_{l = 1}^{2}\sum\limits_{i + |j| = l }&|D_{t}^{i}D_{q}^{j}u(t,q,p)| +  \sum\limits_{l = 1}^{4}\sum\limits_{|j| = l}|D_{p}^{j}u(t,q,p)| \bigg) \leq C(1 + |p|^{2m}), 
\end{align}
where $D_{x}^{j}u = \frac{\partial^{|j|}u}{\partial x^{j_{1}},\dots,\partial x^{j_{d}}} $ and $j$ is a multi-index.

\end{assumption}

The schemes introduced in this paper work for bounded and unbounded domains with boundaries. For ease in the arguments in the proofs of main results, we consider the case of bounded domains unless specified otherwise.    
We note that in the case of unbounded domains to use the considered explicit schemes, 
we could impose the global Lipschitz condition on the gradient of the potential $\nabla U(q)$  or exploit the concept of rejecting exploding trajectories from \cite{a23,40,milstein2021stochastic}. 

There is a dearth of regularity results for solutions of the specular boundary value problem (\ref{eq2.1})-(\ref{eq2.3}) available in the literature. The above four examples demonstrate that there is a sufficiently large class of specular boundary value problems whose solutions satisfy Assumption~\ref{cld_as:4a}.

It is shown in recent work (see \cite{PDE25}) that the problem  (\ref{eq2.1})-(\ref{eq2.3}) has a classical solution.  More precisely (see Corollary~5.5 in \cite{PDE25}), assume that  $\partial G \in C^{9/2}$, $ b(q,p) \in C_{b}^{1,3}(\bar G \times \mathbb{R}^{d})$, and $\varphi (q,p) \in C_b^{4/3 ,4 }(\bar{G}\times \mathbb{R}^{d})$ and  agrees with the specular boundary condition, 
then $u \in C_b^{2,4/3,4}([0,T]\times \bar{G}\times \mathbb{R}^{d})$
and the corresponding derivatives are Lipschitz continuous.  It is also proved in \cite{PDE25} that outside the grazing set $\{(q,p) \in \partial G \times \mathbb{R}^{d} : (p \cdot n_{G}(q)) =0\}$ the solution $u(t,q,p)$ is arbitrary smooth assuming that $\partial G$, $b$ and $\varphi$ are sufficiently smooth.
We note that the examples presented earlier in this section are smooth in the whole domain $[0,T]\times \bar{G}\times \mathbb{R}^{d}$, which does not contradict counterexamples in Section~6 of \cite{PDE25}, 
i.e., the non-compatibility-type condition of Theorem 6.1 is not satisfied. 
Also, it was shown in \cite{a10} that if $\varphi \in C_{c}^{1,1}(G \times \mathbb{R}^{d})$ and $b \in C_{b}^{1,1}(\bar G \times\mathbb{R}^{d})$, then $u(t,q,p) \in C([0,T]\times \bar{G}\times \mathbb{R}^{d})$ and $\nabla_{q}u$ and $\nabla_{p}u$ exist and belong to $C((0,T)\times \bar{G} \times \mathbb{R}^{d})$ for $G = \{(q^1, \dots, q^d)\; ;\; q^d>0 \}$. 

We emphasize that Assumption~\ref{cld_as:4a} is used for proving first-order weak convergence of the schemes introduced in Section~\ref{sec:1st} and higher regularity is needed to prove second-order convergence for the schemes from Section~\ref{sec:2nd}. At the same time, the schemes can be used in practice under weaker assumptions. It is of interest also to study convergence of the presented integrators under some weaker assumptions which might be an interesting subject for future work.  

The Fokker-Plank (i.e., forward Kolmogorov) PDE with specular reflection has the form
\begin{align}
     \frac{\partial \rho}{\partial t} &=  -(p \cdot \nabla_q \rho) - \nabla_p \cdot (\rho b(q,p))  + \frac{\sigma^2}{2} \Delta_p \rho, \quad t >0,\,\, q \in G, \,\, p \in \mathbb{R}^d, \\
     \rho(t, q, p) &= \rho(t, q , p - 2(n(q)\cdot p)p), \quad (q,p) \in \Sigma^{+}, \\
     \rho (0,q,p) &= \delta(q-q_0) \delta(p-p_0), \quad q \in G, \,\, p \in \mathbb{R}^d, 
\end{align} 
for some $q_0 \in G$, $p_0 \in \mathbb{R}^d$. By \cite{PDE25} this PDE problem has a classical solution under our assumptions. 

\begin{example}
Consider the case of \eqref{ecld} with constant potential, $d=1$, and $\beta = 1$. The corresponding Fokker-Plank equation is
\begin{align}\label{fkp1eq}
     \frac{\partial \rho}{\partial t} = -p \frac{\partial \rho}{\partial q} + \gamma \frac{\partial}{\partial p} (p\rho) + \gamma \frac{\partial^2 \rho}{\partial p^2}, \quad  q \in (0,1), \,\, p \in \mathbb{R},\; t > 0,
\end{align}
with specular boundary and initial conditions
\begin{align}\label{fkp1bc}
    &\rho(t, q , p) = \rho(t, q, -p), \quad   q \in \{0,1\}, \,\, p>0,\\
    &\rho (0,q) = \delta(q-q_0) \delta(p-p_0), \quad q \in (0,1), \,\, p \in \mathbb{R}. 
\end{align}
One can show (cf. \cite{FPEexact93}) that the solution of this linear kinetic Fokker-Plank PDE with specular boundary condition is 
\begin{align}\label{sol_fkp1}
  \rho(t,q, p)=  \rho(t,q, p; q_0, p_0) = \sum_{n=-\infty}^{\infty} \left[ F(t, q, p; q_0 + 2n, p_0) 
    + F(t,q, p; -q_0 + 2n, -p_0) \right],
\end{align}
where $F(t,q, p ; q_{0}, p_{0})$ is the fundamental solution of the PDE in the unbounded domain (i.e., $q,p \in  \mathbb{R}$) for a particle starting 
at $(q_{0}, p_{0})$ defined as
\begin{align*}
   F(t,q, p; q_{0}, p_{0}) = \frac{1}{2\pi \sqrt{\det(\boldsymbol{\Sigma}(t))}} \exp\left(-\frac{1}{2} (\mathbf{z} - \boldsymbol{M}(t,q_0, p_0))^T \boldsymbol{\Sigma}(t)^{-1} (\mathbf{z} -  \boldsymbol{M}(t,q_0, p_0))\right) .
\end{align*}
Here $\textbf{z} = (q, p) $, mean vector $\boldsymbol{M}(t,q_0, p_0)$ and covariance matrix $\boldsymbol{\Sigma}(t)$ are given by:
\begin{align*}
    \boldsymbol{M}(t,q_0, p_0) =  \begin{pmatrix} q_{0} + \frac{p_{0}}{\gamma}(1 - e^{-\gamma t}) \\ p_{0} e^{-\gamma t} \end{pmatrix}, \,\,\,\,
    \boldsymbol{\Sigma}(t) = \begin{pmatrix} \sigma_{qq}(t) & \sigma_{qp}(t) \\ \sigma_{pq}(t) & \sigma_{pp}(t) \end{pmatrix}.
\end{align*}
The components of the covariance matrix are:
\begin{align*}
    \sigma_{qq}(t) &= \frac{2}{\gamma^2 } \left(\gamma t - \frac{3}{2} + 2e^{-\gamma t} - \frac{1}{2}e^{-2\gamma t}\right) ,\\
    \sigma_{pp}(t) &=  (1 - e^{-2\gamma t}) , \,\,
    \sigma_{qp}(t) = \sigma_{pq}(t) = \frac{1}{\gamma} (1 - e^{-\gamma t})^2 .
\end{align*}
We recall the important  identity from the method of images that for any $(q_0, p_0)$
$$F(t,1, p ; q_0, p_0) = F(t,1, -p ; 2-q_0, -p_0),$$
which can be used to verify that \eqref{sol_fkp1} satisfies \eqref{fkp1bc}. 
Further, we note that the solution $\rho \in C^\infty((0,\infty)\times[0,1]\times\mathbb{R})$, which in particular means that the corresponding backward Kolmogorov equation with smooth terminal condition $\phi (q,p)$ is smooth on $[0,T]\times [0,1]\times\mathbb{R}$, thus adding another instance to our collection of examples of backward Kolmogorov equations with specular boundary condition satisfying Assumption~\ref{cld_as:4a}.
\end{example}

\subsection{Ergodicity}\label{sec:erg}

Consider (\ref{ecld}) which is a particular case of (\ref{cld}).

\begin{assumption}\label{as:5}
$U(q) \in C^{\infty}(\bar{G})$ and $U(q) > 1$ for all $q \in \bar{G}$.
\end{assumption}
The condition  $U(q) >1$ is imposed for convenience in the analysis but is not necessary and can be relaxed as can be seen in our experiments in Section~\ref{sec:tests}. 

Exploiting the approach used in proving Theorem~3.2 in \cite[Section 3]{MSH02} (see also \cite{TAL02}) one may establish under Assumptions~\ref{cld_as:1} and~\ref{as:5} (note that $\nabla_q U(q)$ is bounded thanks to $G$ being bounded) that (\ref{ecld}) is geometrically ergodic, i.e.  
there exists a unique invariant measure $\mu(dq,dp)$ defined on $G\times \mathbb{R}^{d}$ and  the following inequality holds for any $\varphi (q,p)$, $(q,p) \in \bar G \times \R^d$, bounded in $q$ and with not faster than polynomial growth in $p$  and for any initial point $(q_{0},p_{0})$ and all $T > 0$:
\begin{equation}
    |\mathbb{E}(\varphi(Q(T),P(T))) - \bar{\varphi}| \leq Ke^{-\lambda T},\label{ergo}
\end{equation}
where 
\begin{equation}\label{barvarphi}
\bar{\varphi} = \int_{\mathbb{R}^{d}}\int_{G} \varphi(q,p) \mu(dq, dp). 
\end{equation}
Moreover, the invariant measure $\mu$ has a probability density $\rho_{\infty}$ with respect to Lebesgue measure. In \cite[Section 9.1.5]{a22}, geometric ergodicity is proved in the case of the semi-space $G =(0,\infty) \times \mathbb{R}^{d-1}$.

The stationary Fokker-Plank equation for the density $\rho_{\infty}$ has the form
\begin{equation}
    \frac{\gamma}{\beta} \Delta_{p}\rho_{\infty}-(p\cdot\nabla_{q}\rho_{\infty}) + (\nabla_{q}U(q)\cdot\nabla_{p}\rho_{\infty}) + \gamma (\nabla_{p}\cdot(p\rho_{\infty}))  = 0,\;\;\;\; (q,p) \in G\times \mathbb{R}^{d},  \label{sfkp}
\end{equation}
with boundary condition
\begin{equation}\label{bcsfkp}
    \rho_\infty(q,p) = \rho_{\infty}
    (q,p-2(n_{G}\cdot p)n_{G}), \;\;\;\; (q,p) \in \Sigma^{+}.
\end{equation}
As one may verify, the Gibbs density 
\begin{equation}\label{e2.16}
    \rho_{\infty}(q,p) = \frac{1}{\mathbb{Z}}\exp{\Big(-\beta\big(|p|^{2}/2 + U(q)\big)\Big)}
\end{equation}
indeed satisfies the stationary Fokker-Plank equation (\ref{sfkp}) and the boundary condition (\ref{bcsfkp}). Here $\mathbb{Z}$ is the normalization constant so that $\int_{\mathbb{R}^{d}}\int_{ \bar{G}}\rho_{\infty}(q,p)dq dp =1$. Then the ergodic limit $\bar{\varphi}$ from (\ref{barvarphi}) associated with (\ref{ecld}) can be expressed as 
\begin{equation}\label{barvarphi2}
\bar{\varphi} = \int_{\mathbb{R}^{d}}\int_{G} \varphi(q,p) \rho_{\infty}(q,p) dq, dp. 
\end{equation}

In \cite[Section 9.1.5]{a22}, the stationary Fokker-Plank equation (\ref{sfkp})-(\ref{bcsfkp}) in the case of $G =(0,\infty) \times \mathbb{R}^{d-1}$ was considered.
In \cite{a13}, existence of the solution (in the weak sense) of a McKean-Vlasov Fokker-Plank equation with specular boundary condition was shown.

The above discussion of geometric ergodicity together with numerical methods for (\ref{ecld}) introduced in Section~\ref{cld_sec3} equip us with the tools needed to approximately sample from a target density $\rho(q)$, $q \in \bar{G}$, by choosing $U(q) = -\frac{1}{\beta}\ln (\rho(q))$, $\beta >0$.  

To prove error estimates for computing ergodic limits, we need the following assumption on solutions of the backward Kolmogorov equation with specular boundary condition (\ref{eq2.1})-(\ref{eq2.3}) related to (\ref{ecld}), i.e., of (\ref{eq2.1})-(\ref{eq2.3}) with $b(q,p) = -\nabla_q U(q) - \gamma p$.

\begin{assumption}\label{as:7}
The solution of (\ref{eq2.1})-(\ref{eq2.3}) with $b(q,p) = -\nabla_q U(q) - \gamma p$ satisfies Assumption~\ref{cld_as:4a} and the condition that 
there are positive constants $C$ and $\lambda$ independent of $T$ such that the following bound holds for some $m \ge 1$:
\begin{align}\label{eq:ergKolmest}
    \sum\limits_{l = 1}^{2}\sum\limits_{i + |j| = l }&|D_{t}^{i}D_{q}^{j}u(t,q,p)| +  \sum\limits_{l = 1}^{4}\sum\limits_{|j| = l}|D_{p}^{j}u(t,q,p)|\leq C(1 + |p|^{2m})e^{-\lambda (T-t)}, 
\end{align}
where $D_{x}^{j}u = \frac{\partial^{|j|}u}{\partial x^{j_{1}},\dots,\partial x^{j_{d}}} $, $x = q \in G \;\text{or}\; p \in \mathbb{R}^{d}$ and $j$ is a multi-index.
\end{assumption}

We note that an estimate of the form (\ref{eq:ergKolmest}) in the case of Langevin equations in the whole space is proved in \cite{TAL02}.
\medskip

Let us discuss the relationship between the confined Langevin dynamics (\ref{ecld}) and the reflected gradient SDE (Brownian dynamics with reflection or, in other words, the overdamped reflected Langevin dynamics) \cite{11,17,18,lst23}:
\begin{align}
 dX(t) &= -\nabla  U(X(t))dt + \sqrt{\frac{2}{\beta}} dW(t) + \nu (X(t))I_{\partial G}(X(t)) dL(t), \label{rgsde}
 \end{align}
 where $L(t)$ is the  local time of the process $X(t)$ on the boundary $ \partial G$.
 The solution $X(t)$ of (\ref{rgsde}) belongs to $\bar G$ and it is instantaneously reflected along the internal normal $\nu$ to $\partial G$ when it hits the boundary $\partial G$.
 Under the same conditions as we have imposed on the potential $U$ and the boundary $\partial G$ above, the SDE (\ref{rgsde}) is ergodic with respect to the invariant measure 
 $\rho_{{\rm RSDE}}(x) \propto \exp{(-\beta U(x)}) $ \cite{45}, which coincides with $\rho_{\infty}(q)=\int_{\mathbb{R}^{d}}\rho_{\infty}(q,p)dp$, i.e., both (\ref{ecld}) and (\ref{rgsde}) can be used for sampling 
 from the Gibbs distribution in the position space.
 Moreover, in analogy to the relationship between the Langevin equation and its overdamped limit in the whole $\mathbb{R}^{d}$ space \cite{Nel67} (see also \cite{40,leimkuhler_mathews_15}), the solution $Q(t)$ of (\ref{ecld}) tends in distribution to the solution $X(t)$ of (\ref{rgsde}) when $\gamma \rightarrow \infty $ \cite{costantini1991diffusion,a19, a27} (more precisely, in the corresponding small mass limit).
 In \cite{lst23} we proposed and studied numerical algorithms for sampling from $\rho_{{\rm RSDE}}(x)$, $x \in \bar G$, based on  (\ref{rgsde}) (see also \cite{sharma2022random}).
 The benefit of using kinetic Langevin equations as opposed to overdamped Langevin dynamics in the case of Euclidean space are well established and we expect these to carry over to confined spaces.


\section{Numerical methods}\label{cld_sec3}

In Section~\ref{sec:1st}, we present first-order weak methods for the SDEs (\ref{cld}).
In Section~\ref{sec:2nd}, we propose splitting schemes for the SDEs (\ref{ecld}) which demonstrate (in numerical studies) second-order weak convergence.  In Section~\ref{subsec_3.3}, we compare known numerical methods for (\ref{cld}) with the schemes of Sections~\ref{sec:1st} and~\ref{sec:2nd}, and we also discuss multiple collisions.
We start this section by describing the ingredients needed for approximating collision dynamics. 

Consider a uniform discretization of the time interval $[0,T]$ with time step $h = T/N$, i.e. $t_{n+1} - t_{n} = h$, $ n = 0,\dots, N-1$, and $t_0=0$. We take $Q_{0} \in G$ and $P_{0} \in \mathbb{R}^{d}\backslash \{0_{d}\}$.

Approximating $P(t)$, $Q(t)$ when the position $Q$ is inside the domain can be done by any suitable standard method of weak approximation of SDEs in the whole space (see e.g. \cite{milstein2021stochastic,leimkuhler_mathews_15}).  
However, to construct numerical methods for (\ref{cld}) and (\ref{ecld}), we also need to approximate elastic reflection. Hence the critical part of proposing such methods is to consider steps of a numerical scheme when collisions with $\partial G$ happen.

\textbf{One step with possible collision.}
Let the current state of a Markov chain approximating the solution of (\ref{cld}) be $(q, p) \in \bar{G} \times \mathbb{R}^{d}$. We construct the next state $\cal{Q}$, $\cal{P}$ of the chain as follows.

If $\hat{Q} := q +ps \in G$ for $s \in (0,h]$, 
\begin{align}\label{Ac_step_1}
A_{c}(q,p;h)= (q + ph, p),
\end{align}
else (i.e., if there is  $s \in (0,h]$: $q+ps \in \partial G$) $A_{c}(q,p;h)= ({\cal Q},{\cal P}),$ obtained from
\begin{align}
& Q_{c} = q + p\tau \in \partial G, \label{Ac_step_2}\\
     & {\cal P} = p - 2 (n(Q_{c})\cdot p) n(Q_{c}), \,\,
     {\cal Q} = Q_{c} + (h-\tau){\cal P}, \label{Ac_step_3}
\end{align}
where $\tau=\min_{s \in (0,h]} \{q+ps \in \partial G \}$. 
The random collision time $\tau$  ensures that the intermediate position $Q_{c} \in \partial G$ (note that $\big(p\cdot n_{G}(Q_{c})\big) > 0$). Thereafter, the numerical method follows the step ${\cal P} = p - 2 (n(Q_{c})\cdot p) n(Q_{c})$ 
which results in specular reflection (also known as elastic reflection) of the direction of momentum, and then completes the remaining step in position using the reversed momentum.  

As defined above, if there is more than one solution of   $q + p\tau \in \partial G$, we take the smallest positive value of $\tau$.  For illustration, consider an annular domain in $\mathbb{R}^{2}$: $G = \{ (q_{1}, q_{2})\; |\; r_{1}^{2} < q_{1}^{2} + q_{2}^{2} < r_{2}^{2}\} $ and $\partial G = \{ (q_{1}, q_{2})\; |\; q_{1}^{2} + q_{2}^{2} = r_{1}^{2}\} \; \cup \; \{ (q_{1}, q_{2})\; | \;q_{1}^{2} + q_{2}^{2} = r_{2}^{2}\}$, where $r_{1}$ and $r_{2}$ are positive constants such that $r_{1} < r_{2}$. Let $(q_{1}, q_{2}) = ((r_{1}+r_{2})/2, 0)$ and $ p = (-1, 0)$, then four possible solutions of $ q +\tau p \in \partial G$ are $\tau_{1} = (r_{2} - r_{1}) /2$, $\tau_{2} = (3r_{1} + r_{2})/2 $, $ \tau _{3} = (r_{1} + 3r_{2})/2$ and $\tau_{4} = (r_{1} - r_{2})/2$. Here, $\tau_{4}$ is negative, so we take $\tau = \min \{\tau_{1}, \tau_{2}, \tau_{3}\} = \tau_{1} = (r_{2} -r_{1})/2$. 

We note that we check $\hat{Q} := q  +ps\in G$ for $s \in (0,h]$ because the domain $G$ can be non-convex and we might get a situation when both $q$ and $q+ph $ are in $G$ but the line connecting them crosses $\partial G$ (i.e., it is partly in $G^c$). If $q+ p s \in \partial G$ has no solution or only one solution for $s \in (0, h]$ (e.g., this is the case when $G$ is convex), then the corresponding condition is simplified to $\hat{Q} := q  +ph\in G$. In future we use the indicator function $I_{G^c}(\hat Q)$ by which we will understand that the event   ``there is  $s \in (0,h]$: $q+ps \in G^{c}$'' occurred.

We consider the generalization of the A$_c$ step to the multiple collision setting in Subsection~\ref{sec:mult}.

\subsection{First-order schemes}\label{sec:1st}

In this section we propose two methods, [PA$_{c}$] and [A$_{c}$P]. In the [PA$_c$] scheme the Markov chain  first takes an auxiliary step in momentum $P$ and then updates position $Q$ via an A$_c$ step; the order of updates is reversed in [A$_{c}$P].   We introduce the following notation:
\begin{equation}
    \delta(q,p\; ; h,\xi) = hb(q,p) +  h^{1/2} \sigma \xi,
\end{equation}
where $\xi = (\xi^{1},\dots,\xi^{d})^\intercal $ is an $\mathbb{R}^{d}$-valued random variable such that $\xi^{i}$, $i=1,\dots,d$, are i.i.d. real-valued random variables with all moments bounded and satisfying 
\begin{equation} \label{eq:xicond}
    \mathbb{E}\xi^{i} = \mathbb{E}(\xi^{i})^{3}=0,\;\mathbb{E}(\xi^{i})^{2} = 1.
\end{equation}    
The conditions (\ref{eq:xicond}) are, for instance, satisfied if the $\{ \xi^i\}$ are drawn from the standard normal distribution or as i.i.d. discrete random variables with the distribution:
\begin{equation}
\mathbb{P}(\xi^i=\pm 1)=\frac{1}{2}. \label{eq:xi12}
\end{equation}

In Fig.~\ref{figure_cld_1} we present the proposed schemes. For the sake of clarity, we describe the [PA$_c$] scheme in detail. As already mentioned, the aim is to construct a Markov chain $(t_{k},Q_{k},P_{k})_{(k \geq 0)}$  to approximate  (\ref{cld}) in the weak sense.
At every iteration moving from $(t_{k},Q_{k},P_{k})$ to $(t_{k+1},Q_{k+1},P_{k+1})$, the Markov chain first takes auxiliary steps in momentum 
denoted by $\hat{P}_{k+1}$: 
\begin{align}
    \hat{P}_{k+1} = P_{k} + \delta(Q_{k}, P_{k}\; ; h, \xi_{k+1}),
    \label{cld_eq_3.2}
\end{align}
where $\xi_{k+1}\sim \xi $ and they are mutually independent. 
The next state of the Markov chain, $Q_{k+1}$ and $P_{k+1}$, is evaluated according to A$_c$($Q_{k}$, $\hat{P}_{k+1};h)$, which means that if $\hat{Q}_{k+1}= Q_{k} + s\hat{P}_{k+1} \in G$ for $s \in (0,h] \in G$, then $P_{k+1} = \hat{P}_{k+1}$ and $Q_{k+1} = Q_{k} + h\hat{P}_{k+1}$, else $P_{k+1} = \hat{P}_{k+1} - 2 (n_{G}(Q_{ k+1, c})\cdot \hat{P}_{k+1})n_{G}(Q_{ k+1, c})$ and $Q_{k+1} = Q_{ k+1, c} + (h - \tau)P_{k+1}$, where $ Q_{k+1, c} = Q_{k} + \hat{P}_{k+1}\tau \in \partial G $.  Here, for simplicity of the explanation, we assumed that we can have only one collision within one step, see how the algorithm is adjusted to the case of multiple collisions within one step in Section~\ref{sec:mult} and Algorithm~\ref{A$_c$ step with multiple}.
\begin{figure}[h]

\begin{minipage}[t]{.5\textwidth}

    \begin{flalign} 
   &\textrm{\underline{[PA$_c$]}} \nonumber \\
   & \hat{P}_{k+1} = P_k + \delta(Q_k,P_k;h,\xi_{k+1}) 
   \nonumber\\
   & (Q_{k+1},P_{k+1})=A_{c}(Q_k,\hat{P}_{k+1} ;h) \nonumber
   \end{flalign} 
\end{minipage}%
\begin{minipage}[t]{0.5\textwidth}
  \begin{flalign} 
   &\textrm{\underline{[A$_c$P]}} \nonumber \\
   & (Q_{k+1},\hat P_{k+1})=A_{c}(Q_k,P_{k} ;h) \nonumber \\
   &  P_{k+1} = \hat{P}_{k+1} + \delta(Q_{k+1}, \hat{P}_{k+1}; h ,\xi_{k+1}) \nonumber 
   \end{flalign}   
\end{minipage}%
   \caption{Description of [PA$_c$] and [A$_c$P] schemes.} \label{figure_cld_1}
\end{figure}

Schemes [PA$_c$] and [A$_c$P] have finite-time convergence with first weak order when they are applied to (\ref{cld}). We state below Theorem~\ref{theorem4.1} for scheme [PA$_c$] which is proved in Section~\ref{sec:proof1}. When these schemes are applied to the ergodic SDE (\ref{ecld}), they approximate the ergodic limit with error $\mathcal{O} (h+\exp(-\lambda T))$.  We state below the corresponding Theorem~\ref{theorem4.2} for scheme [PA$_c$] which is proved in Section~\ref{sec:proof2}. We recall that $h=T/N$. 

\begin{theorem}\label{theorem4.1}
 Under Assumptions~\ref{cld_as:1}-\ref{cld_as:4a}, the order of weak convergence of scheme [PA$_c$]  for (\ref{cld}) is $\mathcal{O}(h)$, i.e. the following inequality holds, for sufficiently small $h>0$:
\begin{equation}
    |\mathbb{E}\varphi(Q_{N},P_{N}) - \mathbb{E}\varphi\big(Q(T),P(T)\big)| \leq Ch,
\end{equation}
where 
$(Q(t),P(t))$ is from (\ref{cld}) and $C>0$ is independent of $h$.
\end{theorem}

\begin{theorem}\label{theorem4.2}
  Let Assumptions~\ref{cld_as:1}, \ref{cld_as:3}, \ref{as:5}-\ref{as:7} hold. For sufficiently small $h>0$ the following error estimate holds for scheme [PA$_c$] applied to (\ref{ecld}):
\begin{equation}
    |\mathbb{E}\varphi(Q_{N},P_{N}) - \bar{\varphi}| \leq C(h+\exp(-\lambda T)),
\end{equation}
where $C>0$ and $\lambda >0$ are constants independent of $T$ and $h$. 
\end{theorem}


\subsection{\texorpdfstring{[A$_c$, B, O] splitting schemes}{[A_c, B, O] splitting schemes}}
\label{sec:2nd}

Without a reflective boundary, 
splitting schemes of second weak order for the Langevin equations, which rely on a concatenation of the exact integration for the drift and impulse and of the Ornstein-Uhlenbeck process at each time step, are known for their excellent performance in molecular dynamics and statistics applications (see, e.g. \cite{GNT03,nawaf_houman_10,BenM13,leimkuhler_mathews_15,milstein2021stochastic} and references therein). Hence, it is natural to extend those methods to confined Langevin dynamics (\ref{ecld}). The [A$_c$, B, O] splitting schemes proposed in this subsection degenerate to the standard deterministic St\"{o}rmer-Verlet schemes with collision step A$_c$ when (\ref{ecld}) degenerates to the (deterministic) Hamiltonian system with elastic collisions. Such deterministic schemes have been shown to be first order accurate due to the collision step A$_c$ having local error of ${\cal O}(h)$ \cite{BrianBen2000,leimkuhler_mathews_15} (see also discussion in Section~\ref{prlim_dis_sec}).
Because of this, stochastic numerics experience \cite{milstein2021stochastic} suggests that finite-time weak convergence of the [A$_c$, B, O]  schemes for (\ref{ecld}) should be first order, too. However, the symmetric [A$_c$, B, O]  schemes for (\ref{ecld}) also have weak order $2$, as we demonstrate in several carefully constructed and performed numerical experiments (Section~\ref{sec:tests}) and theoretically justify when $G$ is a half-space, i.e. $G=(0,\infty)\times \mathbb{R}^{d-1}$, in Section~\ref{sec:2ndorder}. 

Let $\xi = (\xi^{1},\dots,\xi^{d})^\intercal $ be an $\mathbb{R}^{d}$ valued random variable  such that $\xi^{i}$, $i=1,\dots,d$, are i.i.d. real-valued random variables with bounded all moments and satisfying 
\begin{equation} \label{eq:xicond2}
    \mathbb{E}\xi^{i} = \mathbb{E}(\xi^{i})^{3}=\mathbb{E}(\xi^{i})^{5}=0, \;\;\mathbb{E}(\xi^{i})^{2} = 1, \;\; \mathbb{E}(\xi^{i})^{4} = 3.
\end{equation}    
The conditions (\ref{eq:xicond}) are, for instance, satisfied for $\xi^i$ having the standard normal distribution or being i.i.d. discrete random variables with the distribution:
\begin{equation}
P(\xi^i=\pm \sqrt{3})=\frac{1}{6}, \,\, P(\xi^i=0)=\frac{2}{3}. \label{eq:xi23}
\end{equation}
In Fig.~\ref{fig:2ndorder}, $\xi_{k}\sim \xi $ and $\zeta_{k}\sim \xi $  and they all are mutually independent. 
Introduce the notation 
\[
O(p;h,\xi,\gamma):=p e^{- \gamma h} +\sqrt{\frac{2\gamma}{\beta }(1-\mathrm{e}^{-2\gamma h})}\xi .
\]
The three proposed methods are presented in Fig.~\ref{fig:2ndorder}. It is straightforward to write down the other three possible symmetric concatenations, namely [BOA$_c$OB], [A$_c$BOBA$_c$], [A$_c$OBOA$_c$], which we test in Experiment~\ref{exper2}.

\begin{figure}[h]
\centering
   \begin{minipage}{.45\textwidth}
     \begin{flalign} 
   &\textrm{\underline{[OBA$_c$BO]}} \nonumber \\
   & P_{k+1/2} = O(P_k;h/2,\xi_{k+1},\gamma)  \nonumber\\
   & \hat{P}_{k+1/2} =     P_{k+1/2} - \frac{h}{2} \nabla_q U(Q_k) \nonumber\\
    &  (Q_{k+1}, P_{k+1/2}^{r}) =  A_c(Q_k,\hat{P}_{k+1/2};h) \;\;\; \nonumber \\
    & \hat{P}_{k+1} = \hat{P}_{k+1/2}^{r} - \frac{h}{2} \nabla_q U(Q_{k+1}) \nonumber \\ 
    & P_{k+1} = O(\hat{P}_{k+1};h/2,\zeta_{k+1},\gamma) \nonumber
   \end{flalign} 
\end{minipage}%
\begin{minipage}{0.45\textwidth}
  \begin{flalign} 
   &\textrm{\underline{[BA$_c$OA$_c$B]}} \nonumber \\
   & P_{k+1/2} =  P_k - \frac{h}{2} \nabla_q U(Q_k)\nonumber\\
    & (Q_{k+1/2},P_{k+1/2}^{r}) =  A_c(Q_k,P_{k+1/2};h/2)   \;\;\; \nonumber \\
      & \hat{P}_{k+1/2} =O(P_{k+1/2}^{r};h,\xi_{k+1},\gamma)   \nonumber\\
      &  (Q_{k+1},P_{k+1}^{r}) =  A_c(Q_{k+1/2},\hat{P}_{k+1/2};h/2)  \nonumber \\
    &P_{k+1} =  P_{k+1}^{r} - \frac{h}{2} \nabla_q U(Q_{k+1}) \nonumber
   \end{flalign} 
\end{minipage}%
\\
\centering
\begin{minipage}{0.3\textwidth}
 \begin{flalign} 
   &\textrm{\underline{[OA$_c$BA$_c$O]}} \nonumber \\
   & P_{k+1/2} = O(P_k;h/2,\xi_{k+1},\gamma) \nonumber\\
   & (Q_{k+1/2}, P_{k+1/2}^{r}) =  A_c(Q_k,P_{k+1/2};h/2) \nonumber \\
   & \hat{P}_{k+1/2} =    \nonumber P_{k+1/2}^{r} - h \nabla_q U(Q_{k+1/2})\\
   &  (Q_{k+1}, P_{k +1}^{r}) =  A_c(Q_{k+1/2},\hat{P}_{k+1/2};h/2) \nonumber \\ 
    & P_{k+1} = O(P_{k+1}^{r};h/2,\zeta_{k+1},\gamma) \nonumber
   \end{flalign} 
    \end{minipage}
    
     \caption{Description of [OBA$_c$BO], [BA$_c$OA$_c$B], [OA$_c$BA$_c$O] schemes for (\ref{ecld}).}
\label{fig:2ndorder}
\end{figure}

\subsection{Comparison with other numerical integrators for (\ref{cld})}\label{subsec_3.3}

Here we present a comparison of previously known schemes for (\ref{cld}) with the numerical integrators proposed in the previous two subsections.

For $G = (0, \infty)$,
in \cite{a10} (see also \cite{a9, a27}), the authors studied a symmetrized scheme in which, given $(q,p)$ at time $t_{0}$, the approximate position is always updated as 
\begin{align}
\bar Q(t_{0}+h) = |q + ph|. \label{3.9eq}
\end{align}
At every step,  $\theta = t_{0} - \frac{q}{p}$ is calculated. A collision happens if $\theta \in (t_0, t_0+h)$. In that case  
\begin{align}
   \bar Q(\theta) = 0, \,\,\,  \bar {P}(\theta -) = p + b(q,p) (\theta - t_0) + \sigma (W(\theta) - W(t_{0})), \,\,\,
     \bar {P}(\theta) = - \bar {P}(\theta -), \label{eq:bossymeth}\\
     \bar {P}(t+h) = p + b(\bar Q(\theta),\bar P(\theta)) (t_0+h - \theta) + \sigma (W(t_0+h) - W(\theta)) . \notag
\end{align}
If no collision occurs then $ \bar {P}(t_0+h) = p + b(q,p) h + \sigma (W(t_0+h) - W(t_0)) $ (i.e., the straightforward Euler update).
This scheme was generalized to half-space in higher dimensions, i.e. $ G = (0,\infty)\times \mathbb{R}^{d-1}$, in \cite{a9, a10,a27}. 

The main difference between the above scheme (\ref{eq:bossymeth}) and our schemes lies in the update of momentum. In the scheme above, it is possible for the momentum to get reversed already in the interval $(t_0, \theta)$ due to a realized large Wiener increment. In such a case the update $\bar {P}(\theta) = - \bar {P}(\theta-)$ results in a direction towards the negative axis, which leads to multiple collisions during the time interval $ (t_0, t_0+h)$. However, this pathological situation does not arise, by construction, in the schemes [A$_c$P], [PA$_c$] and the splitting schemes of Section~\ref{sec:2nd} if implemented for $G = (0,\infty)\times \mathbb{R}^{d-1}$. We further observe that if one extends the symmetrized reflection of (\ref{3.9eq}) around the boundary in curved domains, then the symmetrized reflection around the boundary along the normal does not imitate specular reflection of  (\ref{cld}), unlike the schemes of Sections~\ref{sec:1st} and \ref{sec:2nd}. 
We note that the proof technique in \cite{a10} for weak convergence of (\ref{eq:bossymeth}) relies on being able to write a continuous-time version of the approximation that solves an SDE. This approach is only possible for a narrow class of numerical methods and our geometric integrators do not belong to that class. 

 Further, some penalized schemes for (\ref{cld}) are also proposed in \cite{a27} without analysis. It is known (see \cite{2,milstein2021stochastic, lst23} and references therein) that symmetrized reflection and penalized schemes are good candidates to  approximate the reflected SDEs (\ref{rgsde}),  in which the local time $L(t)$ is responsible for keeping the trajectories $X(t)$ in $\bar G$. In contrast, the confinement of position $Q(t)$ in $ \bar G$ is achieved in (\ref{cld}) by the c\'{a}dl\'{a}g  process $\mathcal{R}(t)$ whose explicit expression is known via values of $(P(t),Q(t))$. 
 In comparison with  the existing schemes, all our schemes are constructed utilizing the expression for  $\mathcal{R}(t)$, which, in particular, ensures  conservation of (conjugate) momenta in the event of collision, i.e.,  
\begin{equation}\label{eq:PhatP}   
|P_{k}|^{2} = |\hat{P}_{k}|^{2}I_{\bar{G}}({\hat Q}_{k}) + |\hat{P}_{k} - 2(n^{\pi}\cdot \hat{P}_{k})n^{\pi}|^{2}I_{\bar{G}^{c}}({\hat Q}_{k})= |\hat{P}_{k}|^{2}, 
\end{equation}
where $n^{\pi} = n_{G}({Q}_{c, k})$. 
 Since the collision is interlaced in the one-step procedure of the discrete dynamics,  estimates related to the average number of collisions become crucial for the global error analysis (see Lemmas~\ref{bl} and \ref{ebl}).  

Although we present our results in a bounded domain $G$ with boundary $\partial G$, the results can be extended to the half-space $G = (0,\infty)\times \mathbb{R}^{d-1}$ or any other unbounded domain with boundary satisfying Assumption~\ref{cld_as:1}. The only additional requirement is moment bounds for $Q_k$ (moments of $Q(t)$ are typically bounded for most forces $b$ and potentials $U$ of interest), which can be obtained under a global Lipschitz assumption on $\nabla_q U$ or replaced by using the concept of rejecting trajectories \cite{a23,milstein2021stochastic} in the non-globally Lipschitz case. 

\subsection{Multiple collisions}\label{sec:mult}

 In our schemes,  multi-collisions in a single time-step can occur for a bounded domain $G$ when $|p|$ is large and/or $(n(Q_{c}) \cdot p)$ (see A$_c$ step) is small, i.e. when with sufficiently large speed, the particle collides with the curved wall in the (approximately) grazing manner (see Figure~\ref{multi_collis_figure}).    In this case $Q_{1}$ calculated in A$_c$ will not lie in $G$. We note that when $\partial G$ is a convex polytope, we can always choose $h>0$  and initial speed $|p|$ sufficiently small such that multi-collisions occur with exponentially low probability as in this case multi-collisions can only happen when $|p|$ is very large. There is no possibility in our schemes for multi-collisions in half-space (unlike \eqref{3.9eq}-\eqref{eq:bossymeth}), but when $\partial G$ is curved, geometry comes into the picture and it may potentially result in multi-collisions in the discrete dynamics. Obviously, our schemes can be adapted to this situation by applying the A$_c$ step recursively.  We define this recursion explicitly.  
 
 Let us denote the current state of a Markov chain approximating the solution of (\ref{cld}) by $(q, p) \in \bar{G} \times \mathbb{R}^{d}$. Given $(q,p)$, let us denote the number of collisions that may happen in time interval $h$ by $\kappa := \kappa (q,p,h)$.  The detailed construction of the next state  ($Q_{1}$, $P_{1}$) of the chain is in Algorithm~\ref{A$_c$ step with multiple} which is multi-collision generalization of (\ref{Ac_step_1})-\eqref{Ac_step_3}.

 \begin{figure}[H]
     \centering 
     \begin{tikzpicture}[
    scale=1.1,
    point/.style={circle, fill=cyan!50!blue!60, inner sep=2pt},
    guide/.style={densely dotted, thick, darkgray},
    lbl/.style={text=black, font=\Large},       
]

    \coordinate (q) at (0,0);
    \coordinate (Qc1) at (2.5, 1.8);
    \coordinate (Qc2) at (5.5, 2.2);
    \coordinate (Qc3) at (7.0, 0.8);
    \coordinate (Qc4) at (6.6, -0.8);
    \coordinate (Q1) at (5.2, -2.2); 

    \coordinate (Qh1) at ($(q)!2.8!(Qc1)$);
    \coordinate (Qh2) at ($(Qc1)!2.3!(Qc2)$);
    \coordinate (Qh3) at ($(Qc2)!2.2!(Qc3)$);
    \coordinate (Qh4) at ($(Qc3)!2.0!(Qc4)$);

    \draw[cyan!50!blue!60, line width=2.5pt] 
        (0.5, 0.8) 
        .. controls (1.5, 1.5) and (2.0, 1.7) .. (Qc1)
        .. controls (3.5, 2.1) and (4.5, 2.3) .. (Qc2)
        .. controls (6.5, 2.0) and (6.8, 1.5) .. (Qc3)
        .. controls (7.1, 0.0) and (6.8, -0.4) .. (Qc4)
        .. controls (6.4, -1.5) and (5.5, -2.5) .. (4.5, -3.5);

    \node[lbl, anchor=north east] at (4.5, -3.5) {$\partial G$};

    
    \draw[guide] (q) -- (Qh1);
    \node[lbl, anchor=north west, xshift=2pt] at ($(q)!0.5!(Qc1)$) {$\tau_1$};

    \draw[guide] (Qc1) -- (Qh2);
    \node[lbl, anchor=north, yshift=-8pt] at ($(Qc1)!0.5!(Qc2)$) {$\tau_2$};

    \draw[guide] (Qc2) -- (Qh3);
    \node[lbl, anchor=north east, xshift=-2pt, yshift=-2pt] at ($(Qc2)!0.5!(Qc3)$) {$\tau_3$};

    \draw[guide] (Qc3) -- (Qh4);
    \node[lbl, anchor=east, xshift=-5pt] at ($(Qc3)!0.5!(Qc4)$) {$\tau_4$};

    \draw[guide] (Qc4) -- (Q1);
    \node[lbl, anchor=east, xshift=-4pt] at ($(Qc4)!0.45!(Q1)$) {$h-\vartheta_4$};

    \foreach \p in {q, Qc1, Qc2, Qc3, Qc4, Q1, Qh1, Qh2, Qh3, Qh4}
        \node[point] at (\p) {};

    
    \node[lbl, below, yshift=-5pt] at (q) {$q$};
    
    \node[lbl, above left] at (Qc1) {$Q_{c,1}$};
    \node[lbl, above, yshift=3pt] at (Qc2) {$Q_{c,2}$};
    \node[lbl, right, xshift=5pt] at (Qc3) {$Q_{c,3}$};
    \node[lbl, right, xshift=5pt] at (Qc4) {$Q_{c,4}$};
    
    \node[lbl, right, xshift=3pt] at (Qh1) {$\widehat{Q}_1$};
    \node[lbl, right, xshift=3pt] at (Qh2) {$\widehat{Q}_2$};
    \node[lbl, right, xshift=3pt] at (Qh3) {$\widehat{Q}_3$};
    \node[lbl, below, yshift=-5pt] at (Qh4) {$\widehat{Q}_4$};
    
    \node[lbl, below, xshift=-12pt,yshift=+2pt] at (Q1) {$Q_1$};

    \node[lbl] at (2.5, -1.0) {$G$};

\end{tikzpicture}
\caption{A depiction of multi-collisional step of Algorithm~\ref{A$_c$ step with multiple}.}
   \label{multi_collis_figure}
 \end{figure}

\begin{algorithm}\caption{A$_c$ step with multiple collisions}\label{A$_c$ step with multiple}
\small  
\setlength{\baselineskip}{0.9\baselineskip}
\begin{algorithmic}
\Require $q$, $p$, $h$
\Ensure $q \in G$, $ t = t_0$
\If{$q + ps \in G$ for all $s \in (0,h]$}
\State ${\cal Q} \gets q+ ph $ and ${\cal P} = p$
\Else{} 
\State $ i \gets 0 $, $\vartheta_0 \gets 0$
\State $\tilde{Q}_i \gets q$, $\tilde{P}_i \gets  p$
\While{$\hat{Q}_i := \tilde{Q}_i + s \tilde{P}_i \in {G}^c$ for some $s \in (0, h- \vartheta_i]$}
\State Find minimum $\tau_{i+1} $ such that $Q_{c, i+1} \gets \tilde{Q}_i + \tilde{P}_i \tau_{i +1} \in \partial G $
\State $\vartheta_{i+1} \gets \vartheta_i + \tau_{i +1} $
\State Reflect the momentum i.e. $ \tilde{P}_{i+1} \gets \tilde{P}_i -2 (n(Q_{c, i+1}) \cdot \tilde{P}_i) n(Q_{c, i+1})$
\State $\tilde{Q}_{i+1} \gets Q_{c, i+1}$
\State $i \gets i+1$.
\EndWhile
\State ${\cal Q} \gets  Q_{c, i} + (  h -\vartheta_{i})\tilde{P}_{i}$ and ${\cal P} \gets \tilde{P}_i $
\EndIf
\end{algorithmic}
\end{algorithm}

We can write the procedure presented in Algorithm~\ref{A$_c$ step with multiple} of obtaining $Q_1$ in the following compact form:
\begin{align}
    {\cal Q} = q + \tau_1 \tilde{P}_0 +  \tau_2 \tilde{P}_1 + \dots + \tau_{\kappa} \tilde{P}_{\kappa-1}  + (h -\vartheta_{\kappa})\tilde{P}_{\kappa},
\end{align}
where $\tilde{P}_0 = p$, $\tilde{P}_i$ are computed recursively  as follows:
\begin{align}
    \tilde{P}_{i} &= \tilde{P}_{i-1} -2 (n(Q_{c,i}) \cdot \tilde{P}_{i-1})n(Q_{c, i}), \quad i=1,\dots, \kappa, \label{cld_new_eqn_3.11}
\end{align}
and 
\begin{align}
    \vartheta_{\kappa} &= \tau_1 + \dots + \tau_{\kappa}.
\end{align}
Therefore, we have
\begin{align*}
    {\cal Q} = q + \tau_1 \tilde{P}_0 +  \tau_2 \tilde{P}_1 + \dots + \tau_{\kappa -1} \tilde{P}_{\kappa-2} +     
    (h - \vartheta_{\kappa-1})\tilde{P}_{\kappa -1} - 2(h -\vartheta_{\kappa})(n(Q_{c,\kappa}) 
    \cdot \tilde{P}_{\kappa-1})n(Q_{c,\kappa}),
\end{align*}
where we have used $\tau_{\kappa} \tilde{P}_{k-1} + (h -\vartheta_{k}) \tilde{P}_{k-1} = (h-\vartheta_{k-1}) \tilde{P}_{k-1} $. 
Doing the same computation $\kappa-1$ times, we obtain
\begin{align}
    {\cal Q} = q + h \tilde{P}_0 -2 \sum_{i=1}^{\kappa} (h -\vartheta_{i})(n(Q_{c,i}) \cdot \tilde{P}_{i-1})n(Q_{c,i}),
\end{align}
which implies
\begin{align}\label{cld_new_eq_3.16}
    |{\cal Q} - q| \leq h |p| + 2 \sum_{i=1}^{\kappa} (h -\vartheta_{i})(n(Q_{c,i}) \cdot \tilde{P}_{i-1}).
\end{align}


 The question that remains to be answered is  how many times we need to utilize this recursion.  Let us again consider the one-step procedure of [PA$_c$] scheme: depending on the update of the momentum P,  A$_c$ is deterministic which means the dynamics becomes billiard-like, i.e., straight line motion between collision points. 
 
 For the sake of the discussion here, let us denote by $(\tilde{q}, \tilde{p})$ the position and new momentum, respectively, after the P step in [PA$_c$].  Also, denote the continuous-time dynamics evolving according to the A$_c$  step by $(\tilde{Q}(t),  \tilde{P}(t))$. As already highlighted, the dynamics of $\tilde{Q}(t)$ follows a straight path, and indeed is the exact solution of the following ODE system:
 \begin{align*}
     d\tilde{Q}(t) &= \tilde{P}(t) dt, \;\;\;\; \tilde{Q}(t_{0}) = \tilde{q},\\
     d \tilde{P}(t) &= d \mathcal{R}(t),  \;\;\;\; \tilde{P}(t_{0}) = \tilde{p},
 \end{align*}
 in the interval $(t_0, t_0 + h)$. It is clear that the dynamics in the duration of $h$ may have multiple collisions. Using measure-theoretic arguments, it is proved in \cite[p.~5]{pulvirenti_elastic_collision_}, that the Lebesgue measure of a set of configurations $(\tilde{q}, \tilde{p})$ for which one encounters infinite collisions in finite time (in the above case in time  from $t_0$ to $ t_0 + h $) and for which the dynamics hit the boundary with zero normal component of velocity is zero. Also, see \cite[Theorem~2.1 and Theorem~2.2]{costantini1991diffusion} where the arguments of \cite{pulvirenti_elastic_collision_} are used to show well-posedness of a physical transport Markov process with jumps with specular reflection boundary condition. Obtaining an estimate on the number of collisions of billiard dynamics under different settings (e.g., curvature, dispersion on the boundary, dimension, initial conditions, final time)  has been a major topic of research in dynamical systems (see \cite{tabachnikov2010arbeitsgemeinschaft}). We mention famous results from the dynamical systems literature in  regard to the number of collisions: \cite{sinai1978billiard} proves that a billiard particle escapes from a convex polyhedral angle and obtains an upper bound  on number of collisions, \cite{burago1998uniform} obtains a uniform estimate on number of collisions in a domain with semi-dispersing boundaries, an interested reader may also explore the works \cite{boldrighini1978billiards, tabachnikov2005geometry}.  
  
 With the discussion above, in our stochastic setting, it is expected that the number of trajectories corresponding to the proposed schemes which experience multiple collisions in a single step will be negligible in the  small $h$ limit assuming a light-tailed initial distribution of $p$ (for example a Gaussian distribution) and, indeed, this has been confirmed in our experiments (see Section~\ref{sec:tests}). For simplicity, we rejected trajectories for which more than one collision occurs in a step of the Markov chain.  In experiments for the ergodic case (i.e., during long time simulations) on a circular domain of radius two in two dimensions, we encountered multiple collisions in less than $10$ out of $10^6$ simulated trajectories for $ 0.2 < h < 0.5 $, and we observed no steps with multiple collisions for $h<0.2$, i.e. the effect of multiple collisions in our schemes is seen to be  negligible, as expected. Similar results were seen in the case where we chose $\nabla_q U$ to deliberately push the dynamics towards the boundary to enhance the number of collisions (see Experiment~\ref{exper4}). In practice, one can also fix a constant $L$ and truncate the recursion of Algorithm~\ref{A$_c$ step with multiple} when the number of collisions in a single step crosses $L$. 

However, when it comes to error analysis in the next section, we need certain bounds.  For our numerical error analysis purposes,  it is sufficient to have the estimate \eqref{cld_new_eq_3.16}, the measure theoretic result of \cite{pulvirenti_elastic_collision_}, and the following bound:
\begin{align}
    \sum_{i=0}^{\kappa -1}(n(Q_{c,i+1})\cdot \tilde{P}_{i}) \leq C (1+|p|^2),
\end{align}
where $C >0$ does not depend on $q , p$ and $h$. We employ Lyapunov-type arguments to obtain the above bound. 
We relabel the dynamics $(\vartheta_{i},Q_{c, i},\tilde{P}_{i-1})$ evolving on $(0, h] \times \partial G \times \mathbb{R}^d$ corresponding to $A_c$ as
$(\vartheta_{i},X_{i},Y_{i})$  with 
\begin{align}
X_{i} = Q_{c, i}, \quad Y_{i} = \tilde{P}_{i-1}, \quad \quad i = 1,\dots, \kappa , 
\end{align}
with $\kappa$ being defined as (the same as before): let $\mathcal{S} : =   \{(t,x,v) \in (0,h]\times \partial G\times \mathbb{R}^d\; ;\;x + (h-t)(y-2(n(x) \cdot y)n(x)) \in \bar{G}\}$, then 
\begin{align}
    \kappa := \inf\{ k \geq 1 \; : \; (\vartheta_k, X_k, Y_k) \in \mathcal{S} \}.
\end{align} 
We introduce this new notation to simplify the proof which follows and also the result of this subsection is self-contained and is of interest independently of $A_c$ step or collisional Langevin dynamics, and hence it is useful to formulate it in general terms.

Let $(x,y)$ be so that  $x + sy \in  G^c$ for some $s \in (0,h]$.
Consider the discrete-time dynamical system $(\vartheta_{i}, X_i, Y_i)_{i=0}^{\kappa}$   
given by
\begin{align}
\vartheta_{0}=0&, \,\, X_0=x, \,\, Y_0=y, \label{cld_eqn3_23}  \\  
X_{i+1} &= X_{i} + {\tau}_{i+1} Y_{i}, \notag     \\ 
    Y_{i+1} &= Y_{i} - 2(n(X_i) \cdot Y_{i}) n(X_i),  \notag\\  
        \vartheta_{i+1} & = \vartheta_{i} + {\tau}_{i+1},  \quad i= 0,\dots, \kappa-1, \notag 
\end{align}
where ${\tau}_{i+1}$ is the smallest time increment so that $X_{i+1} \in \partial G$. 
Note that  $(\vartheta_{i}, X_i, Y_i)_{i=1}^{\kappa}$ evolves on $(0, h] \times \partial G \times \mathbb{R}^d$. Recall that  the Lebesgue measure of initial configurations  $(x, y) \in G \times \mathbb{R}^d$ is zero for which  the number of collisions $\kappa$ can be infinite as explained above based on \cite{pulvirenti_elastic_collision_}.

    In the newly introduced notation, we aim to bound
\begin{align}
    \sum_{i=1}^{\kappa } (n(X_{i}) \cdot Y_{i})  \leq C  (1+|y|^2),
\end{align}
where, as mentioned, $C>0$ does not depend on $h$ and $ (x,y)$, it does depend on $\bar{G}$.

Consider the Koopman-type operator,  acting on a suitable class of functions,  associated with the discrete dynamical system 
\((\vartheta_i, X_i, V_i)\) (evolving according to \eqref{cld_eqn3_23}):
\begin{equation}
    \mathcal{K}V(\vartheta_k, X_k, Y_k) = V(\vartheta_{k+1}, X_{k+1}, Y_{k+1}).
\end{equation}
Introduce the boundary value problem corresponding to the one-step transition of the dynamical system \eqref{cld_eqn3_23}:
\begin{align}
    \mathcal{K}V(t,x,y) - V(t,x,y) &= - g(t,x,y), \quad (t,x,y) \in [0,h)\times \partial G \times \mathbb{R}^d, \\
    V(t,x,y) &= 0, \quad (t,x,y) \in \mathcal{S} \text{ or } x \in  G, \label{eq:Koop}
\end{align}
where 
$
    \mathcal{K}V(t,x,y) = V(\vartheta_1, X_1, Y_1) 
$, $g \geq  0$, and $g(t,x,y)=0$ for $(t,x,y) \in \mathcal{S}$ and $x \in G$. 
As can be verified (cf. \cite{20, 46,milstein2021stochastic}), the solution of the above boundary value problem is given by 
\begin{align}
    V(0, x, y) = \sum_{i=1}^{\kappa -1 } g(\vartheta_k, X_k, Y_k).
\end{align}
If we take $g(t, x, y) = (n(x)\cdot y)$, 
the solution becomes
\begin{align}
    V(0, x, y) = \sum_{i=1}^{\kappa -1 }(n(X_i) \cdot Y_i). 
\end{align}
If we can find a function $v$ with corresponding $g$ such that for $(t,x,y) \in [0,h)\times \partial G \times \mathbb{R}^d$:
\begin{align}
    g(t,x,y) \geq (n(x) \cdot y) 
\end{align}
then $ V(t,x,y) \leq v(t,x,y) $. The proof of  next proposition rests on finding such a Lyapunov function.

\begin{proposition}\label{prop:mult}
    Let $\partial G \in C^{2,\alpha}$  and $G$ be bounded.  Then
    \begin{align}
        \sum_{i=1}^{\kappa } (n(X_i) \cdot Y_i)  \leq C (1+h)|y|(1+|y|),
    \end{align}
where $C>0$ does not depend on $(x,y)$ and $h$, and it depends on $\bar{G}$. 
\end{proposition}
\begin{proof}
Consider a boundary zone $G_{R_{0}} \subset G$ defined as $G_{R_{0}} := \{ x \in \bar{G}\; ; \; \dist(x,\partial G) < R_{0}\}$. Due to \cite[Appendix]{104},  there is an $R_{0}$ such that the projection of $x$ on $\partial G$ is unique whenever $x \in G_{R_{0}}$. Introduce the function
\begin{equation*}
    V(t,x,y) =  \begin{cases}
    0,& (t, x, y) \in \mathcal{S} \text{ or } x \in G,\\
    (A + K(h-t))(| y|^{2} +|y|) +  (B(x)\cdot y), & (t, x, y) \in (0,h)\times \partial {G}\times \mathbb{R}^{d},
    \end{cases} 
\end{equation*}
where $B(x) \in C^{1}(\bar{G})$ is  an $\mathbb{R}^{d}$ valued function such that  $B(x) : = n_{G}(x^{\pi})$ when $x \in \bar{G}_{R_{0}}$; here $x^{\pi}$ is the projection of $x$ on $\partial G$. Note that $n_{G}(x) \in C^{1}(\partial {G})$ due to Assumption~\ref{cld_as:1} (see \cite[Appendix]{104}). We extend the function $B(x)$ from $G_{R_{0}}$  to $\bar{G}$ so that $B(x) \in C^{1}(\bar{G})$ and $|B(x)|_{G}^{1,\alpha} \leq C |n(x^{\pi})|_{G_{R_{0}}}^{1,\alpha}$, where $|\cdot|^{1,\alpha}$ is the H\"older norm and $C>0$ depends on $G$. Such an extension is possible due to \cite[Lemma 6.37]{104}. 
We choose the constant $A > 0$ so that  $V(t,x,y)$ is always positive. Note that
$B(x)$ is a bounded function with bounded first derivatives, hence it is possible to make such a choice of $A$. The choice of the constant $K > 0$ will be discussed later in the proof.

Given $(\vartheta_{k},X_{k}, Y_{k}) = (t,x,y)$, $k < \kappa -1$, we now calculate $-(\mathcal{K}V(t,x,y) - V(t,x,y))$. 
Using Taylor's theorem, we obtain 
\begin{align}
    B(X_{1}) &=  B(x) +  \tau_1
\textbf{J}_{B}(x+ \epsilon \tau_1 y)y , 
 \end{align}
where $\epsilon \in (0,1) $ and $\mathbf{J}_{B}(y)$ is the Jacobian matrix of $B$ evaluated at $y$. Therefore,
\begin{align}
   (B(X_{1}) \cdot Y_1) = \big(B(X_1) \cdot (y - 2 (n(x) \cdot y)n(x))\big) & = (B(x) \cdot y) -2 (n(x) \cdot y) (B(x) \cdot n(x))   \nonumber   \\  &      + {\tau}_1 \big(\textbf{J}_{B}(x+ \epsilon \tau_1 y)y \cdot Y_1\big).
\end{align}
We also have
\begin{align}
    (A + K(h-\vartheta_1))(|Y_1|^2 +|Y_1|)=  (A + K(h-t-{\tau}_1 ))( |y|^2 +|y|),  
\end{align}
since $|Y_1| = |y|$ (conservation of momentum). 

Hence, using $(B(x) \cdot n(x)) =1$, we obtain
\begin{align}
    (A + K(h-\vartheta_1))(|Y_1|^2 + |Y_1|)+(B(X_1) \cdot Y_1)   &=   (A + K(h-t))(|y|^2+|y|) 
    - K{\tau}_1 (|y|^2+|y|)   \nonumber  \\ 
    &+(B(x) \cdot y) -2 (n(x) \cdot y)  
    +{\tau}_1\big(\textbf{J}_{B}(x+ \epsilon {\tau_1} y)y \cdot Y_1\big).
\end{align}
Due our assumption on $G$ and $\partial G$, $J_{\max}:=$$\max_{y \in \bar{G}}\| \textbf{J}_{B}(y)\|$
is bounded. Therefore, 
\begin{align}
    (A + K(h-\vartheta_1))(|Y_1|^2+|Y_1|)+(B(X_1) \cdot Y_1)   &\leq  (A + K(h-t))(|y|^2+|y|)- K\tau_1 (|y|^2+|y|)
       \nonumber  \\  & +(B(x) \cdot y) -2 (n(x) \cdot y)    +\tau_1 J_{\max}|y|^2.
\end{align}
With an appropriate choice of $K $ (i.e.  $K > J_{\max}$), we have constructed a function $g := - (\mathcal{K}V- V)$ such that $-(\mathcal{K}V- V) \geq (n(x) \cdot v) $, the proposition in proved. 
\end{proof}



\begin{remark}
Let $x \in \partial G$ and $\phi =\phi(x,y):=\arcsin{(n(x) \cdot y)/|y|}$, i.e, the angle between $y$ and the tangent to $\partial G$ at $x$. Assume that the angle is bounded from below, $\phi(x,y) \ge \phi_{\min}$. Then, it follows from Proposition~\ref{prop:mult} that
\[
\kappa \leq C (1+h) \frac{1+|y|}{\sin{\phi_{\min}}},
\]
where $C>0$ does not depend on the velocity $y$, position $x$ and time $h$, it depends on the size of the domain $G$ and curvature of its boundary $\partial G$. 
Note that $G$ can be a non-convex domain with $C^{2,\alpha}$ boundary. From the point of view of dynamical systems (in that case time $h$ is usually considered to be large
in contrast to our numerical analysis work where $h>0$ is sufficiently small), this is an interesting estimate for the number of collisions away from the grazing set.

\end{remark}

\section{Proofs}\label{sec:proofs}

In this section, for the first-order scheme, we prove the finite-time convergence Theorem~\ref{theorem4.1} (Section~\ref{sec:proof1}) and Theorem~\ref{theorem4.2} (Section~\ref{sec:proof2}) on the error estimate for computing ergodic limits. We also provide a justification (Section~\ref{sec:2ndorder}) for second-order weak convergence of the splitting schemes from Section~\ref{sec:2nd}.

\subsection{Proof of Theorem~\ref{theorem4.1}}\label{sec:proof1}

The proof is based on  two main ingredients: a one-step error estimate (Lemma~\ref{lemma4.4}) and an estimate related to the average number of collisions (Lemma~\ref{bl}). 
According to the lemmas, inside the domain, where we use the usual Euler-type walker, the one-step error is $\mathcal{O}(h^2)$ and the walker makes $\mathcal{O}(1/h)$ steps; the error in the collision steps is $\mathcal{O}(h)$ but the average number of collision steps is bounded and independent of the time step $h$. The finite-time error is  naturally $\mathcal{O}(h)$. 
We note that the position $Q(t)$ in the continuous dynamics (\ref{cld}) and its discrete approximation $Q_k$ is bounded as we consider (\ref{cld}) in a bounded domain $\bar G$. At the same time, momentum $P(t)$ lives in $\mathbb{R}^d$ and hence, to prove convergence, we need that moments of its discrete approximation $P_k$ are bounded (see Lemma~\ref{cld_lemma_4.2}). The proof relies on the fact that  the drift in  (\ref{cld}) is growing at most linearly in $P(t)$ and the noise is additive. 

\begin{lemma}\label{cld_lemma_4.2}
Let Assumptions~\ref{cld_as:1}-\ref{as:2} hold, then the moments of the momentum $P_{k}$ constructed according to scheme [PA$_c$] applied to (\ref{cld}) are bounded, i.e. for any $m \geq 1$:
\begin{equation}\label{eq:momP}
    \mathbb{E}(|P_{k}|^{2m}) \leq C(1 + |P_{0}|^{2m}),
\end{equation}
where  $C>0$ is a constant independent of $h$.
\end{lemma}
\begin{proof}
Using (\ref{eq:PhatP}), we get
\begin{align*}
 &\mathbb{E}|P_{k+1}|^{2m}=\mathbb{E}|\hat P_{k+1}|^{2m}= \mathbb{E}|(\hat P_{k+1}-P_k)+P_k|^{2m} \\
& \leq \mathbb{E}|P_{k}|^{2m}+m\mathbb{E}|P_{k}|^{2m-2}[2(P_k \cdot(\hat P_{k+1}-P_k))+(2m-1)|\hat P_{k+1}-P_k|^2] \notag \\
&+ K \sum_{l=3}^{2m}\mathbb{E} |P_{k}|^{2m-l}  |\hat P_{k+1}-P_k|^l, \notag
\end{align*}
where $K>0$ depends on $m$ only. Thanks to (\ref{linegb}), it is not difficult to obtain  
\begin{equation*}
|\mathbb{E}[(\hat P_{k+1}-P_k)|P_k]| \leq Kh|P_k|, \,\,\, 
\mathbb{E}[|\hat P_{k+1}-P_k|^l|P_k] \leq Kh^{l/2} (1+|P_k|^l), \; l \ge 2,
\end{equation*}
with $K>0$ dependent on $m$, $b$ and the moments of $\xi_k$. 
We observe that
\begin{align}
 \mathbb{E}|P_{k+1}|^{2m} \leq \mathbb{E}|P_{k}|^{2m}+Kh \mathbb{E}|P_{k}|^{2m} +Kh \label{eq:momrecu}
\end{align}
and, using the Gr\"onwall lemma, we get (\ref{eq:momP}).
\end{proof}

Next we estimate the one-step error in the case when  collision with the boundary occurs.  
We recall that for fixed $(q,p)$ the output ${\cal Q},{\cal P}$ of A$_c$ step is deterministic as well as the number of multiple collisions $\kappa$ and collision times $\tau_i$, $i=1,\ldots, \kappa$.

\begin{lemma}\label{bouncary_collision_error_lemma}
Under Assumptions~\ref{cld_as:1}-\ref{cld_as:4a}, the  error incurred in A$_c$ step (see Algorithm~\ref{A$_c$ step with multiple}), given the current state of Markov chain is $(q,p)$,  is bounded by
\begin{align*}
    & \big|u(t_{1},{\cal Q},{\cal P}) - u(t_{1},q, p)\big|I_{\bar{G}^c}(\hat{Q}_1) \\
   & \leq C h \sum_{i = 1}^{\kappa}(n_{G}(Q_{c,i})\cdot \tilde{P}_{i-1})(1+|p|^{2m})I_{G^{c}}(\hat{Q}_{1}) +  C h^2 (1 + |p|^{2m}),
\end{align*}
where $\kappa := \kappa (q, p ,h)$ denotes the number  of multiple collisions in one step of duration $h$,   $u(t,q,p)$ is the solution of (\ref{eq2.1})-(\ref{eq2.3}), and $C>0$ is a constant independent of $h$ and $p$.
\end{lemma}

\begin{proof}
In the error expansions below in this proof, we consider the scenario in which the chain undergoes boundary collision, characterized by the condition $I_{\bar{G}^c}(\hat{Q}_1) = 1$.  
Application of Taylor's theorem gives
    \begin{align*}
        u(t_1,{\cal Q},{\cal P}) = u(t_1, \hat{Q}_{\kappa}, \tilde{P}_\kappa) = u(t_1, Q_{c,\kappa},  \tilde{P}_{\kappa}) + (h - \vartheta_{\kappa}) (\nabla_q u(t_1, Q_{c, \kappa}, \tilde{P}_{\kappa}) \cdot \tilde{P}_\kappa) + r_{\kappa},
    \end{align*}
where 
\begin{align}
    r_{\kappa} =  \frac{1}{2}  (h - \vartheta_{\kappa})^2\tilde{P}_{\kappa}^{\top} \nabla^2 u(t_1, Q_{c, \kappa} + \alpha \tilde{P}_{\kappa}, \tilde{P}_{\kappa}) \tilde{P}_{\kappa}, \quad \text{with}\quad \alpha \in (0,1),
\end{align}
and, therefore, thanks to Assumption~\ref{as:7}, $|r_{\kappa}| \leq C (h -\vartheta_{\kappa})^2 |p|^{2m}$ with $C>0$ being independent of $h$ and $p$.     
From \eqref{cld_new_eqn_3.11}, we have
\begin{align}\label{eq:Pbcagain}
    \tilde{P}_{\kappa} = \tilde{P}_{\kappa-1} -2 (n(Q_{c, \kappa})\cdot \tilde{P}_{\kappa-1})n(Q_{c, \kappa}).
\end{align}
Hence, using the boundary condition $u(t_1, Q_{c,\kappa},  \tilde{P}_{\kappa}) = u(t_1, Q_{c,\kappa},  \tilde{P}_{\kappa-1})$, we get
 \begin{align}
        u(t_1, {\cal Q}, {\cal P}) &= u(t_1, Q_{c,\kappa},  \tilde{P}_{\kappa}) + (h - \vartheta_{\kappa}) 
        (\nabla_q u(t_1, Q_{c, \kappa}, \tilde{P}_{\kappa}) \cdot \tilde{P}_{\kappa}) + r_{\kappa} \nonumber \\   
       &= u(t_1, Q_{c,\kappa},  \tilde{P}_{\kappa-1}) + (h - \vartheta_{\kappa}) (\nabla_q u(t_1, Q_{c, \kappa}, \tilde{P}_{\kappa}) \cdot \tilde{P}_{\kappa}) + r_{\kappa}. 
    \end{align}
We shift our focus to $u(t_1, Q_{c,\kappa},  \tilde{P}_{\kappa-1})$ and again use Taylor's theorem to obtain
\begin{align}
u(t_1, Q_{c,\kappa},  \tilde{P}_{\kappa-1}) = u(t_1, Q_{c, \kappa-1}, \tilde{P}_{\kappa-1})  + \tau_{\kappa}(\nabla_q u(t_1, Q_{c,\kappa-1}, \tilde{P}_{\kappa-1})\cdot \tilde{P}_{\kappa-1}) + \bar r_{\kappa-1},
\end{align}
where $|\bar r_{\kappa-1}| \leq C \tau_{\kappa}^{2} |p|^{2m}$ with $C>0$ being independent of $h$ and $p$.  

Using (\ref{cld_new_eq_3.16}) and \eqref{eq:Pbcagain} along with Taylor's theorem, we have 
\begin{align}
    &(\nabla_q u(t_1, Q_{c,\kappa-1}, \tilde{P}_{\kappa-1})\cdot \tilde{P}_{\kappa-1}) = (\nabla_q u(t_1, q, \tilde{P}_{\kappa-1}) \cdot \tilde{P}_{\kappa-1}) + \hat{r}_{\kappa-1} \label{cld_new_eq_4.6} \\ 
&\hspace{-45pt}\text{and} \nonumber \\   
 &   (\nabla_q u(t_1, Q_{c,\kappa}, \tilde{P}_{\kappa})\cdot \tilde{P}_{\kappa}) = (\nabla_q u(t_1, q, \tilde{P}_{\kappa})\cdot \tilde{P}_{\kappa}) +  \tilde{r}_{\kappa-1}  
    = (\nabla_q u(t_1, q, \tilde{P}_{\kappa})\cdot \tilde{P}_{\kappa-1})  +  \tilde{r}_{\kappa-1}  \nonumber   \\  
  &  \quad -2 (\nabla_q u(t_1, q, \tilde{P}_{\kappa}) \cdot n(Q_{c,\kappa}))( \tilde{P}_{\kappa-1}\cdot n(Q_{c,\kappa})) \nonumber  \\  
    & = (\nabla_q u(t_1, q, \tilde{P}_{\kappa-1})\cdot \tilde{P}_{\kappa-1})  +  \tilde{r}_{\kappa-1} 
      - 2(n(Q_{c,\kappa})\cdot \tilde{P}_{\kappa-1})(\textbf{J}_{\nabla_q u}(\tilde{P}_{\kappa-1, \theta}) n(Q_{c,\kappa})
      \cdot \tilde{P}_{\kappa-1})
  \nonumber   \\  
  &  \quad -2 (\nabla_q u(t_1, q, \tilde{P}_{\kappa}) \cdot n(Q_{c,\kappa}))( \tilde{P}_{\kappa-1}\cdot n(Q_{c,\kappa})) , \label{cld_new_eq_4.7}
\end{align}
where $|\hat{r}_{\kappa-1}| + |\tilde{r}_{\kappa-1}|  \leq C h |p|^{2m}$ and $\textbf{J}_{\nabla_q u}(\tilde{P}_{\kappa-1, \theta}) $ is the Jacobian of $\nabla_q u$ with respect to $p$ evaluated at $(t_1, q, \tilde{P}_{\kappa-1, \theta})$ with
\begin{align}
    \tilde{P}_{\kappa-1, \theta} = \tilde{P}_{\kappa-1} - 2 \theta (n(Q_{c,\kappa})\cdot \tilde{P}_{\kappa-1})n(Q_{c,\kappa}), \quad \theta \in (0,1).  
\end{align}

We combine the terms from \eqref{cld_new_eq_4.6} and \eqref{cld_new_eq_4.7} to yield 
\begin{align}
\tau_{\kappa}&(\nabla_q u(t_1, Q_{c, \kappa-1}, \tilde{P}_{\kappa-1})\cdot \tilde{P}_{\kappa -1})  
+ (h-\vartheta_{\kappa})(\nabla_q u(t_1, Q_{c,\kappa}, \tilde{P}_{\kappa}) \cdot \tilde{P}_{\kappa})    \nonumber 
\\  
& = (h- \vartheta_{\kappa-1})(\nabla_q u(t_1, q, \tilde{P}_{\kappa-1}) \cdot \tilde{P}_{\kappa-1}) + \tau_\kappa \hat{r}_{\kappa-1} +  (h-\vartheta_{\kappa})\tilde{r}_{\kappa-1} \nonumber \\ 
& \quad -  2( h - \vartheta_{\kappa})(n(Q_{c,\kappa})\cdot \tilde{P}_{\kappa-1})(\textbf{J}_{\nabla_q u}(\tilde{P}_{\kappa-1, \theta}) n(Q_{c,\kappa}) \cdot \tilde{P}_{\kappa-1}) \nonumber \\  
& \quad -2(h - \vartheta_{\kappa}) (\nabla_q u(t_1, q, \tilde{P}_{\kappa}) \cdot n(Q_{c,\kappa}))( \tilde{P}_{\kappa-1}\cdot n(Q_{c,\kappa})),
\end{align}
since $h- \vartheta_{\kappa}  + \tau_{\kappa} = h- \vartheta_{\kappa -1}$. 

As a result of the above computations, we arrive at
\begin{align}
    u(t_1, {\cal Q}, {\cal P}) &= u(t_1, Q_{c,\kappa-1}, \tilde{P}_{\kappa-1}) + (h- \vartheta_{\kappa-1})(\nabla_q u(t_1, q, \tilde{P}_{\kappa-1}) \cdot \tilde{P}_{\kappa-1}) 
    \nonumber \\  
    &  \quad + r_{\kappa} + \bar r_{\kappa-1} + \tau_{\kappa} \hat{r}_{\kappa-1} +  (h-\vartheta_{\kappa})\tilde{r}_{\kappa-1} \nonumber 
\\ 
& \quad  -2(h - \vartheta_{\kappa}) (\nabla_q u(t_1, q, \tilde{P}_{\kappa}) \cdot n(Q_{c,\kappa}))( \tilde{P}_{\kappa-1}
\cdot n(Q_{c,\kappa}))  \nonumber \\ 
& \quad 
-  2( h - \vartheta_{\kappa})(n(Q_{c,\kappa})\cdot \tilde{P}_{\kappa-1})(\textbf{J}_{\nabla_q u}(\tilde{P}_{\kappa-1, \theta}) n(Q_{c,\kappa}) \cdot \tilde{P}_{\kappa-1} ).
\end{align}
The repeated application of the boundary condition $u(t_1, Q_{c,i}, \tilde{P}_{i}) = u(t_1, Q_{c,i}, \tilde{P}_{i-1})$, $i=1,\dots, \kappa-1$, and Taylor's theorem, in the similar manner as above, yields 
\begin{align}
    & u(t_1, {\cal Q}, {\cal P})  = u(t_1, Q_{c, 1}, \tilde{P}_1) + (h-\vartheta_1)(\nabla_q u(t_1, q, \tilde{P}_{1})\cdot \tilde{P}_1)  + r_{\kappa} + \sum_{i=1}^{\kappa-1}(\bar r_i + \tau_{i +1} \hat{r}_{i}+(h-\vartheta_{i+1})\tilde{r}_{i})    \nonumber  \\
    & \quad -2 \sum_{i=1}^{\kappa-1}(h-\vartheta_{i+1})(\nabla_q u(t_1, q, \tilde{P}_{i+1})\cdot n(Q_{c, i+1}))(n(Q_{c, i+1})\cdot \tilde{P}_{i})
    \nonumber \\   
    & \quad -2 \sum_{i=1}^{\kappa-1}(h- \vartheta_{i+1}) (\textbf{J}_{\nabla_q u}(\tilde{P}_{i, \theta})  n(Q_{c, i+1}) \cdot \tilde{P}_i)( \tilde{P}_{i} \cdot n(Q_{c, i+1})). \label{cld_neweq_4.11}
\end{align}
One more application of Taylor's theorem gives
\begin{align*}
    u(t_1 , Q_{c,1}, \tilde{P}_1)  &=  u(t_1, q, \tilde{P}_1) + \tau_1 (\nabla_q u(t_1, q, \tilde{P}_1) \cdot \tilde{P}_1 ) + r_0 .
\end{align*}
Therefore,
\begin{align}
    \tau_1 (\nabla_q u(t_1, q, \tilde{P}_1) &+ ( h-\vartheta_1) \nabla_q u(t_1, q, \tilde{P}_1)  = h(\nabla_q u(t_1, q, \tilde{P}_1)\cdot \tilde{P}_1) \nonumber \\ 
   &  = h(\nabla_q u(t_1, q, p) \cdot p)   - 2(n(Q_{c,1})\cdot p)(\textbf{J}_{\nabla_q u}(\tilde{P}_{0, \theta}) n(Q_{c,1}) \cdot \tilde{P}_{0})
  \nonumber   \\  &  \quad -2 (\nabla_q u(t_1, q, \tilde{P}_{1}) \cdot n(Q_{c,1}))( p\cdot n(Q_{c,1})) \label{cld_neweq_4.12} 
\end{align}
since $\vartheta_1 = \tau_1$. 

Note that (see Assumption~\ref{cld_as:4a} and (\ref{eq:PhatP})):   
\begin{align*}
|n_{G}(Q_{c, i+1})\cdot \nabla_{q}u(t_{1}, q, \tilde{P}_{1})| &\leq C (1 + |\tilde{P }_{1}|^{2m}), \\
|\tilde{P}_{i,\theta}| \leq C|\mathcal{P}| , \\ 
|(\tilde{P}_{i} \cdot \textbf{J}_{\nabla_{q} u}(\tilde{P}_{i,\theta})n_{G}(Q_{c, i+1}))| \leq C(1 + |\mathcal{P}|^{2m}).
\end{align*}
Using \eqref{cld_neweq_4.12} in \eqref{cld_neweq_4.11}, we obtain 
\begin{align}
    |u(t_1, {\cal Q}, {\cal P}) - u(t_1, q, p)|  \leq  Ch (1 + |\mathcal{P}|^{2m})\sum_{i=1}^{\kappa (q,p,h)} (n(Q_{c, i}) \cdot \tilde{P}_{i-1}) + C h^2 (1 + |p|^{2m}). 
\end{align}
\end{proof}

The next lemma gives an estimate for the weak one-step error of [PA$_c$] scheme.

\begin{lemma}\label{lemma4.4}
Under Assumptions~\ref{cld_as:1}-\ref{cld_as:4a}, the one-step error  of scheme [PA$_c$] is estimated as
\begin{align*}
    & \big|\mathbb{E}[u(t_{k+1},Q_{k+1},P_{k+1}) - u(t_{k},Q_{k},P_{k})\mid Q_{k},P_{k}]\big| \\
   & \leq C\big(h^{2}(1+|P_{k}|^{2m}) + h\mathbb{E}\big[\sum_{i=1}^{\kappa}(n_{G}(Q_{k,c,i})\cdot \tilde{P}_{k+1,i-1})(1+|P_{k+1}|^{2m})I_{G^{c}}(\hat{Q}_{k+1})\big|Q_{k},P_{k}\big]\big) ,
\end{align*}
where $\tilde{P}_{k+1, 0} = \hat{P}_{k+1}$, $\kappa := \kappa(Q_k, \hat{P}_{k+1}, h)$,  $u(t,q,p)$ is the solution of (\ref{eq2.1})-(\ref{eq2.3}), and $C>0$ is a constant independent of $h$ and $P_{k}$.
\end{lemma}
\begin{proof}
Recall that the letter $C$ is used for positive constants and $m$ for positive integers which can change from line to line.
We note that it follows from (\ref{eq:momrecu}) that 
\begin{align*}
    \mathbb{E}[|P_{k+1}|^{2m} \mid Q_{k},P_{k}]\leq C(h+|P_{k}|^{2m}),
\end{align*}
where $C>0$ is independent of $h$, $P_{k}$, $Q_k$.

We have
\begin{align}
u(&t_{k+1},Q_{k+1},P_{k+1}) -u(t_{k},Q_{k},P_{k}) =  (u(t_{k+1},Q_{k+1},P_{k+1}) - u(t_{k+1},Q_{k},\hat{P}_{k+1}))I_{G}(\hat{Q}_{k+1}) 
\label{eq:lem44_1} \\ 
& +(u(t_{k+1},Q_{k+1},P_{k+1}) - u(t_{k+1},Q_{k},\hat{P}_{k+1}))I_{\bar{G}^{c}}(\hat{Q}_{k+1}) +
(u(t_{k+1},Q_{k},\hat{P}_{k+1}) - u(t_{k},Q_{k},P_{k})) \notag \\
&=: A_{1} + A_{2} + A_{3}. \notag
\end{align}
Consider the first term. Using Taylor's expansion and Assumption~\ref{cld_as:4a}, we obtain 
\begin{align*}
    A_{1} &:= (u(t_{k+1},Q_{k+1},P_{k+1}) - u(t_{k+1},Q_{k},\hat{P}_{k+1}))I_{G}(\hat{Q}_{k+1}) 
   \\ & = (u(t_{k+1},Q_{k+1},\hat{P}_{k+1}) - u(t_{k+1},Q_{k},\hat{P}_{k+1}))I_{G}(\hat{Q}_{k+1})
   \\&= \Big(h\big(\hat{P}_{k+1}\cdot\nabla_{q}u(t_{k+1},Q_{k},\hat{P}_{k+1})\big) + r_1\Big)I_{G}(\hat{Q}_{k+1}), 
\end{align*}
where the remainder $r_1$ is
\begin{equation*}
    r_1 =  (h^2/2) (\hat{P}_{k+1})^{\top}\nabla^2 u(t_{k+1}, Q_k + \alpha h \hat{P}_{k+1})\hat{P}_{k+1} 
\end{equation*}
with $\alpha \in (0,1)$ and it 
satisfies 
\[
|\mathbb{E}[r_1 I_{G}(\hat{Q}_{k+1})\mid Q_{k},P_{k}]| \leq Ch^2 (1 + |P_{k}|^{2m})
\]
with $C>0$ being a constant independent of $h$, $P_{k}$ and $Q_k$.

Using Lemma~\ref{bouncary_collision_error_lemma}, we get
\begin{align*}
    &A_{2}:= (u(t_{k+1},Q_{k+1},P_{k+1}) - u(t_{k+1},Q_{k},\hat{P}_{k+1}))I_{G^{c}}(\hat{Q}_{k+1}) \\   & = h(\nabla_q u(t_{k+1}, Q_k , \hat{P}_{k+1}) \cdot \hat{P}_{k+1}) + C h( 1+ |P_{k+1}|^{2m}) \sum_{i=1}^{\kappa}(n_G(Q_{k+1, c, i})\cdot \tilde{P}_{k+1, i-1}) \\   &  \quad + C h^2 (1 + |P_{k+1}|^{2m}).
\end{align*}
 
 Combining the obtained expressions for $A_{1}$ and $A_{2}$, we ascertain that
 \begin{align*}
A_{1} + A_{2} &=  h (\hat{P}_{k+1} \cdot \nabla_{q}u(t_{k+1}, Q_{k}, \hat{P}_{k+1}))  \nonumber \\ 
& \;\;\; + Ch( 1+ |P_{k+1}|^{2m}) \sum_{i=1}^{\kappa}(n_G(Q_{k+1, c, i})\cdot \tilde{P}_{k+1, i-1}) + r_2,
 \end{align*}
 with $r_2$ satisfying 
  \begin{equation}
 |\mathbb{E}[r_2 \mid Q_{k},P_{k}]| \leq Ch^2 (1 + |P_{k}|^{2m}) \label{eq:lem44_r5}
 \end{equation} 
 with $C>0$ being a constant independent of $h$, $P_{k}$, $Q_k$. 
Taking the expectation conditioned on $(Q_{k}, P_{k})$, we get
\begin{align}
    &\mathbb{E}(A_{1} + A_{2}|\;Q_{k},P_{k}) =  h \mathbb{E}\big[(\hat{P}_{k+1} \cdot \nabla_{q}u(t_{k+1}, Q_{k}, \hat{P}_{k+1}))|Q_{k},P_{k}\big] 
      \nonumber \\ 
     & \;\;\; + C h( 1+ |P_{k}|^{2m}) \mathbb{E}\Big(\sum_{i=1}^{\kappa}(n_G(Q_{k+1, c, i})\cdot \tilde{P}_{i-1})\; | Q_k ,P_k\Big)  + \mathbb{E}[ r_{2} \mid Q_{k},P_{k}]\nonumber\\ 
     & =  h \mathbb{E}\big[(P_{k} \cdot \nabla_{q}u(t_{k}, Q_{k}, P_{k}))|Q_{k},P_{k}\big]  + r_{3}  + C h( 1+ |P_{k}|^{2m}) \mathbb{E}\Big(\sum_{i=1}^{\kappa}(n_G(Q_{k+1, c, i})\cdot \tilde{P}_{i-1})\; | Q_k , P_k \Big) ,\label{cld_eq_5.20}
\end{align}
where we have expanded $\nabla_{q}u(t_{k+1}, Q_{k}, \hat{P}_{k+1})$ around $P_{k}$ and then $\nabla_{q} u(t_{k+1}, Q_{k}, P_{k})$ around $t_{k}$,
and $r_3$ satisfies  
  \begin{equation}
 |r_3| \leq Ch^2 (1 + |P_{k}|^{2m}) \label{eq:lem44_r6}
 \end{equation} 
 with $C>0$ being a constant independent of $h$, $P_{k}$, $Q_k$. 
 
With the usual Taylor expansion calculation for the Euler scheme for SDEs (see e.g. \cite[Chap. 7]{milstein2021stochastic}), we can write
\begin{align}
    \mathbb{E}(A_{3}|\;Q_{k},P_{k}) & = h\frac{\partial}{\partial t}u(t_{k},Q_{k}, P_{k}) +  h(b_{k}\cdot\nabla_{p}u(t_{k},Q_{k},P_{k})) + h\frac{\sigma^{2}}{2}\Delta_{p}u(t_{k},Q_{k},P_{k})  +r_{4}, \label{cld_eq_5.21}
\end{align}
where $r_{4}$ satisfies (\ref{eq:lem44_r6}).

Thus, using (\ref{cld_eq_5.20}), (\ref{cld_eq_5.21}) and (\ref{eq2.1}), we infer
\begin{align*}
    &|\mathbb{E}(u(t_{k+1}, Q_{k+1}, P_{k+1}) - u(t_{k},Q_{k},P_{k})|\;Q_{k}, P_{k})|\\
    & \leq C\big(h^{2}(1+|P_{k}|^{2m}) + C h( 1+ |P_{k}|^{2m}) \mathbb{E}\Big(\sum_{i=1}^{\kappa}(n_G(Q_{k+1, c, i})\cdot \tilde{P}_{k+1, i-1})\; | Q_k , P_k \Big),
\end{align*}
where $C>0$ is independent of $h$ and $P_{k}$. 

\end{proof} 

The estimate in the above local error lemma contains the term of order $\mathcal{O}(h)$ corresponding to the error of approximating elastic reflections. Consequently, the next lemma, related to the average number of collisions when the Markov chain $(Q_{k},P_{k})$ encounters the boundary $\Sigma_{+}$, plays a major role in global convergence analysis of the considered first-order methods.   
Its proof rests on the boundary value problem associated with the Markov chain $(t_{k},Q_{k}, P_{k})$ (cf. (\ref{eq:Koop})):
\begin{align}
    \mathbb{T}V(t,q,p) - V(t,q,p) &= -g(t,q,p),  &(t,q,p) \in [T_{0}, T-h]\times\bar{G}\times \mathbb{R}^{d}, \label{cld_eq:5.1} \\
    V(t,q,p) &= 0,       &(t,q,p) \in \{T\} \times \bar{G}\times \mathbb{R}^{d},\label{cld_eq:5.2}
\end{align}
where $\mathbb{T}$ is the one-step transition operator:
\begin{equation*}
    \mathbb{T}V(t,q,p) = \mathbb{E}V(t_{1}, Q_{1}, P_{1}), \; \; (t_{0},Q_{0},P_{0}) = (t,q,p),
\end{equation*}
and $g(t,q,p) \geq 0$. The solution to this problem is given by \cite{20, 46,milstein2021stochastic}:
\begin{equation}
    V(t_{0},q,p) =  \mathbb{E}\bigg(\sum\limits_{k=0}^{N-1}g(t_{k},Q_{k},P_{k})\Big| Q_{0} = q,P_{0} = p\bigg). \label{eq:5.3}
\end{equation}

\begin{lemma} \label{bl}
Let  Assumptions~\ref{cld_as:1}-\ref{as:2} hold, then for any integer $m \geq 1$ the following estimate holds for 
[PA$_c$] scheme:
\begin{equation}
    \mathbb{E}\bigg(\sum\limits_{k=0}^{N-1}\sum_{i=1}^{\kappa_k}(n_{G}(Q_{k+1, c, i})\cdot \tilde{P}_{k+1, i-1})(1+|P_{k+1}|^{2m})I_{G^{c}}\big(\hat{Q}_{k+1}\big) \bigg)\leq C, \label{eqn_cld_5.11}
\end{equation}
where $\kappa_k := \kappa (Q_k, \hat{P}_{k+1}, h)$, $k=0, \dots, N-1$, and $C>0$ is a constant independent of $h$.
\end{lemma}

\begin{proof}

 Let us take $g(t_k,q,p)$ in (\ref{cld_eq:5.1}) as 
\[
g(t_k,q,p)= \mathbb{E}\bigg[ \sum_{ i = 1}^{\kappa }(n_{G}(Q_{k+1, c, i})\cdot \tilde{P}_{k+1, i-1})(1+|P_{k}|^{2m})I_{G^{c}}(\hat{Q}_{k+1})|\;Q_k=q,P_k=p\bigg],
\]
where $ \kappa = \kappa (q, \hat{P}, h)=\kappa (q, p, h,\xi_{k+1})$ with $\hat P =p+hb(q,p)+h^{1/2}\sigma \xi_{k+1}$. Note that thanks to Proposition~\ref{prop:mult}, $g(t_k,q,p)$ is well defined. 
Then the solution of (\ref{cld_eq:5.1})-(\ref{cld_eq:5.2}) is given by
\begin{align*}
   v&(t_{0},q,p) =  
   \mathbb{E}\bigg(\sum_{k=0}^{N-1} g(t_{k},Q_{k}, P_{k}) \big| Q_{0} = q , P_{0} = p\bigg)
   \\ & =\mathbb{E}\bigg(\sum\limits_{k=0}^{N-1}\mathbb{E}\Big( \sum_{i = 1}^{\kappa_{k}}(n_{G}(Q_{k+1 ,c, i})\cdot \tilde{P}_{k+1, i-1})(1+|P_{k+1}|^{2m})I_{G^{c}}\big(\hat{Q}_{k+1}\big)\big| Q_{k},P_{k}\Big) \Big| Q_{0} = q, P_{0} = p \bigg)  
   \\ &=\mathbb{E}\bigg(\sum\limits_{k=0}^{N-1}\sum_{i=1}^{\kappa_{k} }\big(n_{G}(Q_{k+1, c, i})\cdot \tilde{P}_{k+1, i-1}\big)\big(1 + |P_{k+1}|^{2m}\big)I_{G^{c}}\big(\hat{Q}_{k+1}\big)\Big| Q_{0} = q, P_{0} = p \bigg).
\end{align*}

If we can find a function $V(t,q,p)$ with $g(t,q,p)$, satisfying (\ref{cld_eq:5.1})-(\ref{cld_eq:5.2}), such that 
\begin{equation}\label{eq:forg}
g(t,q,p) \geq \mathbb{E}\big[ \sum_{i=1}^{\kappa }(n_{G}(Q_{c, i})\cdot \tilde{P}_{i-1})(1+|P_{1}|^{2m})I_{G^{c}}(\hat{\mathcal{Q}}_{1}) | q, p\big],\quad  \kappa := \kappa (q,\hat{P}_1, h),
\end{equation}
then the inequality, $v(t_{0},Q_{0},P_{0}) \leq V(t_{0},Q_{0},P_{0})$ will hold. Therefore, the aim is to construct 
a function $g = -(\mathbb{T}V - V)$ satisfying (\ref{eq:forg}). 

Consider a boundary zone $G_{R_{0}} \subset G$ defined as $G_{R_{0}} := \{ q \in \bar{G}\; ; \; \dist(q,\partial G) < R_{0}\}$. Due to \cite[Appendix]{104},  there is an $R_{0}$ such that the projection of $q$ on $\partial G$ is unique whenever $q \in G_{R_{0}}$. Introduce the function
\begin{equation*}
    V(t,q,p) =  \begin{cases}
    0,& (t, q, p) \in \{T\}\times\bar{G}\times \mathbb{R}^{d},\\
    (A + e^{K(T-t)})(1+| p|^{2m +2})  +  (B(q)\cdot p)(1+|p|^{2m}), & (t, q, p) \in [0, T-h]\times\bar{G}\times \mathbb{R}^{d},
    \end{cases} 
\end{equation*}
where $B(q) \in C^{3}(\bar{G})$ is  an $\mathbb{R}^{d}$ valued function such that  $B(q) : = n_{G}(q^{\pi})$ when $q \in \bar{G}_{R_{0}}$; here $q^{\pi}$ is the projection of $q$ on $\partial G$. Note that $n_{G}(q) \in C^{4}(\bar{G})$ due to Assumption~\ref{cld_as:1} (see \cite[Appendix]{104}). We extend the function $B(q)$ from $G_{R_{0}}$  to $\bar{G}$ so that $B(q) \in C^{3}(\bar{G})$ and $|B(q)|_{G}^{2,\alpha} \leq C |n(q^{\pi})|_{G_{R_{0}}}^{2,\alpha}$, where $|\cdot|^{2,\alpha}$ is the H\"older norm and $C>0$ depends on $G$. Such an extension is possible due to \cite[Lemma 6.37]{104}. 
We choose the constant $A>0$  so that $V(t,q,p)$ is always positive. 
The choice of the constant $K>0$ will be discussed later in the proof. 

The proof of Lemma~\ref{cld_lemma_4.2} implies for any $\ell \ge 1$
 \begin{align}
     & |P_{1}|^{2\ell} = |\hat{P}_{1}|^{2\ell}  
     = |p|^{2\ell}  +2\ell\sigma h^{1/2}(\xi \cdot p)|p|^{2\ell-2} +  r_{1}, \label{cld_eqn_4.16}
 \end{align}
 and
 \begin{align}
       & \mathbb{E}[|P_{1}|^{2\ell}\;|q,p] = |p|^{2\ell} +  r_{2}, \label{cld_eqn_4.10}
 \end{align}
 where $r_{1}$ and $r_{2}$ satisfy for any $\ell\ge 1$:
 \[
     (\mathbb{E}(|r_{1}|^{2\ell}))^{1/(2\ell)} \leq Ch(1 + |p|^{2\ell })
 \]
 and
 \[
     |r_{2}| \leq Ch(1+|p|^{2\ell})
 \]
 with $C>0$ being independent of $h$ and $p$.
 
Given $(t_{k},Q_{k}, P_{k}) = (t,q,p)$, we now aim at calculating $-(\mathbb{T}V(t,q,p) - V(t,q,p))$. 
Using Taylor's theorem, we obtain 
\begin{align}
    B(Q_{1}) &= \big(B(q)  + h\textbf{J}_{B}(q +\epsilon_{0} h\hat{P}_{1})\hat{P}_{1} \big)I_{G}(\hat{Q}_{1})
     \nonumber \\ 
     & \;\;\; +  \big(B(Q_{c, \kappa}) + (h-\vartheta_{\kappa})\textbf{J}_{B}(Q_{c, \kappa}+ \epsilon_{\kappa}(h- \vartheta_{\kappa})\tilde{P}_{\kappa})\tilde{P}_{\kappa}\big)I_{G^{c}}(\hat{Q}_{1}), 
 \end{align}
where $\epsilon_{\kappa} \in (0,1)$ and $\textbf{J}_{B}(x)$ denotes the Jacobian matrix of $B$ evaluated at $x$. Therefore, we get
\begin{align}
     (B(Q_{1})&\cdot P_{1}) = (B(Q_{1})\cdot \hat{P}_{1}) I_{G}(\hat{Q}_{1})  +  (B(Q_{1}) \cdot P_{1})I_{G^{c}}(\hat{Q}_{1}) \nonumber  \\
     &  = \big((B(q)\cdot \hat{P}_{1}) + h(\textbf{J}_{B}(q +\varepsilon_{0} h\hat{P}_{1})\hat{P}_{1} \cdot \hat{P}_{1}) \big)I_{G}(\hat{Q}_{1}) \nonumber \\& \;\;\;  +   \big((B(Q_{c,\kappa})\cdot  \tilde{P}_{\kappa}) + (h-\vartheta_{\kappa})(\textbf{J}_{B}(Q_{c, \kappa}+ \epsilon_{\kappa}(h- \vartheta_{\kappa})\tilde{P}_{\kappa})\tilde{P}_{\kappa} \cdot \tilde{P}_{\kappa})I_{G^{c}}(\hat{Q}_{1}) \big)
    \nonumber \\ 
     &  = \big((B(q)\cdot \hat{P}_{1}) + h(\textbf{J}_{B}(q +\varepsilon_{0} h\hat{P}_{1})\hat{P}_{1} \cdot \hat{P}_{1}) \big)I_{G}(\hat{Q}_{1}) \nonumber \\
     & \;\;\;  +   \big((B(Q_{c,\kappa})\cdot  \tilde{P}_{\kappa-1}) - 2 (\tilde{P}_{\kappa-1} \cdot n_{G}(Q_{c, \kappa}))(B(Q_{c,\kappa})\cdot n_{G}(Q_{c,\kappa}))
     \nonumber  \\
     & \;\;\; 
     + (h-\vartheta_{\kappa})(\textbf{J}_{B}(Q_{c,\kappa}+ \epsilon_{\kappa}(h- \vartheta_{\kappa})\tilde{P}_{\kappa})\tilde{P}_{\kappa} \cdot P_{1})\big)I_{G^{c}}(\hat{Q}_{1}) ,
\end{align}
where $\varepsilon_{0} \in (0,1)$. 
Since $(B(Q_{c, \kappa})\cdot n_{G}(Q_{c, \kappa})) = 1$ (by construction of $B$), we have
\begin{align}
     (B(Q_{1})\cdot P_{1})  &  = \big((B(q)\cdot \hat{P}_{1}) + h(\textbf{J}_{B}(q +\varepsilon_{0} h\hat{P}_{1})\hat{P}_{1} \cdot \hat{P}_{1}) \big)I_{G}(\hat{Q}_{1}) \nonumber \\
     & \;\;\;  +   \big((B(Q_{c,\kappa})\cdot  \tilde{P}_{\kappa-1}) - 2 (\tilde{P}_{\kappa-1} \cdot n_{G}(Q_{c, \kappa}))I_{G^{c}}(\hat{Q}_{1})
     \nonumber  \\
     & \;\;\; 
     + (h-\vartheta_{\kappa})(\textbf{J}_{B}(Q_{c,\kappa}+ \epsilon_{\kappa}
     (h- \vartheta_{\kappa})\tilde{P}_{\kappa})\tilde{P}_{\kappa} \cdot P_{1}) I_{G^{c}}(\hat{Q}_{1}) \big). 
\end{align}
For the sake of convenience in notation below, we denote $Q_{c,0} := q$. Using Taylor's theorem, we have
 \begin{align}    
     B(Q_{c, i}) &= B(Q_{c, i-1}) +  \tau_i \textbf{J}_{B}(Q_{c,i-1} +\epsilon_{i-1} \tau_{i} \tilde{P}_{i-1})\tilde{P}_{i-1},\quad i = 1,\dots, \kappa, \label{cld_neweq_5.13}      
\end{align} 
where $\epsilon_i \in (0,1) $, $i = 1, \dots, \kappa-1$. Also,
\begin{align}
     \tilde{P}_i = \tilde{P}_{i-1} - 2 (n(Q_{c,i})\cdot \tilde{P}_{i-1}) \tilde{P}_{i-1}, \quad i=1,\dots, \kappa-1.   \label{cld_neweq_5.14}   
\end{align}
With the repeated use of (\ref{cld_neweq_5.13})-(\ref{cld_neweq_5.14}) and the fact that $(B(Q_{c,i}) \cdot n_G(Q_{c,i})) = 1$, $i=1,\dots,\kappa$, we obtain
\begin{align}   
(B(Q_{1})&\cdot P_{1})   = \big((B(q)\cdot \hat{P}_{1}) + h(\textbf{J}_{B}(q +\varepsilon_{0}  h\hat{P}_{1})\hat{P}_{1} \cdot \hat{P}_{1}) \big)I_{G}(\hat{Q}_{1}) \nonumber \\
& \;\;\;  +   \big((B(q)\cdot  \tilde{P}_{0})  - 2\sum_{i=1}^{\kappa} (\tilde{P}_{i-1} \cdot n_{G}(Q_{c, i}))
\nonumber \\ 
& \;\;\; + (h-\vartheta_{\kappa})(\textbf{J}_{B}(Q_{c, \kappa}+ \epsilon_{\kappa}(h- \vartheta_{\kappa}) \tilde{P}_{\kappa})\tilde{P}_{\kappa} \cdot \tilde{P}_{\kappa})
\nonumber \\  
& \;\;\; 
+ \sum_{i=1}^{\kappa}\tau_i (\textbf{J}_{B}(Q_{c,i-1} + \epsilon_{i-1} \tau_{i}\tilde{P}_{i-1})\tilde{P}_{i-1} \cdot \tilde{P}_{i-1})\big)I_{G^{c}}(\hat{Q}_{1}) 
. \label{eq:lem44_18}
\end{align}
Due to the boundedness of first derivatives of $ B$, we ascertain
\begin{align*}
    (\textbf{J}_{B}(q +\varepsilon_{0}  h\hat{P}_{1})\hat{P}_{1} \cdot \hat{P}_{1})& \leq C|\hat{P}_{1}|^{2}, \\
    (\textbf{J}_{B}(Q_{c, L}+ \epsilon_{\kappa}(h- \vartheta_\kappa)\tilde{P}_{\kappa})\tilde{P}_{\kappa} \cdot \tilde{P}_{\kappa}) & \leq  C|\tilde{P}_{\kappa}|^{2} =C|\hat{P}_{1}|^{2}, \\
    (\textbf{J}_{B}(Q_{c,i} +\epsilon_{i-1} \tau_i \tilde{P}_{i})\tilde{P}_{i} \cdot \tilde{P}_{i}) &  \leq C|\tilde{P}_{i}|^{2} = C|\hat{P}_1|^2, 
\end{align*}
where $C>0$ is a generic constant independent of $\hat{P}_{1}$. Using these estimates, (\ref{eq:lem44_18}) becomes 
\begin{align}   \label{cld_eqn_4.32}
(B(Q_{1})&\cdot P_{1}) = (B(q)\cdot \hat{P}_{1}) - 2 \sum_{i=1}^{\kappa}(\tilde{P}_{i-1} \cdot n_{G}(Q_{c, i}))I_{G^{c}}(\hat{Q}_{1}) + r_{3},
\end{align}
where $|r_{3}| \leq Ch|\hat{P}_{1}|^{2}$ with $C >0$ being independent of $h$ and $\hat{P}_{1}$.


Using (\ref{cld_eqn_4.16}) and (\ref{cld_eqn_4.10}), we get
\begin{align}
    \mathbb{E}\big[(B(q) &\cdot \hat{P}_{1}) (1 + |P_{1}|^{2m})\big] 
    = (B(q)\cdot p)\mathbb{E}[(1 + |P_{1}|^{2m})] + h(B(q)\cdot b)\mathbb{E}[(1 + |P_{1}|^{2m})] \nonumber \\ 
    & \;\;\; + \sigma h^{1/2}\mathbb{E}[(B(q)\cdot \xi)(1 + |P_{1}|^{2m})]
    = (B(q)\cdot p)(1 + |p|^{2m}) + r_{4},   \label{cld_eqn_4.22}
\end{align}
where $|r_{4}|\leq Ch(1 + |p|^{2m+1})$ with $C>0$ being independent of $h$ and $p$. From (\ref{cld_eqn_4.32}) and (\ref{cld_eqn_4.22}), we have
\begin{align}
    \mathbb{E}[(B(Q_{1})\cdot P_{1})(1 + |P_{1}|^{2m})]
    &= (B(q)\cdot p)(1 + |p|^{2m})   \nonumber  \\ 
    & - 2\mathbb{E}\Big[\sum_{i=1}^{\kappa} (\tilde{P}_{i-1} \cdot n_{G}(Q_{c, i}))(1 + |P_{1}|^{2m})I_{G^{c}}(\hat{Q}_{1}) \Big] + r_{5}, \label{cld_eqn_4.23}
\end{align}
where $|r_{5}|\leq Ch(1 + |p|^{2m+2})$ with $C>0$ being independent of $h$ and $p$.

Using (\ref{cld_eqn_4.10}) with $\ell = m+1$, we have
\begin{align}
    &(A + e^{K(T-t- h)})(1 + \mathbb{E}[|P_{1}|^{2m +2}])  
    =  \big(A + e^{K(T-t)}(1 - Kh + K^{2}r_6\big)\nonumber  \big(1 + |p|^{2m +2} +  r_{2}\big) \nonumber \\ 
    & =  (A + e^{K(T-t)})(1 + |p|^{2m +2} + r_{2}) 
    -K h e^{K(T-t)}(1 + |p|^{2m +2}) +r_7 e^{K(T-t)}, \label{cld_eqn_4.24}
\end{align}
where $r_6 \leq Ch^{2}$ and $r_7 \leq Ch^{2}(1+K^2)(1 + |p|^{2m +2})$ with $C>0$ independent of $h$, $K$ and $p$.
Recall that $r_{2} \leq Ch(1+|p|^{2m +2})$, where $C>0$ is independent of $h$, $K$ and $p$.
 Therefore, using (\ref{cld_eqn_4.23}) and (\ref{cld_eqn_4.24}), we get
\begin{align}
\mathbb{E}[&V(t_{1},Q_{1}, P_{1})] - V(t,q,p) =  (A + e^{K(T-t- h)})(1 + \mathbb{E}[|P_{1}|^{2m +2}]) \nonumber \\ 
& \;\;\;\;+ \mathbb{E}[(B(Q_{1})\cdot P_{1})(1 + |P_{1}|^{2m})]  - (A + e^{K(T-t)})(1 + |p|^{2m + 2}) 
- (B(q)\cdot p)(1 + |p|^{2m}) \nonumber  \\ 
& = (A + e^{K(T-t)})r_{2}-K h e^{K(T-t)}(1 + |p|^{2m +2}) +r_7 e^{K(T-t)}+r_5\nonumber  \\ 
& \;\;\;  - 2 \mathbb{E}\Big[\sum_{i=1}^{\kappa}(\tilde{P}_{i-1} \cdot n_{G}(Q_{c, i}))( 1+ |P_{1}|^{2m}) I_{G^{c}}(\hat{Q}_{1})\big| q, p\Big]. 
\label{cldeqn5.15}
\end{align}
Hence, there are sufficiently large $K>0$ and sufficiently small $h>0$ so that $g(t,q,p) := - (\mathbb{T}V-V) \geq \mathbb{E}\big[ \sum_{i=1}^{\kappa}(\tilde{P}_{i-1}\cdot n_{G}(Q_{c, i}))(1+|P_{1}|^{2m})I_{G^{c}}(\hat{Q}_{1})|q,p\big] $. This proves that (\ref{eqn_cld_5.11}) holds with a constant $C>0$  independent of $h$ (it depends on $P_0$ and $T$). 

\end{proof}

\begin{remark}
On the physical level of rigour, the meaning of Lemma~\ref{bl} is that the average number of steps when the Markov chain has collisions with the boundary does not depend on the time step $h$. This is in contrast to the case of reflected SDEs (\ref{rgsde}), where the average number of steps when an approximating Markov chain interacts with the boundary is of order ${\cal O} (1/\sqrt{h})$ \cite{lst23}. The latter leads to additional difficulties in constructing higher order weak methods for (\ref{rgsde}) (see \cite{lst23,sharma2022random}), while the former contributes to possibilities for constructing second-order schemes for Langevin dynamics.  
\end{remark}

Now we are ready to prove Theorem~\ref{theorem4.1}.

\begin{proof}[\textbf{Proof of Theorem~\ref{theorem4.1}}] Using Lemma~\ref{lemma4.4}, we have 
\begin{align*}
    |\mathbb{E}(\varphi(Q_{N},P_{N})) - \mathbb{E}(\varphi(Q(T),P(T)))| = \Big|\mathbb{E}\Big(\sum\limits_{k=0}^{N-1}\mathbb{E}\big(u(t_{k+1},Q_{k+1},P_{k+1}) -u(t_{k},Q_{k},P_{k}) | Q_{k},P_{k}\big)\Big)\Big| \nonumber  \\ 
    \leq Ch^{2}\mathbb{E}\Big(\sum\limits_{k=0}^{N-1}|P_{k}|^{2m}\Big) + Ch\mathbb{E}\Big(\sum\limits_{k=0}^{N-1}\sum\limits_{i = 1}^{\kappa_k}(n_{G}({Q}_{k+1,c,i})\cdot\tilde{P}_{k+1, i-1})I_{G^{c}}(\hat{Q}_{k+1}) (1+|P_{k+1}|^{2m})\Big),
\end{align*}
which on applying Lemma~\ref{cld_lemma_4.2} and Lemma~\ref{bl} gives the required result.
\end{proof}

\subsection{Proof of Theorem~\ref{theorem4.2}}\label{sec:proof2}

The proof of Theorem~\ref{theorem4.2} has the same ingredients as the proof of Theorem~\ref{theorem4.1} in the previous section with the major difference that the global error bound in Theorem~\ref{theorem4.2} should not depend on the integration time $T$ since here we are interested in ergodic limits associated with (\ref{ecld}), while the error bound in Theorem~\ref{theorem4.1} for the approximation of (\ref{cld}) on the finite time interval $[0,T]$ can depend on $T$. We start (Lemma~\ref{cld_lemma_5.6}) by estimating moments of the momentum $P_{N}$ uniformly in $N$ and $h$. Next, in Lemma~\ref{lemma5.8},  we consider the one-step (local) error. An estimate related to the average number of collisions of the Markov chain with the boundary is obtained in Lemma~\ref{ebl}. 

\begin{lemma}\label{cld_lemma_5.6}
Let Assumptions~\ref{cld_as:1} and \ref{as:5} hold and $ h >0 $ is sufficiently small, then the moments of the momentum $P_{k}$ constructed according to the scheme [PA$_c$] applied to (\ref{ecld}) are bounded, i.e. for any $m \geq 1$:
\begin{equation}\label{cld_eqn_5.45}
    \mathbb{E}(|P_{k}|^{2m}) \leq C,
\end{equation}
where  $C>0$ is independent of $h$ and $T$.
\end{lemma}
\begin{proof} 
    The constant $K>0$ is changing from line to line. Using (\ref{eq:PhatP}), we
get (see the proof of Lemma~\ref{cld_lemma_4.2}):%
\begin{align*}
& \mathbb{E}|P_{k+1}|^{2m}=\mathbb{E}|\hat{P}_{k+1}|^{2m}\leq \mathbb{E}%
|P_{k}|^{2m}+m\mathbb{E}|P_{k}|^{2m-2}[2(P_{k}\cdot (\hat{P}%
_{k+1}-P_{k}))+(2m-1)|\hat{P}_{k+1}-P_{k}|^{2}] \\
& +K\sum_{l=3}^{2m}\mathbb{E}|P_{k}|^{2m-l}|\hat{P}_{k+1}-P_{k}|^{l},
\end{align*}%
where $K>0$ depends on $m$ only. It is not difficult to obtain%
\begin{eqnarray*}
\mathbb{E}[P_{k}\cdot (\hat{P}_{k+1}-P_{k})|P_{k}] &\leq &-\gamma
h|P_{k}|^{2}+Kh|P_{k}|,\,\,\, 
\mathbb{E}[|\hat{P}_{k+1}-P_{k}|^{2}|P_{k}] \leq Kh^{2}|P_{k}|^{2}+\frac{%
2\gamma }{\beta }h, \\
\mathbb{E}[|\hat{P}_{k+1}-P_{k}|^{l}|P_{k}] &\leq
&Kh^{l}|P_{k}|^{l}+Kh^{l/2},\;l\geq 3,
\end{eqnarray*}%
with $K>0$ independent of $h$ (it depends on $\max_{q\in \bar{G}}|\nabla_q
U(q)|,\ \gamma ,$ $\beta ,$ $l$ and moments of $\xi )$. Using the
generalized Young's inequality, i.e. $ ab  \leq  \epsilon^{i}\frac{a^{i}}{i} + \epsilon^{-j} \frac{b^{j}}{j} $; $a, b \geq 0$, $\epsilon>0$; $i,j >1$; $\frac{1}{i} + \frac{1}{j} =1$, with $a = |P_{k}|^{2m-1}$, $b = K$, $ i = \frac{2m}{(2m -1)}$, $ j = 2m$ and $\epsilon = \gamma^{\frac{2m-1}{2m}}$, we get  
\begin{eqnarray*}
&& 2m\mathbb{E}\left[ |P_{k}|^{2m-2}(P_{k}\cdot (\hat{P}_{k+1}-P_{k}))\right] 
\leq -2m\gamma h \mathbb{E}|P_{k}| ^{2m} + 2mh K\mathbb{E}(|P_{k}|^{2m-1}) \\ 
&& \leq -2m\gamma h \mathbb{E}|P_{k}| ^{2m} + (2m -1)\gamma h \mathbb{E}|P_{k}|^{2m} + h K^{2m} \left(\frac{1}{\gamma}\right)^{2m -1}
\\&&\leq -\gamma h\mathbb{E}|P_{k}|^{2m}  + h K^{2m} \left(\frac{1}{\gamma}\right)^{2m -1}.
\end{eqnarray*}
In the similar manner, we have
\begin{eqnarray*}
(2m-1)\mathbb{E}|P_{k}|^{2m-2}[|\hat{P}_{k+1}-P_{k}|^{2}] \leq Kh^{2}%
\mathbb{E}|P_{k}|^{2m}+\frac{\gamma h}{2}\mathbb{E}|P_{k}|^{2m} + Kh,
\end{eqnarray*}%
and similarly with some new $K>0$%
\begin{equation*}
K\mathbb{E}|P_{k}|^{2m-l}|\hat{P}_{k+1}-P_{k}|^{l}\leq \frac{\gamma h^{l/2}}{%
2m}\mathbb{E}|P_{k}|^{2m}+Kh^{l/2}+Kh^{l/2+m}|P_{k}|^{2m},\;l\geq 3.
\end{equation*}%
Thus, with a new $K>0$ again:
\begin{equation*}
\mathbb{E}|P_{k+1}|^{2m}\leq \mathbb{E}|P_{k}|^{2m}-\frac{\gamma h}{2}(1-%
\sqrt{h}-Kh)\mathbb{E}|P_{k}|^{2m}+Kh
\end{equation*}%
and, using the Gronwall lemma, we get 
\begin{equation*}
\mathbb{E}|P_{k}|^{2m}\leq \left( 1-\frac{\gamma h}{2}(1-\sqrt{h}-Kh)\right)
^{k}|P_{0}|^{2m}+\frac{2K}{\gamma (1-\sqrt{h}-Kh)}\left[ 1-\left( 1-\frac{\gamma h%
}{2}(1-\sqrt{h}-Kh)\right) ^{k}\right] .
\end{equation*}
\end{proof}

\begin{lemma}\label{lemma5.8}
Under Assumptions~\ref{cld_as:1}, \ref{cld_as:3}, \ref{as:5}-\ref{as:7}, the one-step error of scheme [PA$_c$] applied to (\ref{ecld}) is estimated as 
\begin{align*}
    &\big|\mathbb{E}[u(t_{k+1},Q_{k+1},P_{k+1}) - u(t_{k},Q_{k},P_{k})\mid Q_{k},P_{k}]\big| \leq C\big(h^{2}(1+|P_{k}|^{2m})e^{-\lambda(T - t_{k})} \\ & \;\;\;\; +  he^{-\lambda (T - t_k)}\mathbb{E}\big[\sum_{i=1}^{\kappa_k}(n_{G}(Q_{k+1,c,i})\cdot \tilde{P}_{k+1,i-1})(1+|P_{k+1}|^{2m})I_{G^{c}}(\hat{Q}_{k+1})\big|Q_{k},P_{k}\big]\big),
\end{align*}
where $\tilde{P}_{k+1, 0} = \hat{P}_{k+1}$, $\kappa_k = \kappa (Q_k, \hat{P}_{k+1}, h)$,  $u(t,q,p)$ is the solution of (\ref{eq2.1})-(\ref{eq2.3}) with $b(q,p) = - \nabla_q U(q) - \gamma p$ and $C,\lambda>0$  are constants independent of $h$, $T$ and $P_{k}$.
\end{lemma}
The proof follows exactly the same steps as in the proof of Lemma~\ref{lemma4.4} with the modification that Assumption~\ref{as:7} is used for derivatives' decay to bound the terms in the Taylor expansion.

\begin{lemma} \label{ebl}
Let  Assumptions \ref{cld_as:1} and \ref{as:5} hold. For sufficiently small $h>0$ the following inequality holds for any integer $m \geq 1$:
\begin{equation}
    \mathbb{E}\bigg(\sum\limits_{k=0}^{N-1}\sum_{i=1}^{\kappa_k}e^{\lambda t_k}(n_{G}(Q_{k+1, c, i})\cdot \tilde{P}_{k+1, i-1})(1+|P_{k+1}|^{2m})I_{G^{c}}\big(\hat{Q}_{k+1}\big) \bigg)\leq Ce^{\lambda T}, \label{eq_ebl}
\end{equation}
where $\kappa_k := \kappa (Q_k, \hat{P}_{k+1}, h)$  and ${Q}_{k+1, c, i}$, ${Q}_{k+1}$, $\hat{P}_{k+1}$, 
$ \hat{Q}_{k+1} $, $k=0, \dots, N-1$, are according to the  Markov chain constructed by scheme [PA$_c$] with multi-collision  Algorithm~\ref{A$_c$ step with multiple},  and $C>0$ is a constant independent of $T$ and $h$.  
\end{lemma}

\begin{proof}
If in (\ref{cld_eq:5.1})  $g(t_k,q,p)= \mathbb{E}\big[\sum_{i=1}^{\kappa(q,\hat{P}_1,h)}e^{\lambda t_{k}}(n_{G}(Q_{k+1,c,i})\cdot \tilde{P}_{k+1,i-1})(1+|P_{1}|^{2m})I_{G^{c}}(\hat{Q}_{1})|\;Q_k = q, P_k = p\big]$ then the solution of (\ref{cld_eq:5.1})-(\ref{cld_eq:5.2}) is 
\begin{align*}
   v&(t_{0},q,p) =  
   \mathbb{E}\bigg(\sum_{k=0}^{N-1} g(t_{k},Q_{k}, P_{k}) \big| Q_{0} = q , P_{0} = p\bigg)
   \\ & =\mathbb{E}\bigg(\sum\limits_{k=0}^{N-1}\sum_{i=1}^{\kappa_k}\mathbb{E}\Big(e^{\lambda t_{k}}(n_{G}(Q_{k+1,c,i})\cdot \tilde{P}_{k+1, i-1})(1+|P_{k+1}|^{2m})I_{G^{c}}\big(\hat{Q}_{k+1}\big)\big| Q_{k},P_{k}\Big) \Big| Q_0 = q, P_0 = p \bigg)  
   \\ &=\mathbb{E}\bigg(\sum\limits_{k=0}^{N-1}\sum_{i=1}^{\kappa_k}e^{\lambda t_{k}}\Big(n_{G}(Q_{k+1,c, i})\cdot \tilde{P}_{k+1,i-1}\Big)\big(1 + |P_{k+1}|^{2m}\big)I_{G^{c}}\big(\hat{Q}_{k+1}\big)\Big| Q_{0} = q, P_{0} = p \bigg),
\end{align*}
where $\kappa_k = \kappa (Q_k, \hat{P}_{k+1}, h)$. 
If we can find a function $V(t,q,p)$ with $g(t,q,p)$, satisfying (\ref{cld_eq:5.1})-(\ref{cld_eq:5.2}), such that 
\begin{equation*}
g(t,q,p) \geq \mathbb{E}\big[\sum_{i=1}^{\kappa(q,\hat{P}_1,h)}e^{\lambda t}(n_{G}(Q_{c,i})\cdot \tilde{P}_{i-1})(1+|P_{1}|^{2m})I_{G^{c}}(\hat{Q}_{1})|\; q, p\big],
\end{equation*}
then the inequality $v(t_{0},Q_{0},P_{0}) \leq V(t_{0},Q_{0},P_{0})$ holds. 

As in the proof of Lemma~\ref{bl},  we consider the boundary zone $G_{R_{0}}$. Introduce the function 
\[
V(t,q,p)=%
\begin{cases}
0, & (t,q,p)\in \{T\}\times \bar{G}\times \mathbb{R}^{d}, \\ 
Ae^{\lambda t}|p|^{2\ell }+K (e^{\lambda T}-e^{\lambda t}) &  \\ 
+(B(q)\cdot p)(1+|p|^{2m})e^{\lambda t}, & (t,q,p)\in \lbrack 0,T-h]\times 
\bar{G}\times \mathbb{R}^{d},%
\end{cases}%
\]%
where $A=\max_{q\in \bar{G}}|B(q)|$; integer $\ell \geq m+1$ will be chosen
later in the proof; $K>0$ is a constant, which value will also be chosen
later in the proof; and $B(q)$ is the same as in the proof of Lemma~\ref{bl}
(i.e., $B(q)=n(q^{\pi })$ for $q\in G_{R_{0}}$ and $q^{\pi }$ is projection
of $q$ on $\partial G$). Note that $V(t,q,p)\geq 0$.

Using Taylor's expansion, we get for any $l\geq 1$%
\begin{align*}
\mathbb{E}|P_{1}|^{2l} &=\mathbb{E}|\tilde{P}_{0}|^{2l} = \mathbb{E}|\hat{P}_{1}|^{2l}=|p|^{2l}-2lh\gamma
(p\cdot p)|p|^{2l-2} - 2lh(\nabla_q U(q) \cdot p)|p|^{2l-2}\\  & +\frac{2\gamma }{\beta }lh(2(l-1)+d)|p|^{2l-2}+r_{1},
\end{align*}%
$\allowbreak $where $r_{1}$ satisfies 
\begin{equation}
|r_{1}|\leq Cl^{4}h^{2}(1+|p|^{2l})  \label{eq:lem6r1}
\end{equation}%
with $C>0$ being independent of $h,$ $p$ and $l$. Expanding $%
e^{\lambda (t+h)}$, using the generalized Young inequality twice and then simplifying the expression, we obtain 
\begin{align}
&\mathbb{E}e^{\lambda (t+h)}|P_{1}|^{2\ell }-e^{\lambda t}|p|^{2\ell
} \label{eq:lem6_1}\\
&\leq e^{\lambda t}\left[ -\ell h\gamma |p|^{2\ell }+\ell h\gamma \left[ \frac{%
2\ell }{\beta (\ell -1)}(2(\ell -1)+d)\right] ^{\ell }+(\lambda+1) h|p|^{2\ell }%
+h (2 \ell |\nabla_q U(q)|)^{2l}
\right] +r_{2}, \nonumber 
\end{align}%
where $r_{2}$ satisfies (\ref{eq:lem6r1}) with $l=\ell $. 

From (\ref{cld_eqn_4.23}), we have 
\begin{align}
\mathbb{E}[(B(Q_{1})\cdot P_{1})(1+|P_{1}|^{2m})] & =(B(q)\cdot p)(1+|p|^{2m}) \nonumber \\    &  
- 2\mathbb{E} \sum_{i=1}^{\kappa}\big[ (\tilde{P}_{i-1} \cdot n_{G}(Q_{c, i}))(1 + |P_{1}|^{2m})I_{G^{c}}(\hat{Q}_{1}) \big] +r_{3},  \label{eq:lem6_2}
\end{align}%
where $r_{3}$ satisfies 
\begin{equation}
|r_{3}|\leq Ch(1+|p|^{2m+2})  \label{eq:lem6r2}
\end{equation}%
with $C>0$ being independent of $h$ and $p$ (it depends on $m)$.

From (\ref{eq:lem6_1}) and (\ref{eq:lem6_2}) and using $e^{\lambda
(t+h)}\geq e^{\lambda t}+\lambda he^{\lambda t},$ it follows that 
\begin{eqnarray*}
&&\mathbb{T}V-V=\mathbb{E}Ae^{\lambda (t+h)}|P_{1}|^{2\ell }+K(e^{\lambda T}-e^{\lambda
(t+h)})+\mathbb{E}[(B(Q_{1})\cdot P_{1})(1+|P_{1}|^{2m})]e^{\lambda (t+h)} \\
&&-Ae^{\lambda t}|p|^{2\ell }-K(e^{\lambda T}-e^{\lambda t})-(B(q)\cdot
p)(1+|p|^{2m})e^{\lambda t} \\
&\leq &Ae^{\lambda t}\left[ -\ell h\gamma |p|^{2\ell }+\ell h\gamma \left[ \frac{%
2\ell }{\beta (\ell -1)}(2(\ell -1)+d)\right] ^{\ell }+(\lambda +1) h|p|^{2\ell }%
+h (2 \ell |\nabla_q U(q)|)^{2l}
\right] \\
&&-Ke^{\lambda t}\lambda h - 2e^{\lambda t}\mathbb{E}\big[ \sum_{i=1}^{\kappa} (\tilde{P}_{i-1} \cdot n_{G}(Q_{c, i}))(1 + |P_{1}|^{2m})I_{G^{c}}(\hat{Q}_{1}) \big]+e^{\lambda t}(r_{2}+r_{4})(1 + r_5),
\end{eqnarray*}%
where $r_{4}$ satisfies (\ref{eq:lem6r2}) and $r_5 \leq C h^2 $ with $C>0$ being independent of $m$, $\ell$ and $K$. 
Choosing a sufficiently 
large $\ell $ and, depending on the choice of $\ell ,$ choosing a sufficiently small $h>0$ and large $K$, we get $g(t,q,p)=-(\mathbb{T}V-V)\geq \mathbb{E}\big[\sum_{i=1}^{\kappa}%
e^{\lambda t}(n_{G}(Q_{c,i})\cdot \tilde{P}_{i-1})(1+|P_{1}|^{2m})I_{G^{c}}(\hat{Q}%
_{1})\big]$.
Hence, the lemma is proved.

\end{proof}

\begin{proof}[\textbf{Proof of Theorem~\ref{theorem4.2}}] Using Assumption~\ref{as:7} and Lemma~\ref{lemma5.8}, we have 
\begin{align*}
    &|\mathbb{E}\varphi(Q_{N},P_{N}) - \bar{\varphi}| \leq |\mathbb{E}(\varphi(Q(T),P(T))) - \bar{\varphi}| +  |\mathbb{E}(\varphi(Q_{N},P_{N})) - \mathbb{E}(\varphi(Q(T),P(T)))| \\
     &
   \leq C\exp(-\lambda T)+ \Big|\mathbb{E}\Big(\sum\limits_{k=0}^{N-1}\mathbb{E}\big(u(t_{k+1},Q_{k+1},P_{k+1}) -u(t_{k},Q_{k},P_{k}) | Q_{k},P_{k}\big)\Big)\Big|   \\ 
   & \leq C\exp(-\lambda T)+ Ch^{2}\mathbb{E}(|P_{k}|^{2m})\sum\limits_{k=0}^{N-1}e^{-\lambda (T - t_{k})} \\
   & + Che^{-\lambda T}\mathbb{E}\Big(\sum\limits_{k=0}^{N-1}\sum_{i=1}^{\kappa_k}e^{\lambda t_{k}}(n_{G}({Q}_{k+1,c,i})\cdot\hat{P}_{k+1,i-1})(1+|P_{k+1}|^{2m})I_{G^{c}}(\hat{Q}_{k+1}) \Big),
\end{align*}
which together with Lemmas~\ref{cld_lemma_5.6} and \ref{ebl} gives the desired result.
\end{proof}

\subsection{Order of consistency of [OBA$_{c}$BO]}\label{sec:2ndorder}

As mentioned previously, splitting techniques, presented in Subsection~\ref{sec:2nd},  composed of collisional drift, gradient forces, and stochastic impulses achieve second-order convergence. From numerical analysis experience it is expected that, due to the additional noise term, stochastic schemes generally would maintain equivalent or reduced convergence rates compared to their deterministic counterparts. Consequently, splitting schemes with collision events would typically be expected to achieve first-order convergence at best. This section is devoted to investigation of the above mentioned result of second-order convergence.

\subsubsection{Preliminary discussion }\label{prlim_dis_sec}

Let us first look at the error expansion of one-step of [BA$_c$B], i.e. the vanilla Verlet method having elastic collisions  with the boundary $\partial G$. 
For simplicity, we consider $G = (0, \infty)$ with boundary $\partial G = \{0\}$. Setting $\gamma = 0$ in equation \eqref{ecld} gives us the deterministic collisional dynamics. Let the dynamics start at $t = 0$ with initial condition $(q,p) \in G \times \mathbb{R}  \setminus \{0\}$ and evolve for time $h$, with a collision occurring at time $\tau $. Using Taylor's theorem, we obtain:
\begin{align}
    Q(h) &= q + (2\tau - h)p - \frac{1}{2}(h - 2\tau)^2 \frac{\partial U}{\partial q}(q) + \mathcal{O}(h^3), \label{cld_neweq_4.44}\\
    P(h) &= - p + (2 \tau - h) \frac{\partial U}{\partial q}(q) + \frac{1}{2} (h -2 \tau)^2 \frac{\partial^2 U}{\partial q^2}(q)p + \mathcal{O}(h^3). \label{cld_neweq_4.45}
\end{align}
Applying  Taylor's theorem to the one-step [BA$_c$B] update $(Q_1, P_1)$ from the initial conditions $(q,p)$, we get 
\begin{align}
    Q_1 &= q + (2\tau - h) p -  (2 \tau - h)\frac{h}{2}\frac{\partial U}{\partial q}(q), \label{cld_neweq_4.46}\\
P_1 & =  -p - \frac{h}{2}(-h + 2 \tau) \frac{\partial^2 U}{\partial q^2}(q) p + \mathcal{O}(h^3). \label{cld_neweq_4.47}
\end{align}
It is clear from \eqref{cld_neweq_4.44}-\eqref{cld_neweq_4.47} that the error of the Verlet method [BA$_c$B] in the collisional step is of order $\mathcal{O}(h)$:
\begin{align}\label{eq:ham3}
    |Q_1 - Q(h)| + |P_1 - P(h)| = \mathcal{O}(h).
\end{align}
At the same time, we see from \eqref{cld_neweq_4.44}-\eqref{cld_neweq_4.47} that if the collision occurs in the middle of time-step, 
\begin{align} \label{cld_new_eq_4.49}
    \tau = h/2,
\end{align}
then the optimal, i.e. $\mathcal{O}(h^2)$, local order in the collisional step would be observed
(see also \cite{BrianBen2000,leimkuhler_mathews_15}). If there are a finite number of collisions (i.e., independent of $h$) and they happen at the middle of the step as above, then the global convergence of the collisional Verlet scheme would be optimal, i.e. $\mathcal{O}(h^2)$. The scenario of middle step collisions mentioned above  cannot be expected in the deterministic setting (see detailed discussion in  \cite[pp. 133-134]{leimkuhler_mathews_15}). To design a second order method, it is necessary to solve algebraic equations \cite{BrianBen2000,leimkuhler_mathews_15} to determine the time of collision in the numerical discretization.

The above discussion has demonstrated that $\tau$ has significant impact on the convergence order of collisional schemes. 

For further intuition, let us computationally analyze statistics of $\tau$ corresponding to the first collision of the discrete dynamics of [OBA$_c$BO] on a particular example.  
We consider the Langevin dynamics \eqref{ecld} evolving in the domain $G = (0, \infty)$ under the quadratic potential $U(q) = q^2/2$. We apply [OBA$_c$BO] scheme with $h = 0.01$ and focus on the first-passage event when the discrete time dynamics reach the boundary $\{0\}$. During the specular reflection (A$_c$) step of duration $h = 0.01$, we denote the collision time as $\tau_1$, where the subscript highlights the fact that it is related to first passage time. 
In Figure~\ref{cld_figure_firsttau_1_histogram}, the histogram of collision times $\tau_1$ illustrates the empirical distribution, based on $10^6$ trajectories of [OBA$_c$BO]. 

\begin{figure}[htbp]
    \centering
    \includegraphics[width=1\linewidth, height=0.15\textheight]{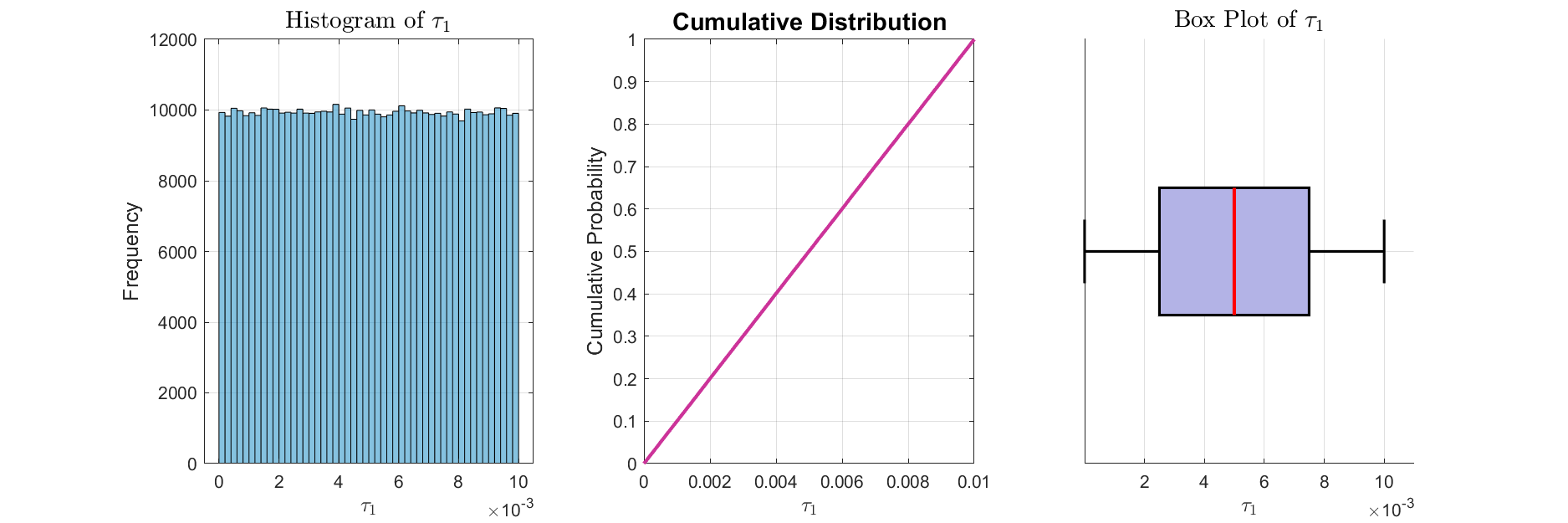}
    \caption{Plots showing the distribution of $\tau_1$. }
    \label{cld_figure_firsttau_1_histogram}
\end{figure}

From the error expansions in \eqref{cld_neweq_4.44}-\eqref{cld_neweq_4.47}, two statistics related to $\tau_1$ are of particular significance:
\begin{align*}
    \Lambda_1 = \mathbb{E}(\tau_1) - h/2 \,\, \text{ and } \,\,
    \Lambda_2 = \mathbb{E}(\tau_1/h)^2.
\end{align*}
We denote the corresponding Monte Carlo estimators as $\hat{\Lambda}_1$ and $\hat{\Lambda}_2$. Our numerical experiments, for the quadratic potential case, show that $\hat{\Lambda}_1$ decreases with order $\mathcal{O}(h^2)$ as $h $ decreases (see the left plot in Fig.~\ref{cld_figure_first_tau_1}), while $\hat{\Lambda}_2$ remains approximately constant as $h$ decreases (see the right plot in Fig.~\ref{cld_figure_first_tau_1}). 
Then one can infer that the second-order convergence of [OBA$_c$BO] scheme is a consequence of the fact that on average the discrete dynamics collides with boundary at the middle of time-step $h$ with second-order error, i.e., 
\begin{align*}
\mathbb{E} \tau = h/2 + \mathcal{O}(h^2). %
\end{align*}
We observed the same phenomenon in our experiments presented in Section~\ref{sec:tests} for different symmetric compositions of [A$_c$, B, O] splitting. 

\begin{figure}[htbp]
    \centering
    \includegraphics[width=1\linewidth, height=0.20\textheight]{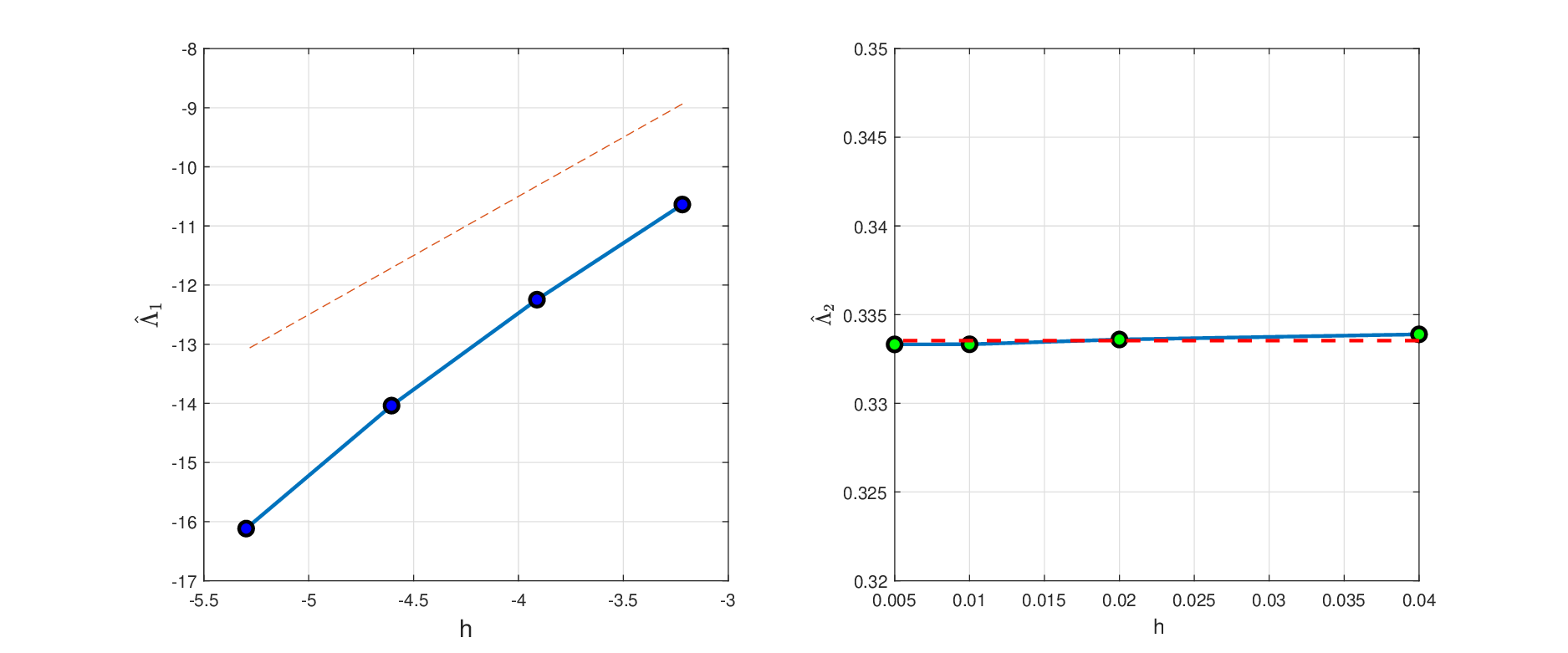}
    \caption{Left plot shows second-order in $h$ decrease of the statistics $\hat{\Lambda}_1$ (blue line); red line gives the second order slope. Right plot compares the statistics $\hat{\Lambda}_2$ (green line) vs the constant dashed red line, i.e. it shows that $\hat{\Lambda}_2$ remains almost constant.}  
    \label{cld_figure_first_tau_1}
\end{figure}

These experimental observations about $\tau_1$ can be explained as follows. Make the natural assumption that $\tau_1 $ has a normalized density $p(t)$
given by
\[
p(t)=\frac{\rho(t)}{\int_0^h \rho(s) ds}, \,\, t \in [ 0,h],
\]
with $\rho(s) >0$, $ s \in [ 0,h]$, being continuously differentiable.
Then, by Taylor's theorem, we have 
\begin{align}
 \mathbb{E} \tau  = \frac{\int_0^h t \rho (t)dt}{\int_0^h \rho (t) dt} = \frac{h}{2} + \mathcal{O}(h^2). 
 \label{cld_new_eq_4.52}
\end{align}
Note that the above expansion is true for any positive differentiable $p(t)$ -- roughly speaking, the main idea of the one-step lemma for [OBA$_c$BO] of the next subsection rests on this fact. 

We also expect that if we randomize the initial condition of otherwise deterministic Hamiltonian dynamics with elastic collisions, we should see the same effect as for Langevin equations: geometric integrators such as [BA$_c$B] are of 2nd order.
Indeed, our numerical illustration in Appendix~\ref{sec:colisHD} indicates that this supposition is plausible.
We leave a study of convergence rates for geometric integrators in random or chaotic Hamiltonian dynamics to possible future work.

\subsubsection{One-step lemma}

The relation \eqref{cld_new_eq_4.52} is the essence of the next lemma for [OBA$_c$BO] which provides theoretical justification of second order-convergence of  [OBA$_c$BO] when $G$ is half-space. We will not pursue the extension of lemma for other splitting schemes (e.g. [BA$_c$OA$_c$B]),  as that would be expected to follow a similar line of reasoning, and we leave the one-step analysis in the case of curved domains for future work. 
Recall (see Section~\ref{sec:mult}) that in the case of half-space multi-collisions cannot occur for any of our schemes. 

Consider $G=(0,\infty )\times \mathbb{R}^{d-1}.$ In this case the one step of scheme
[OBA$_{c}$BO] from Section~\ref{sec:2nd} takes the form:%
\begin{eqnarray}
&&\mathcal{P}_{1} =P_0e^{-\gamma h/2}+\sqrt{\frac{2\gamma}{ \beta}(1-%
\mathrm{e}^{-\gamma h})}\xi, \,\,
\mathcal{P}_{2} =\mathcal{P}_{1}-\frac{h}{2}\nabla_q U(Q_0), \,\,
\mathcal{Q} =Q_0+h\mathcal{P}_{2}; \label{eq:halfspace}\\
&&\textbf{If }\mathcal{Q}^{1} <0,\ \ \tau =-\frac{Q_0^{1}}{\mathcal{P}_{2}^{1}}%
,\ \mathcal{P}_{3}^{1}=-\mathcal{P}_{2}^{1},\ \mathcal{Q}_{c}^{1}=0=Q_0^{1}-%
\tau \mathcal{P}_{3}^{1}=Q_0^{1}+\tau \mathcal{P}_{2}^{1},\  \\
&&\mathcal{Q}_{c}^{i} =Q_0^{i}+\tau \mathcal{P}_{3}^{i}=Q_0^{i}+\tau \mathcal{P}%
_{2}^{i}\ ,\ i\neq 1; \,\,
Q^{1} =(h-\tau )\mathcal{P}_{3}^{1}=\left( h+\frac{Q_0^{1}}{\mathcal{P}%
_{3}^{1}}\right) \mathcal{P}_{3}^{1},\  \\
&&Q^{i} =\mathcal{Q}_{c}^{i}\mathcal{+}(h-\tau )\mathcal{P}%
_{3}^{i}=Q_0^{i}+h\mathcal{P}_{2}^{i},\ i\neq 1; \\
&&\textbf{else }Q =\mathcal{Q},\ \mathcal{P}_{3}=\mathcal{P}_{2};\\
&&\mathcal{P}_{4} =\mathcal{P}_{3}-\frac{h}{2}\nabla_q U(Q), \;\;
P =\mathcal{P}_{4}e^{-\gamma h/2}+\sqrt{\frac{2\gamma}{\beta }(1-%
\mathrm{e}^{-\gamma h})}\zeta.
\end{eqnarray}%
Here the components $\xi^i$ of $\xi$ and $\zeta^i$ of $\zeta$ are i.i.d. with the standard normal distribution and $\mathcal{Q}^{i}$, $i= 1,\dots, d$, denotes $i-$th component of $\mathcal{Q}$ and likewise for $q$ and $\mathcal{P}_{j}$, $j= 2, 3$.

In this subsection we need the following assumptions. The potential $U(q)$
satisfies Assumption~\ref{as:5}, and $\nabla_q U(q)$ is globally Lipschitz and
derivatives of $\nabla_q U(q)$ are bounded. To ensure ergodicity (see \cite{MSH02} and \cite{a22}[Section 9.1.5]), we need dissipativity of the extended generator of the Markov process $(Q(t), P(t))$ which can be guaranteed by assuming that there exists an $\alpha_1 >0$ and $0<\alpha_2 <1$ such that %
\begin{equation}
\frac{1}{2}(\nabla_q U(q),q)\geq \alpha_2 U(q)+\gamma ^{2}\frac{\alpha_2 (2-\alpha_2 )}{8(1-\alpha_2)}|q|^2+ \alpha_1.  \label{eq:ergunb}
\end{equation}
Further, we assume that the function $\varphi (q,p)\in C^{6,6}(\bar{G}\times 
\mathbb{R}^{d})$ and all its existing derivatives have growth at most
polynomial in $q$ and $p$. As usual for proofs of higher order numerical methods, we also require  that the solution of the backward Kolmogorov equation (\ref%
{eq2.1})-(\ref{eq2.3}) with $b(q,p)=-\nabla_q U(q)-\gamma p$ is sufficiently smooth and, in particular, satisfies the
condition that there are positive constants $C$ and $\lambda $ 
independent of $T$ such that the following bound holds for some $%
m_{1},m_{2}\geq 1$ (cf. (\ref%
{eq:ergKolmest}) in Assumption~\ref{as:7}): 
\begin{equation*}
\sum\limits_{l=1}^{3}\sum\limits_{i+|j|=l}|D_{t}^{i}D_{q}^{j}u(t,q,p)|+\sum%
\limits_{l=1}^{6}\sum\limits_{|j|=l}|D_{p}^{j}u(t,q,p)|\leq
C(1+|q|^{2m_{1}}+|p|^{2m_{2}})e^{-\lambda (T-t)},
\end{equation*}%
where $C>0$ is independent of $T,$ $q$ and $p$.

By standard means it is not difficult to establish that under the stated conditions the moments of $Q_k$ and $P_k$ from scheme [OBA$_c$BO] are bounded uniformly in time (cf. Lemma~\ref{cld_lemma_5.6}). 

\begin{lemma} \label{lem:just2nd}
Assume that the above conditions hold and that $(Q_0,P_0)$ is random with a density so that all the required moments with respect to this density exist, then the one-step error of scheme [OBA$_c$BO] has
the following estimate 
\begin{equation}\label{eq:lem10}
|\mathbb E\left[ u(t+h,Q,P)-u(t,Q_0,P_0)\right] |\leq Ce^{-\lambda (T-t)}
h^{3} ,
\end{equation}
where $C>0$ is independent of $T,$ $t,$ and $h$.
\end{lemma}


Since the one-step error is ${\cal O}(h^3)$, the global weak error of  [OBA$_c$BO] scheme is ${\cal O}(h^2)$.
We remark that existence of a smooth density for [OBA$_c$BO] scheme is proved in Appendix~\ref{sec:density}.

\begin{proof}
We write the one step error as
\begin{eqnarray*}
&&R :=\E\left[ u(t+h,Q,P)-u(t,Q_0,P_0)\right]  \\
&&=\E\left[ u(t+h,Q,P)-u(t+h/2,Q,\mathcal{P}_{3})\right]  
+\E\left[ u(t+h/2,Q,\mathcal{P}_{3})-u(t+h/2,\mathcal{Q}_{c},\mathcal{P}%
_{3})\right] I_{G^{c}}(\mathcal{Q}) \\
&&+\E\left[ u(t+h/2,\mathcal{Q}_{c},\mathcal{P}_{2})-u(t+h/2,Q_0,\mathcal{P}%
_{2})\right] I_{G^{c}}(\mathcal{Q}) \\
&&+\E\left[ u(t+h/2,Q,\mathcal{P}_{2})-u(t+h/2,Q_0,\mathcal{P}_{2})\right]
I_{G}(\mathcal{Q}) 
+\E\left[ u(t+h/2,Q_0,\mathcal{P}_{2})-u(t+h/2,Q_0,\mathcal{P}_{1})\right]  \\
&&+\E\left[ u(t+h/2,Q_0,\mathcal{P}_{1})-u(t+h/2,Q_0,P_0)\right] +\left[
u(t+h/2,Q_0,P_0)-u(t,Q_0,P_0)\right] .
\end{eqnarray*}%
Appropriately applying Taylor's expansion to each of the brackets in the above re-writing of $R$, we arrive at
\begin{eqnarray*}
R &=&\E\left[ (\frac{h}{2}-\tau )\mathcal{P}_{3}\cdot \nabla _{q}u(t+h/2,\mathcal{Q}_{c},%
\mathcal{P}_{3})I_{G^{c}}(\mathcal{Q})\right] +\E\left[ (\tau
-\frac{h}{2})\mathcal{P}_{2}\cdot \nabla _{q}u(t+h/2,Q_0,\mathcal{P}%
_{2})I_{G^{c}}(\mathcal{Q})\right]  \\
&&+r_{1}+r_{2},
\end{eqnarray*}%
where $|r_{1}|\leq Ce^{-\lambda (T-t)}h^{3}$
and $|r_{2}|\leq Ce^{-\lambda (T-t)}h^{2} \E \left[ I_{G^{c}}(\mathcal{Q})%
\right]$. Since $\mathbb E\left[ I_{G^{c}}(\mathcal{Q})\right]={\cal O}(h)$, 
$|r_{2}|\leq Ce^{-\lambda (T-t)}h^{3}$.

Introduce $\pi$ so that $\pi \mathcal{P}_{3}=\mathcal{P}_{2},$ i.e., $\pi \mathcal{P}_{3}=(-%
\mathcal{P}_{3}^{1},\mathcal{P}_{3}^{2},\ldots ,\mathcal{P}_{3}^{d}).$
We have
\begin{eqnarray}
R &=&\E\left[ (\frac{h}{2}-\tau )\left( \mathcal{P}_{3}\cdot \nabla
_{q}u(t+h/2,\mathcal{Q}_{c},\mathcal{P}_{3})-\pi \mathcal{P}_{3}\cdot \nabla
_{q}u(t+h/2,\mathcal{Q}_{c},\pi \mathcal{P}_{3})\right) I_{G^{c}}(\mathcal{Q}%
)\right]  \label{eq:lem10_2} \\
&&+r_{1}+r_{3}, \notag
\end{eqnarray}
where $r_{3}$ satisfies the same estimate as $r_{2}$ above. 

Let us first consider the following component of (\ref{eq:lem10_2}).
Using the expression for $\tau ,$ we get 
\begin{eqnarray*}
R_{1} &=&\mathbb{E}\tau \mathcal{P}_{3}^{1}\left[ \frac{\partial }{\partial q^{1}}%
u(t+h/2,0,Q_0^{2},\ldots ,Q_0^{d},\mathcal{P}_{3})+\frac{\partial }{\partial
q^{1}}u(t+h/2,0,Q_0^{2},\ldots ,Q_0^{d},\pi \mathcal{P}_{3})\right] I_{G^{c}}(\mathcal{Q}) \\
&&=\mathbb{E}\left[Q_0^{1}v_{1}(Q_0^{2},\ldots ,Q_0^{d},\mathcal{P}_{3}) I_{G^{c}}(\mathcal{Q})\right],
\end{eqnarray*}%
where 
\begin{equation*}
v_{1}(q^{2},\ldots ,q^{d},p):=\frac{\partial }{\partial q^{1}}%
u(t+h/2,0,q^{2},\ldots ,q^{d},p)+\frac{\partial }{\partial q^{1}}%
u(t+h/2,0,q^{2},\ldots ,q^{d},\pi p).
\end{equation*}%
Let $\rho (q,p)$ be the joint density for  $Q_0,$ $\mathcal{P}_{3}.$ Assume
that all the required moments with respect to this density exist.
Observe that $\{Q \in G^{c}\}=\{Q_0^1-h{\cal P}^1_3<0\}$. Then
\begin{equation*}
R_{1}=\int_{\mathbb{R}^{2d-2}}\int_{0}^{\infty
}\int_{0}^{p^{1}h}q^{1}v_{1}(q^{2},\ldots ,q^{d},p)\rho (q,
p)dq^{1}dp^{1}dq^{2}\cdots dq^{d}dp^{2}\cdots dp^{d}.
\end{equation*}%
Note that%
\begin{equation*}
\int_{0}^{p^{1}h}q^{1}\rho (q, p)dq^{1}=\frac{1}{2}\left( p^{1}\right)
^{2}h^{2}\rho (0,q^{2},\ldots ,q^{d}, p)+{\cal O}(h^{3}).
\end{equation*}%
Hence 
\begin{equation*}
R_{1}=\frac{h^2}{2}\int_{\mathbb{R}^{2d-2}}\int_{0}^{\infty }\left( p^{1}\right)
^{2}\rho (p,0,q^{2},\ldots ,q^{d})U_{1}(p,q^{2},\ldots
,q^{d})dp^{1}dq^{2}\cdots dq^{d}dp^{2}\cdots dp^{d}+{\cal O}(h^3).
\end{equation*}%
Next consider the following component of (\ref{eq:lem10_2}):
\begin{eqnarray*}
R_{2} &=&\frac{h}{2}\mathbb{E}\bigg(\mathcal{P}_{3}^{1}\left[ \frac{\partial }{\partial
q^{1}}u(t+h/2,0,Q_0^{2},\ldots ,Q_0^{d},\mathcal{P}_{3})+\frac{\partial }{%
\partial q^{1}}u(t+h/2,0,Q_0^{2},\ldots ,Q_0^{d},\pi \mathcal{P}_{3})\right]I_{G^{c}}(\mathcal{Q})\bigg)  \\
&=&\frac{h}{2}\int_{\mathbb{R}^{2d-2}}\int_{0}^{\infty
}\int_{0}^{p^{1}h}p^{1}v_{1}(q^{2},\ldots ,q^{d},p)\rho
(q,p)dq^{1}dp^{1}dq^{2}\cdots dq^{d}dp^{2}\cdots dp^{d} \\
&=&\frac{h^2}{2}\int_{\mathbb{R}^{2d-2}}\int_{0}^{\infty }\left( p^{1}\right) ^{2}\rho
(0,q^{2},\ldots ,q^{d},p)v_{1}(q^{2},\ldots ,q^{d},p)dp^{1}dq^{2}\cdots
dq^{d}dp^{2}\cdots dp^{d}+{\cal O}(h^3).
\end{eqnarray*}%
Thus $|R_1-R_2| \leq Ch^3 e^{-\lambda (T-t)}$.
The terms for $i\neq 1$%
\begin{align*}
A_{i}&=\mathbb{E}\left( \tau -h/2\right) \mathcal{P}_{3}^{i} \\
& \times \left[ \frac{\partial }{%
\partial q^{i}}u(t+h/2,0,Q_0^{2},\ldots ,Q_0^{d},\mathcal{P}_{3})-\frac{\partial 
}{\partial q^{i}}u(t+h/2,0,Q_0^{2},\ldots ,Q_0^{d},\pi \mathcal{P}%
_{3})\right]  I_{G^{c}}(\mathcal{Q})
\end{align*}%
are analyzed analogously. 
\end{proof}

\section{Numerical experiments}\label{sec:tests}

In this section we experimentally confirm theoretical results for the numerical integrators from Section~\ref{cld_sec3} in both the finite-time and ergodic case. 

The following notation is used in this section. The Monte Carlo estimator for $\mathbb{E}(\varphi(Q_{N}, P_{N}))$ is $\check{\varphi}_{M} = \frac{1}{M}\sum_{k=1}^{M}\varphi(Q_{N}^{(k)}, P_{N}^{(k)})$, where $(Q_{N}^{(k)}, P_{N}^{(k)})$ are independent realizations of $(Q_{N}, P_{N})$; $\chi$ denotes the number of collisions along a trajectory and  $\check{\chi}_{M}$ is the Monte-Carlo estimator for the average number of collisions; and ${D}_{M}$ denotes the usual sample variance of Monte Carlo estimators.


\begin{experiment}\label{exper1}
In this experiment we consider finite-time weak convergence.
Consider Example~\ref{example2.2} with $G := \{ q^{2}_{1} + q^{2}_{2} < 4 \}]$ and $U(q_1, q_2) = -q_{1}^{2} - q_{2}^{2}$, $\beta = 1$, $\alpha = 0.25$ and the final time $T=4$. Hence, the solution is $u(0, q,p) = \exp(-|p|^{2}/2  + |q|^{2} - 2)$. The point at which $u(0, q,p)$ is evaluated in the experiments  is $(q_1, q_2) = (1,1)$ and $(p_1, p_2) = (-0.1,-0.1)$. To compute the global error $e = \check{\varphi}_{M} - u(0,(1,1),(-0.1,-0.1))$, we calculate the exact solution $u(0,(1,1),(-0.1,-0.1)) = 0.99005$ (5 d.p.).
We test the first-order scheme [PA$_c$] and two second-order schemes [OBA$_c$BO] and [BA$_c$OA$_c$B]. 
Figure~\ref{fig:ft} confirms second-order accuracy of [OBA$_c$BO] and [BA$_c$OA$_c$B]. The first-order method [PA$_c$] had slightly higher accuracy than first order but clearly lower than second order. 
We also confirmed in this experiment that the average number of collisions $\check{\chi}_{M}$ does not depend on step size $h$: it was approximately $3.7$ for all the methods and time steps. 
In this experiment the $O$ step in the [OBA$_c$BO] and [BA$_c$OA$_c$B] schemes is
\[
O(p;h,\xi,\alpha):=p e^{ \alpha h} +\sqrt{\frac{\sigma ^{2}}{2\alpha }(\mathrm{e}^{2\alpha h} - 1)}\xi .
\]

\begin{figure}[H]
  \centering
   
        \includegraphics[width=7cm]{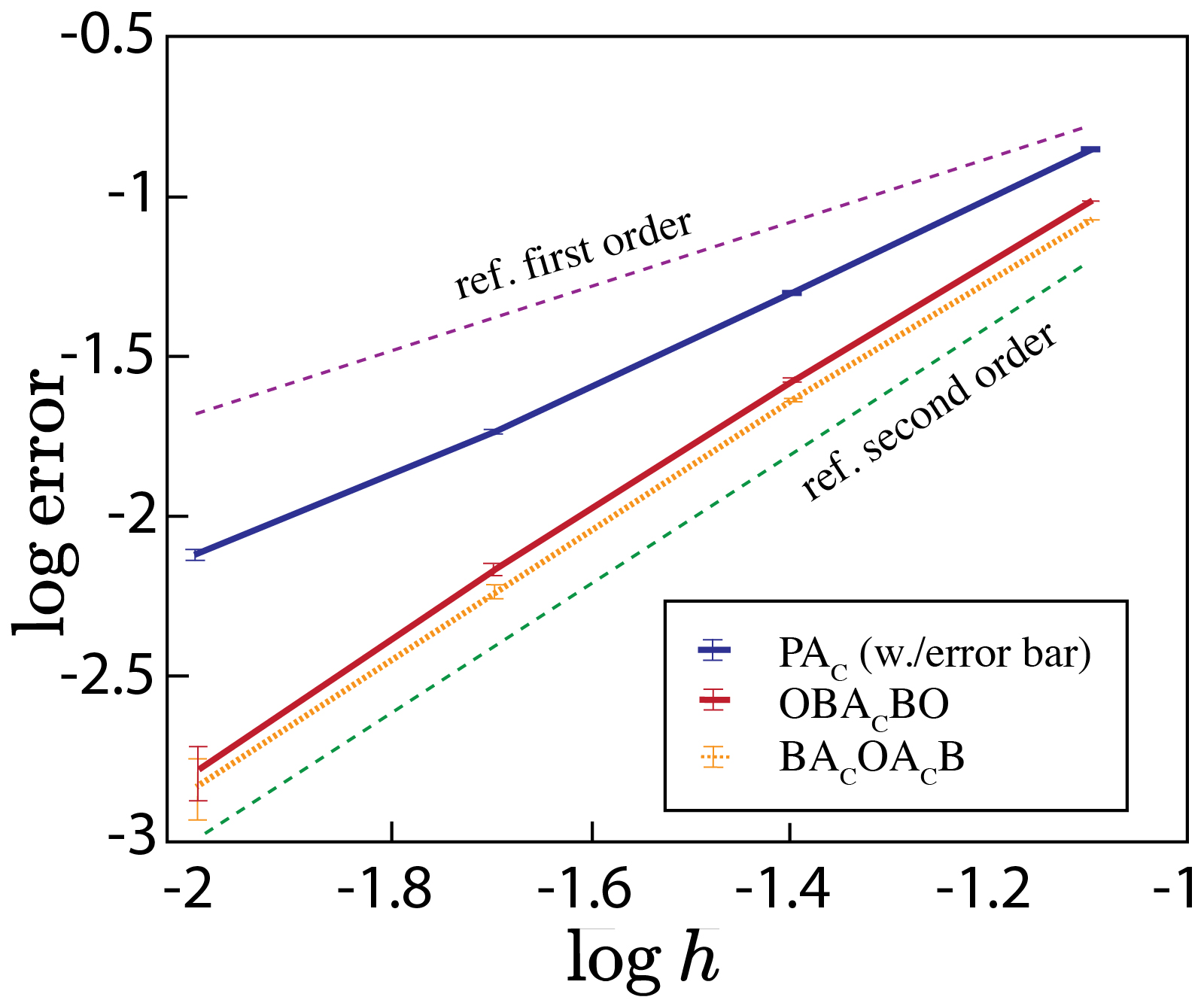}
 \captionof{figure}{Experiment~\ref{exper1}. The global error $e$ for the first-order method [PA$_c$] and for [OBA$_c$BO] and [BA$_c$OA$_c$B]. Error bars indicate the Monte Carlo error.\label{fig:ft}}

\end{figure}

 \end{experiment}


\begin{experiment}\label{exper2}
In this experiment we test convergence of second-order integators [A$_c$BOBA$_c$], [A$_c$OBOA$_c$], [BA$_c$OA$_c$B], [BOA$_c$OB], [OA$_c$BA$_c$O] and [OBA$_c$BO] to ergodic limits.

Take $G := (1,\infty) $ and $\partial G := \{1\} $.  Consider (\ref{ecld}) with the potential
\begin{align*}
    U(q) = q^{2}/2
\end{align*}
and $ \beta = 1$, $\gamma = 1$.  
We note that this $U(q)$ does not satisfy Assumption~\ref{as:5} in its part that we required $U(q)>1$.  
At the same time, the SDE stays the same if we consider $U(q) = q^2/2 + 1$ instead, which does satisfy the assumption. This change will only affect 
the normalization constant of the invariant densisty and hence the value we are computing but not anything else.
See also the comment after Assumption~\ref{as:5}.

For numerical simulation, we choose $T = 20$ as a sufficiently large time to achieve accurate values of the ergodic limit of $\varphi(q) = q^{2}/2$ with respect to this choice of (\ref{ecld}). The exact value of $\bar{\varphi} $ is $1.2626$ (4 d.p.). The starting point of Markov chain is $Q_{0} = 2$ and $P_0 = -0.1$.  We denote $e:=|\check{\varphi}_{M} - \bar{\varphi}|$.
The results are presented in 
Fig.~\ref{fig:erg}, which clearly shows second order convergence for all the integrators. As expected \cite{BenM13,leimkuhler_mathews_15},   [BA$_c$OA$_c$B] and [OA$_c$BA$_c$O] are the most accurate schemes. 

\begin{figure}[H]
\centering
\hspace{0.6cm}
        \includegraphics[width=10cm]{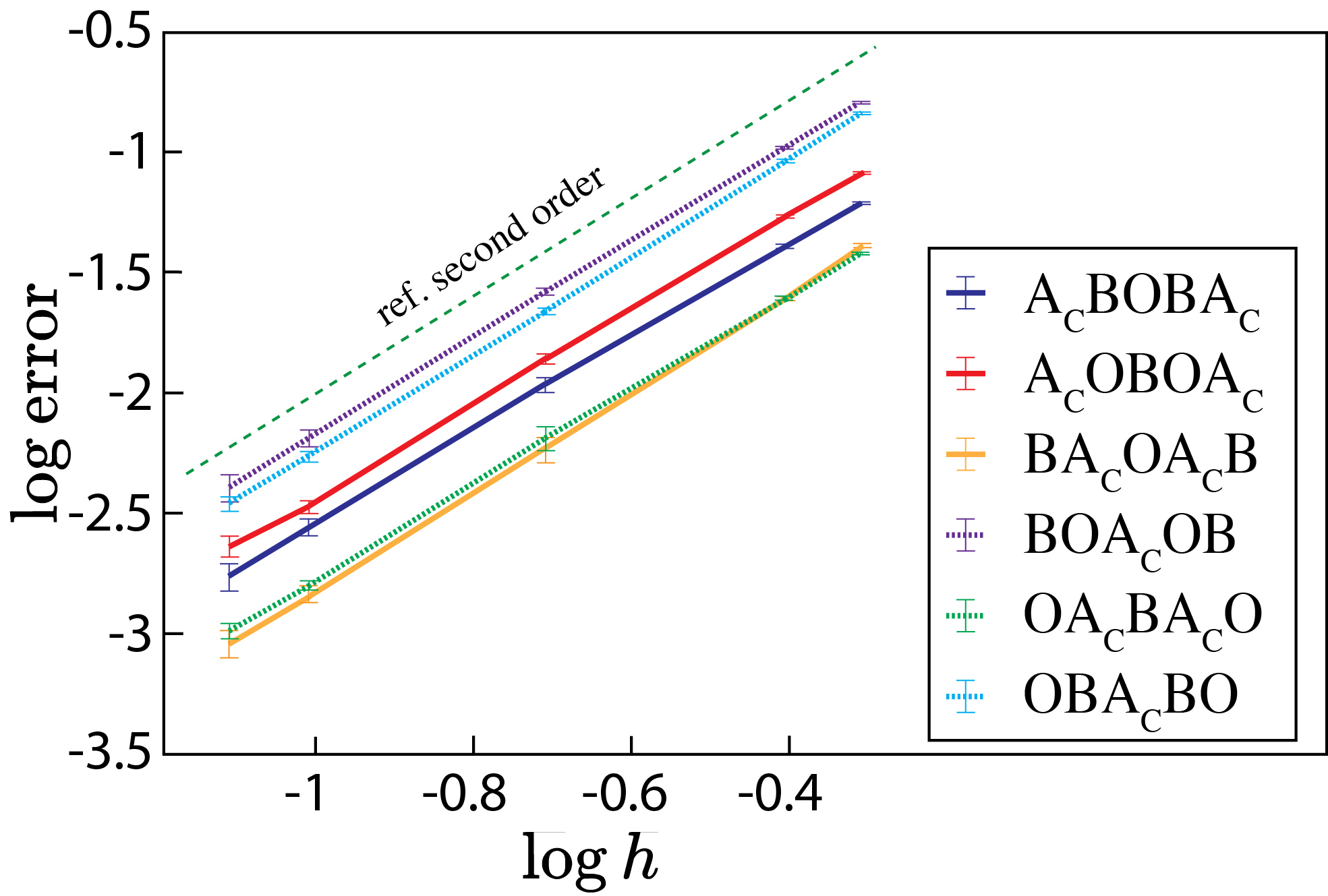}
 \captionof{figure}{Experiment \ref{exper2}. Dependence of the global error $e$ on $h$ for [A$_c$BOBA$_c$], [A$_c$OBOA$_c$], [BA$_c$OA$_c$B], [BOA$_c$OB], [OA$_c$BA$_c$O] and [OBA$_c$BO]. Error bars indicate the Monte Carlo error.\label{fig:erg}}
\end{figure}

\end{experiment}


\begin{experiment}\label{exper3}
In this experiment we test convergence of second-order integrators  [BA$_c$OA$_c$B] and [OBA$_c$BO] to ergodic limits.
Take G:= $\{ q_{1}^{2} + q_{2}^{2} < 4\}$. Consider (\ref{ecld}) with the potential
\[
U(q_{1},q_{2}) = \frac{1}{2}(q_{1} - q_{2})^{2} + \frac{q_{1}^{2}(q_{1}^{2}-12)}{12} + \frac{q_{2}^{2}(q_{2}^{2}-24)}{12}
\]
and $\beta = 1$, $\gamma = 4$. We choose 
\[
\varphi(q_{1},q_{2}) = U(q_1,q_2).
\]
We note that the potential $U$ is not symmetric around zero and that its local minima lie outside $G$. 
We are interested in approximating $\bar{\varphi}$. We take $T = 12$. We calculate the integral $\bar{\varphi}$ using integral2 function in Matlab with tolerance $10^{-10}$ which gives the reference value $-4.18006$ (5 d.p.). The starting point of the Markov chain is $(q_1, q_2) = (1,1)$ with $(p_1, p_2) = (-0.1, -0.1)$. We denote $e:=|\check{\varphi}_{M} - \bar{\varphi}|$. The results are presented in 
Fig.~\ref{fig:erg21}, which clearly shows second-order convergence for both integrators with [BA$_c$OA$_c$B] being more accurate.

\begin{figure}[H]
 \centering
\hspace{0.6cm}
        \includegraphics[width=7cm]{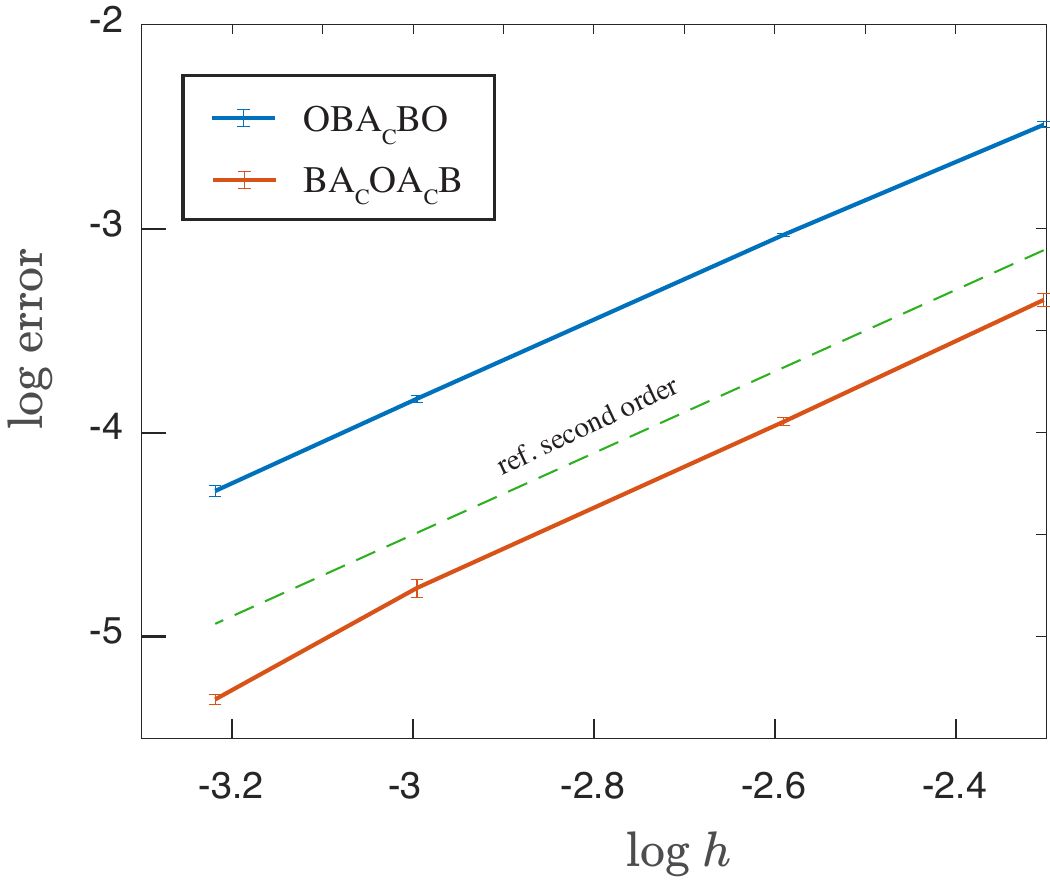}
 \captionof{figure}{Experiment \ref{exper3}. Dependence of the global error $e$ on $h$ for [OBA$_c$BO] and [BA$_c$OA$_c$B]. Error bars indicate the Monte Carlo error.\label{fig:erg21}}

\end{figure}

\end{experiment}

\begin{experiment}\label{exper4}
We take $G = \{ q_{1}^{2} + q_{2}^{2} < 4\}$ and $ \partial G = \{ q_{1}^{2} + q_{2}^{2} = 4\}$. We sample from distribution given by 
\begin{align}
\rho(q,p) = \frac{1}{\mathbb{Z}}e^{K |q|^{2} - |p|^{2} },\;\;\;\; q = (q_{1},q_{2}) \in \bar{G},\; p = (p_{1}, p_{2}) \in \mathbb{R}^{d}, 
\end{align}
and compute 
\begin{align}
\bar{\varphi} = \int_{G} (q_{1}^{2} + q_{2}^{2})e^{K (q_{1}^{2} + q_{2}^{2})}dq_{1} dq_{2} \Big/ \int_{G} e^{K (q_{1}^{2} + q_{2}^{2})}dq_{1} dq_{2}
\end{align}
via its Monte Carlo estimator based on two schemes [OBA$_{c}$BO] and [BA$_c$OA$_c$B] which are used to discretize (\ref{ecld}). We take $K = 5$ resulting in $\bar{\varphi} = 3.8000 $ (5.d.p.). This choice of $K$ is made in order to increase the number of collisions with the boundary and the aim is to check if in this case the two schemes  [OBA$_{c}$BO] and [BA$_c$OA$_c$B] maintain second order of convergence asymptotically as the  discretization step $h$ decreases. We take the initial conditions  $ (q_{1}, q_{2}) = (0,0)$ and $(p_{1}, p_{2}) = (-0.1, -0.1)$.

It can be seen from 
Fig.~\ref{fig:erg5.4} that both numerical schemes [OBA$_c$BO] and [BA$_c$OA$_c$B] have first roughly $\mathcal{O}(h)$ weak error then roughly $\mathcal{O}(h^{3/2})$ error and, as $h$ decreases, we clearly observe $\mathcal{O}(h^{2})$ convergence.  

\begin{figure}[H]
 \centering
\hspace{0.6cm}
   
        \includegraphics[width=10cm]{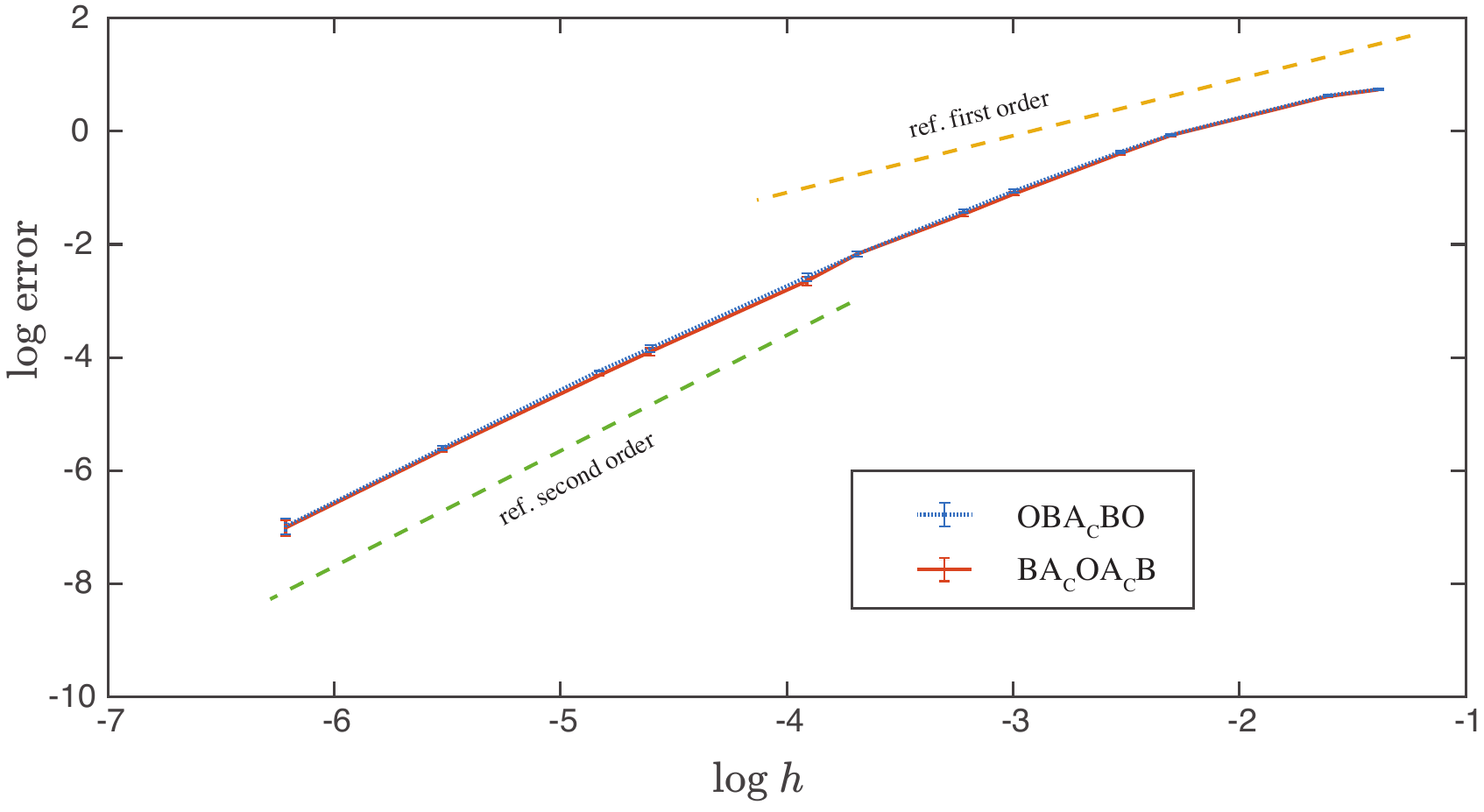}
 \captionof{figure}{Experiment \ref{exper4}. Dependence of the global error $e$ on $h$ for [OBA$_c$BO] and [BA$_c$OA$_c$B]. Error bars indicate the Monte Carlo error.\label{fig:erg5.4}}
\end{figure}

\end{experiment}

\begin{experiment}\label{exp_Neal_funnel} (\textbf{Truncated funnel})  
We consider Neal's funnel \cite{neal2001annealed} in this experiment which  is a hierarchical model defined as:
\begin{align}
 \theta &\sim \mathcal{N}(0, 3^2) \\
x_i \mid \theta &\sim \mathcal{N}(0, e^{\theta/2}),\quad i=1,\dots, 8.
\end{align}
We truncate $\theta$ to $[-3,1]$. Therefore, $G$ in this experiment is $(-3, 1)\times \mathbb{R}^{8}$.   The potential function is $U(\theta, \mathbf{x}) = \frac{\theta^2}{18} + 4\theta + \frac{1}{2} e^{-\theta} \sum_{i=1}^8 x_i^2$. We aim to estimate $\mathbb{E}[U]$, i.e. expectation with respect to Gibbs measure, using discretizations of \eqref{cld} and compare with exact value $0.6670 $ (4 d.p.). 

We test [OBA$_c$BO] and [BA$_c$OA$_c$B] schemes  in this experiment. Trajectory of [OBA$_c$BO] for  the first two dimensions can be seen in Fig.~\ref{neal_exp_traj}. We observe a second-order convergence of [OBA$_c$BO] scheme  (see Figure~\ref{Nealfunnel_OBABO}) and the very high accuracy of [BA$_c$OA$_c$B] (see Figure~\ref{exact_BAOAB}), the latter is quite remarkable and consistent with observations of higher accuracy of [BAOAB] in comparison with other 2nd order concatenations in the case of the whole Euclidean space \cite{leimkuhler_mathews_15}.    

\begin{figure}[H]
    \centering
    \includegraphics[width=0.4\linewidth]{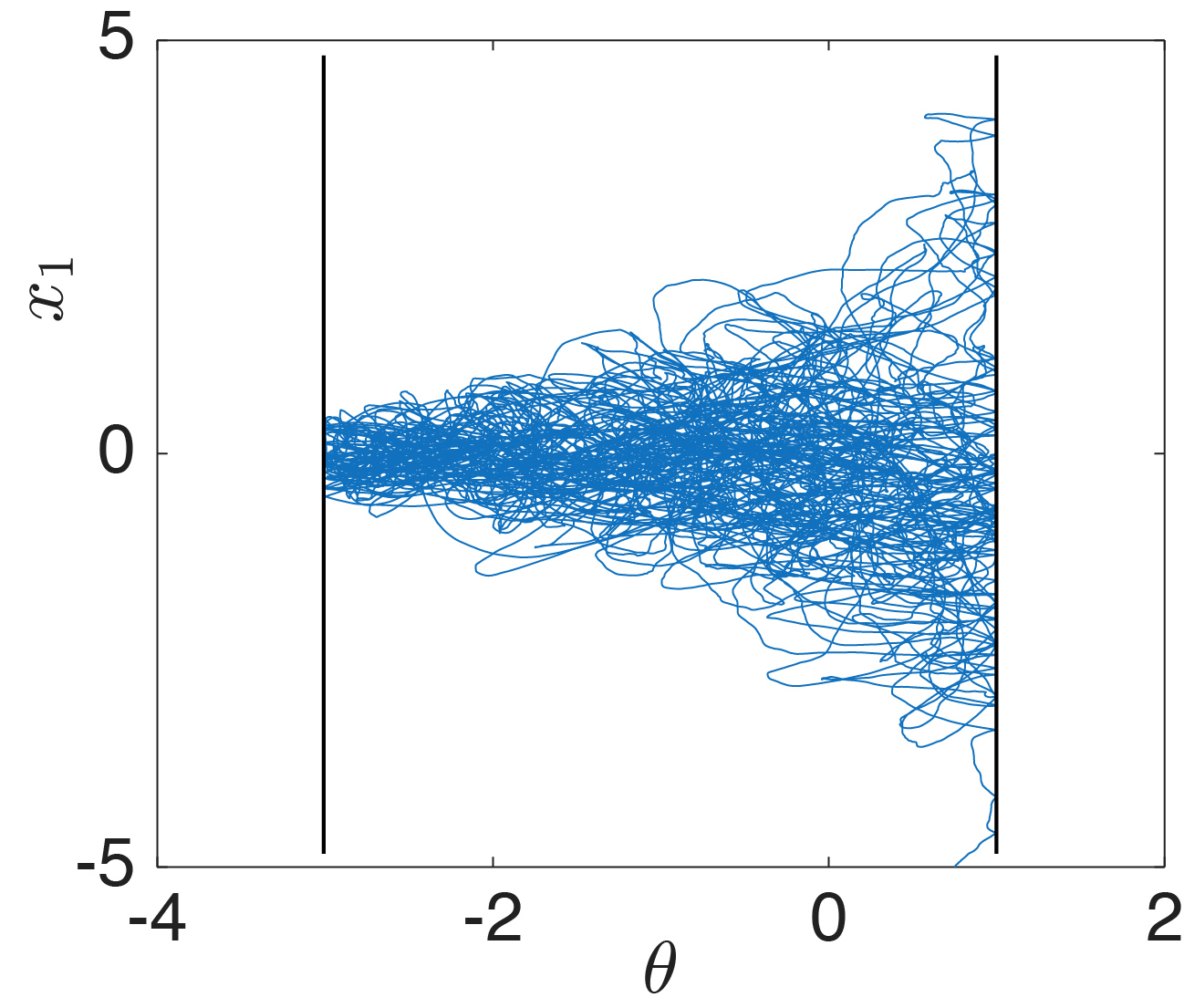}
    \caption{Experiment~\ref{exp_Neal_funnel}. Trajectorial evolution of [OBA$_c$BO] in the first two dimensions of the potential energy surface. }
    \label{neal_exp_traj}
\end{figure}

\begin{figure}[H]
    \centering
    \includegraphics[width=0.4\linewidth]{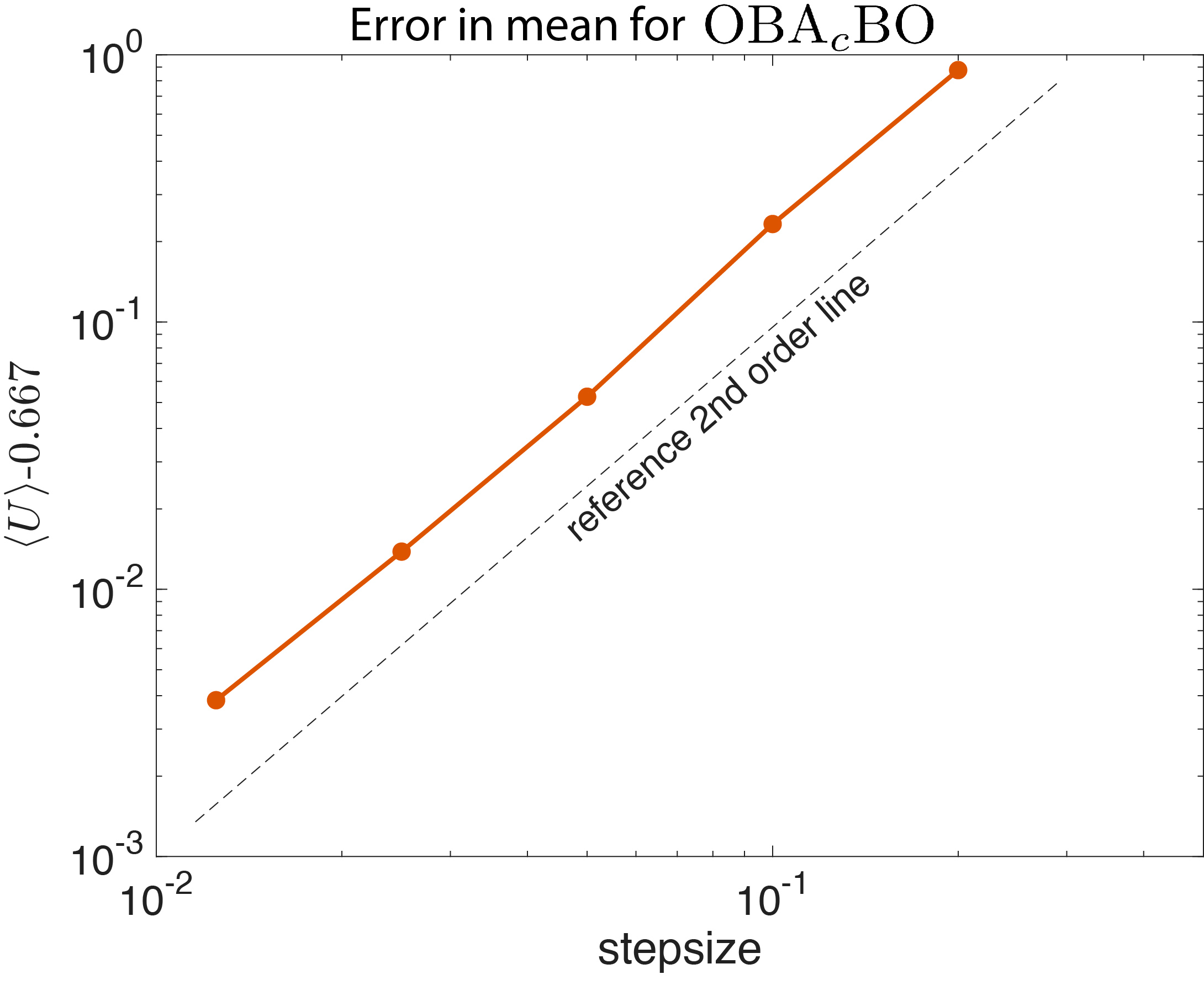}
    \caption{Experiment~\ref{exp_Neal_funnel}. Second-order convergence of [OBA$_c$BO]. Here, we have denoted with $\langle  U \rangle$ the estimator of $\mathbb{E}(U)$ where expectation is with respect to the Gibbs measure with potential $U$. }
    \label{Nealfunnel_OBABO}
\end{figure}

\begin{figure}[H]
    \centering
    \includegraphics[width=0.4\linewidth]{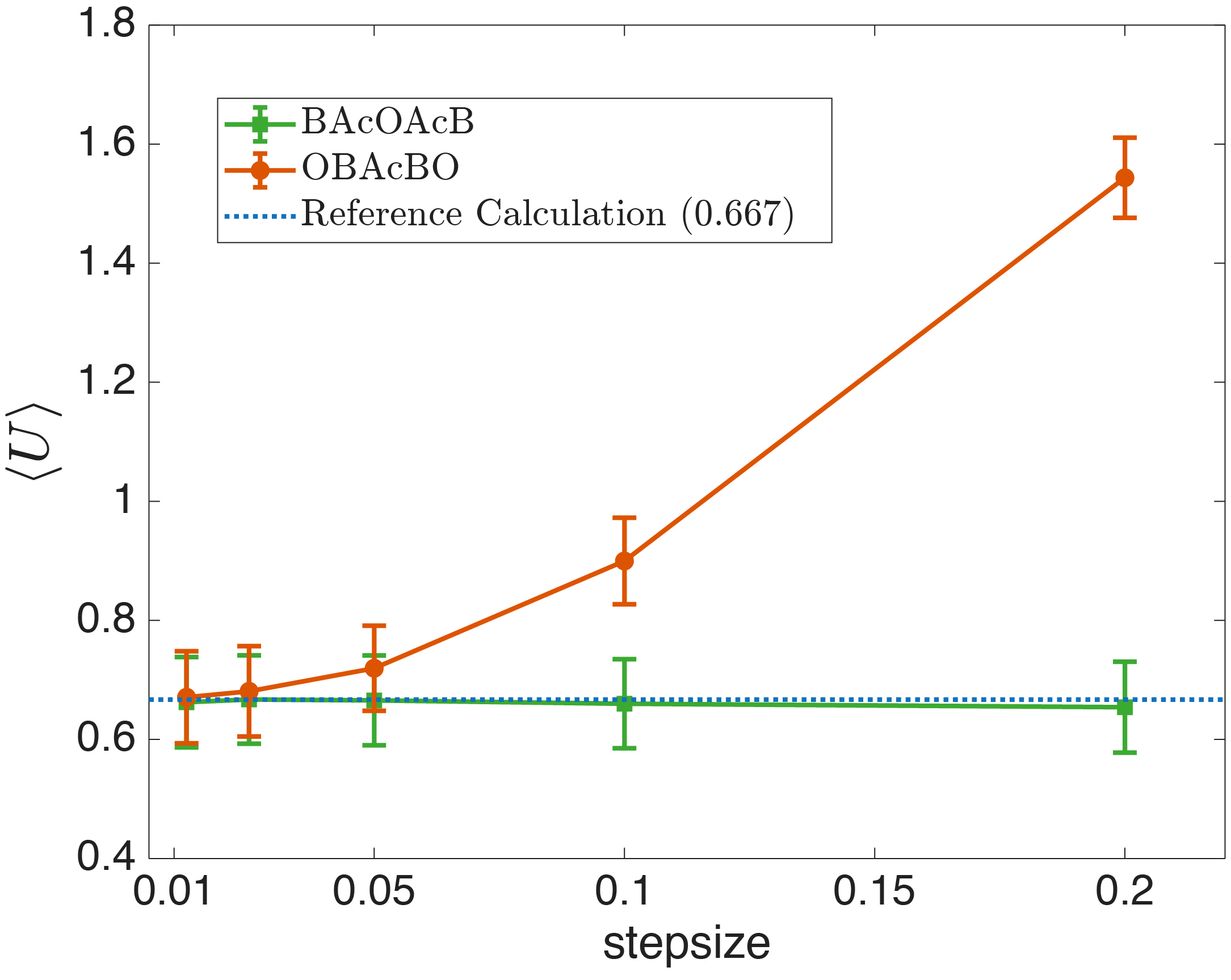}
    \caption{Experiment~\ref{exp_Neal_funnel}. Error bars denote Monte Carlo error. }
    \label{exact_BAOAB}
\end{figure}

\end{experiment}

\begin{experiment} (\textbf{Bayesian inference for SIR model}) 
The SIR (Susceptible-Infected-Recovered) model describes epidemic dynamics through a system of ordinary differential equations
\cite{SIR}:

\begin{align}
\frac{dS}{dt} &= -\eta \frac{SI}{N}, \label{eq:SIR} \\
\frac{dI}{dt} &= \eta \frac{SI}{N} - \alpha I , \notag \\
\frac{dR}{dt} &= \alpha I, \notag
\end{align}
where $S(t)$, $I(t)$, $R(t)$ are the numbers of susceptible, infected, and recovered individuals at time $t$ with $N = S + I + R$ is the total population being constant;  $\eta > 0$ is the transmission rate
and $\alpha > 0$ is the recovery rate. Then the basic reproduction number is $R_0 = \eta/\alpha$. 
We choose Gamma priors for both parameters: $
\eta \sim \text{Gamma}(2, 2) \quad \text{(shape = 2, rate = 2)}, \;\;
\alpha \sim \text{Gamma}(2, 4) \quad \text{(shape = 2, rate = 4)}$. We also include a prior belief that $\eta$ and $\alpha$ are greater than $1.5$. 
Given observed data $\mathbf{I}_{\text{obs}} = \{I_{\text{obs}}(t_1), \ldots, I_{\text{obs}}(t_n)\}$ and parameters $Q = (\eta, \alpha)$, likelihood function, denoted as $\mathcal{L}$, is:

\begin{align*}
\mathcal{L}(Q | \mathbf{I}_{\text{obs}}) &= \prod_{j=1}^n \rho(I_{\text{obs}}(t_j) | Q) 
= \prod_{j=1}^n \frac{1}{\sqrt{2\pi\sigma^2}} \exp\left(-\frac{(I_{\text{obs}}(t_j) - I_{\text{pred}}(t_j; Q))^2}{2\sigma^2}\right), \label{eq_cld_5.8}
\end{align*}
where $I_{\text{pred}}(t_j; \boldsymbol{\theta})$ is the predicted infected count from solving the SIR system (\ref{eq:SIR}) with parameters $Q$. By Bayes' theorem:
\begin{align*}
\rho(Q | \mathbf{I}_{\text{obs}}) &\propto \mathcal{L}(Q | \mathbf{I}_{\text{obs}}) \times \rho(Q), \\
\log \rho(Q | \mathbf{I}_{\text{obs}}) &= \log \mathcal{L}(Q | \mathbf{I}_{\text{obs}}) + \log \rho(\eta) + \log \rho(\alpha) + \text{constant}.
\end{align*}

 We take the total population as $1000$ and initial number of infected individuals as $10$: $I(0) = 10$, $R(0) = 0$ and $S(0) = 990 $.  As the ground truth, we take $\eta $ = 0.7 and $\alpha = 0.2$  (therefore, $R_0$ = 3.5) and  simulate  (\ref{eq:SIR}) over time $t \in [0, 50]$ to create $51$ noisy observations at regular unit intervals by adding Gaussian noise with mean $0$ and standard deviation $4$. Therefore, $n = 51$ in $\mathbf{I}_{\text{obs}} = \{I_{\text{obs}}(t_1), \ldots, I_{\text{obs}}(t_n)\}$. 
We use $\sigma  = 100$.

In this Bayesian inference example, we compare four algorithms. First two algorithms are two numerical schemes for reflected overdamped Langevin dynamics \eqref{rgsde}: projection scheme  and symmetrized reflection scheme \cite{2, lst23}. We call them Projected Langevin algorithm (PLA) and Reflected Langevin algorithm (RLA), respectively. The other two are [OBA$_c$BO] and [BA$_c$OA$_c$B] from Section~\ref{cld_sec3}.  For the sake of completeness, we mention below the one-step of procedure of PLA and RLA: 

 \begin{equation}
 \tag{PLA}
 \begin{aligned}
 Q_{k+1} = \Pi(Q_{k+1}^{'})
 \end{aligned}
 \end{equation}
 and
 \begin{equation}
 \tag{RLA}
 \begin{aligned}
 Q_{k+1} = Q_{k+1}^{'}  - 2 \dist(\Pi(Q_{k+1}^{'}), Q_{k+1}^{'})\Pi(Q_{k+1}^{'})
 \end{aligned}
 \end{equation}
 where $\Pi(q)$ denotes projection of $q$ on boundary $\partial G$ and
 \begin{align}
 Q_{k+1}' = Q_k + h \nabla \log \rho(Q_k | \mathbf{I}_{\text{obs}}) + \sqrt{2h} \xi_{k+1}
 \end{align}
 with $\xi_{k+1} \sim \mathcal{N}(\mathbf{0}, \mathbf{I})$. We approximate  $\nabla \log \rho(Q | \mathbf{I}_{\text{obs}})$ by forward finite difference with discretization parameter $\epsilon = 10^{-8}$.

 We compute the mean and standard deviation based on the time averaging estimators. 
The estimator used for estimating $R_0$ is
\begin{align}
    \frac{\sum_{i=1}^{M} \eta^{i}}{\sum_{i=1}^{M}\alpha^{i}},
\end{align}
where $M $ denotes the number of steps after the burn-in period of length $ 10$. We used $h=0.001, 0.0005$ and $T= 100$ resulting in $M = T/h - 10/h$ number of steps. 
Using the sample means of recovery and transmission rates, where samples are generated using five different algorithms, as the best estimates for the true rates, we simulate the evolution of the number of susceptible, infected, and recovered individuals with time from $0$ to $50$. In Fig.~\ref{fig:placeholder_noPq}, the solid line represents the evolution of SIR ODEs with true parameters, while dotted lines show the evolution with estimated parameters. For uncertainty quantification, we calculate a $95\%$ credible interval using the 95th percentile of samples. The superior performance of [OBA$_c$BO] and [BA$_c$OA$_c$B] methods is clear from Fig.~\ref{fig:placeholder_noPq}.

 Tables~\ref{tab:SIR_posterior_comparison_h_002}-\ref{tab:SIR_posterior_comparison_h_001} demonstrate that [BA$_c$OA$_c$B] provides the estimate $(\hat{\eta}, \hat{\alpha}) = (0.724, 0.216)$ for  $h= 0.002$ whereas comparable accuracy requires $h = 0.001$ for PLA and RLA. The same is true for [OBA$_c$BO]. 

Comparing Table~\ref{tab:SIR_posterior_comparison_h_001} and Table~\ref{tab:SIR_posterior_comparison_h_0.0005}, we see an improvement in estimates of PLA and RLA. However, no improvement is observed in case of [OBA$_c$BO] and [BA$_c$OA$_c$B] for $h = 0.0005$ suggesting that the statistical error  dominates the discretization error.

\begin{table}[htbp]
\centering
\caption{Posterior Statistics comparison $h = 0.002$.}
\begin{tabular}{|l|cc|cc|cc|}
\hline
\multirow{2}{*}{\textbf{Algorithm}} & \multicolumn{2}{c|}{\textbf{$\eta$ (transmission rate)}} & \multicolumn{2}{c|}{\textbf{$\alpha$ (recovery rate)}} & \multicolumn{2}{c|}{\textbf{$R_0 = \eta/\alpha$}} \\
\cline{2-7}
& \textbf{Mean} & \textbf{Std} & \textbf{Mean} & \textbf{Std} & \textbf{Mean} & \textbf{Std} \\
\hline
PLD  & 0.759 & 0.082 & 0.255 & 0.095 & 3.441 & 1.325 \\
RLD  & 0.746 & 0.078 & 0.256 & 0.080 & 3.195 & 1.010  \\
OBA$_c$BO & 0.724 & 0.011 & 0.216 & 0.049 & 3.542 & 0.854 \\
BA$_c$OA$_c$B & 0.724 & 0.011 & 0.217 & 0.049 & 3.529 & 0.851 \\
\hline
\textbf{True Values} & \multicolumn{2}{c|}{0.700} & \multicolumn{2}{c|}{0.200} & \multicolumn{2}{c|}{3.500} \\
\hline
\end{tabular}
\label{tab:SIR_posterior_comparison_h_002}
\end{table}

\begin{table}[htbp]
\centering
\caption{Posterior Statistics comparison $h = 0.001$.}
\begin{tabular}{|l|cc|cc|cc|}
\hline
\multirow{2}{*}{\textbf{Algorithm}} & \multicolumn{2}{c|}{\textbf{$\eta$ (transmission rate)}} & \multicolumn{2}{c|}{\textbf{$\alpha$ (recovery rate)}} & \multicolumn{2}{c|}{\textbf{$R_0 = \eta/\alpha$}} \\
\cline{2-7}
& \textbf{Mean} & \textbf{Std} & \textbf{Mean} & \textbf{Std} & \textbf{Mean} & \textbf{Std} \\
\hline
PLA  & 0.724 & 0.058 & 0.219 & 0.052 & 3.505 & 0.866 \\
RLA  & 0.724 & 0.057 & 0.222 & 0.049 & 3.413 & 0.766  \\
OBA$_c$BO & 0.703 & 0.002 & 0.202 & 0.001 & 3.476 & 0.011 \\
BA$_c$OA$_c$B & 0.702 & 0.002 & 0.202 & 0.001 & 3.482 & 0.009 \\
\hline
\textbf{True Values} & \multicolumn{2}{c|}{0.700} & \multicolumn{2}{c|}{0.200} & \multicolumn{2}{c|}{3.500} \\
\hline
\end{tabular}
\label{tab:SIR_posterior_comparison_h_001}
\end{table}


\begin{table}[htbp]
\centering
\caption{Posterior Statistics comparison $h = 0.0005$.}
\begin{tabular}{|l|cc|cc|cc|}
\hline
\multirow{2}{*}{\textbf{Algorithm}} & \multicolumn{2}{c|}{\textbf{$\eta$ (transmission rate)}} & \multicolumn{2}{c|}{\textbf{$\alpha$ (recovery rate)}} & \multicolumn{2}{c|}{\textbf{$R_0 = \eta/\alpha$}} \\
\cline{2-7}
& \textbf{Mean} & \textbf{Std} & \textbf{Mean} & \textbf{Std} & \textbf{Mean} & \textbf{Std} \\
\hline
PLA  & 0.714 & 0.049 & 0.209 & 0.033 & 3.501 & 0.577 \\
RLA  & 0.715 & 0.049 & 0.210 & 0.032 & 3.481 & 0.551 \\
OBA$_c$BO & 0.703 & 0.002 & 0.202 & 0.0004 & 3.481 & 0.005 \\
BA$_c$OA$_c$B & 0.702 & 0.002 & 0.202 & 0.0004 & 3.481 & 0.005 \\
\hline
\textbf{True Values} & \multicolumn{2}{c|}{0.700} & \multicolumn{2}{c|}{0.200} & \multicolumn{2}{c|}{3.500} \\
\hline
\end{tabular}
\label{tab:SIR_posterior_comparison_h_0.0005}
\end{table}

\begin{figure}
    \centering
    \includegraphics[width=1\linewidth]{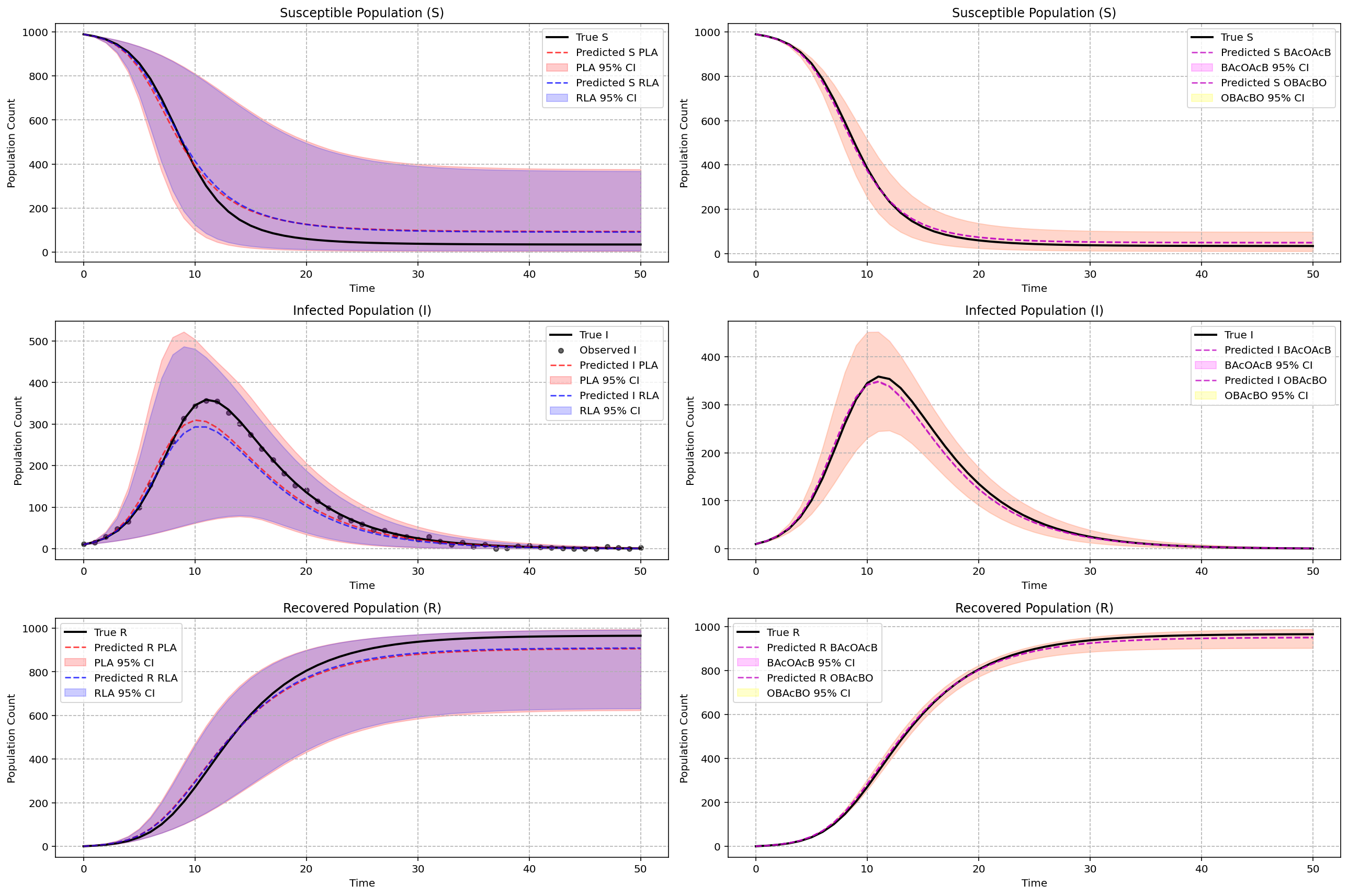}
    \caption{The solid line depicts the temporal evolution of the SIR model (\ref{eq:SIR}) using the true parameter values, i.e. $\eta = 0.7$ and $\alpha = 0.2$, whereas the dotted lines illustrate the corresponding SIR dynamics obtained with the estimated parameters where PLA, RLA, OBA$_c$BO, BA$_c$OA$_c$B are simulated using $h = 0.002$.}
    \label{fig:placeholder_noPq}
\end{figure}



\end{experiment}

\section{Conclusions}\label{sec:concl}

In this article, we have constructed and analyzed a family of numerical methods for simulating confined Langevin dynamics based on modification of splitting schemes for the unconfined case. The considered algorithms can be used in a variety of applications from molecular dynamics to statistical sampling from distributions with compact support and optimization with constraints. We proved finite-time convergence and convergence to ergodic limits for our first-order schemes. We observed and justified (in the case of  half-space) a surprising result that the proposed [A$_c$, B, O] St\"{o}rmer-Verlet-type splitting schemes are of second weak order despite the fact that their deterministic counterparts are only of first order due to the use of the simple collision step A$_c$. 

Our paper opens quite a few avenues for potential future research directions related to Langevin dynamics with specular reflection.  These include further testing of the proposed geometric integrators, e.g. applying them to train neural networks with imposed constraints on weights \cite{leimkuhler_constraint_net}. It is of interest to extend the developed approach to the case when particles are not only colliding with the boundary of a domain but  also  with each other. Constructing efficient numerical methods for Langevin equations with non-elastic collisions on the boundary is another  possible topic for future work. An additional line of research is approximation of McKean-Vlasov-type specularly reflected kinetic SDEs with applications to sampling and optimization (e.g., distribution dependence can occur due to covariance preconditioning \cite{garbuno2020interacting, leimkuhler2018ensemble, ringh2025kalman}) and their comparison with constrained sampling and optimization methods based on reflected (using local time) McKean-Vlasov SDEs  \cite{hinds2024well}. Our work also presents further motivation to studying linear second-order PDEs with specular boundary conditions building on the recent paper \cite{PDE25}. 
We made an experimental observation that for the considered methods, which are 2nd order for  confined Langevin dynamics, the corresponding deterministic schemes with random initial data are also of 2nd order (recall that with fixed initial data they are of 1st order); consequently, numerical analysis of such schemes for random and chaotic Hamiltonian systems is likely to be of interest.  Further theoretical analysis of the constructed second-order integrators deserves further attention too.

\section*{Acknowledgments}
AS and MVT were supported by the Engineering and Physical Sciences Research Council [grant number EP/X022617/1]. The authors express their  gratitude to the International Centre for Mathematical Sciences (ICMS) for its support via the research-in-groups scheme.  We are grateful to Professor Marvin Weidner for providing additional insights in the work \cite{PDE25}.
For the purpose of open access, the authors have applied a CC BY public copyright license to any Author Accepted Manuscript version arising.

\bibliographystyle{plain}
\bibliography{references}

\appendix

\section{Numerical integration of collisional  Hamiltonian dynamics}\label{sec:colisHD}

Here, we continue the discussion from Section~\ref{prlim_dis_sec} by providing further insights into Hamiltonian dynamics with collision. 
Consider the following Hamiltonian dynamics $(Q(t), P(t))$ with elastic reflection at the boundary $\partial G$:
\begin{align*}
    Q(t) &= Q_0 + \int_{0}^{t} P(s) ds, \quad \quad Q_0 \in G, \\ 
    P(t) &= P_0 -\int_{0}^{t} \nabla U(Q(s)) ds -\sum\limits_{0<s\leq t}2(P({s^{-}})\cdot n_{G}(Q(s)))n_{G}(Q(s))I_{ \partial G}(Q(s)), \quad\quad P_0 \in \mathbb{R}^d.
 \end{align*}
 It can easily be verified that 
 \begin{align}
     H(Q(T), P(T)) = H(Q_0, P_0),
 \end{align}
 where $H(q,p) = U(q) + |p|^2/2$. 
If $G = \mathbb{R}^d$, i.e. if there were no boundary and hence no collision, we know that the energy error incurred by Störmer-Verlet method is $\mathcal{O}(h^2)$, i.e., 
\begin{align}
    H(Q_N, P_N) - H(Q_0, P_0) = \mathcal{O}(h^2), 
\end{align}
where $(Q_N, P_N)$ are $N$-th step of Störmer-Verlet with $T = Nh$. 
 
Recall that in the error expansion (\ref{cld_neweq_4.45})-\ref{cld_neweq_4.47}) in Section~\ref{prlim_dis_sec}, we show that the collision time $\tau$ within a time step $h$ is critical for the error analysis. This collision time significantly changes the convergence properties: while standard deterministic Hamiltonian dynamics yields second-order accuracy, the introduction of collisions in deterministic collisional Hamiltonian dynamics reduces the method to first-order accuracy, i.e. 
\begin{align}
     H(Q_N, P_N) - H(Q_0, P_0) = \mathcal{O}(h),  
\end{align}
where $(Q_N, P_N)$ are $N$-th step of collisional Störmer-Verlet (i.e. [BA$_c$B] scheme). 
A numerical illustration comparing Störmer-Verlet  with its collisional counterpart is provided in 
Fig.~\ref{collis_Hamil_compare_append}. 

\begin{figure}[H]
  \centering
   
        \includegraphics[width=7cm]{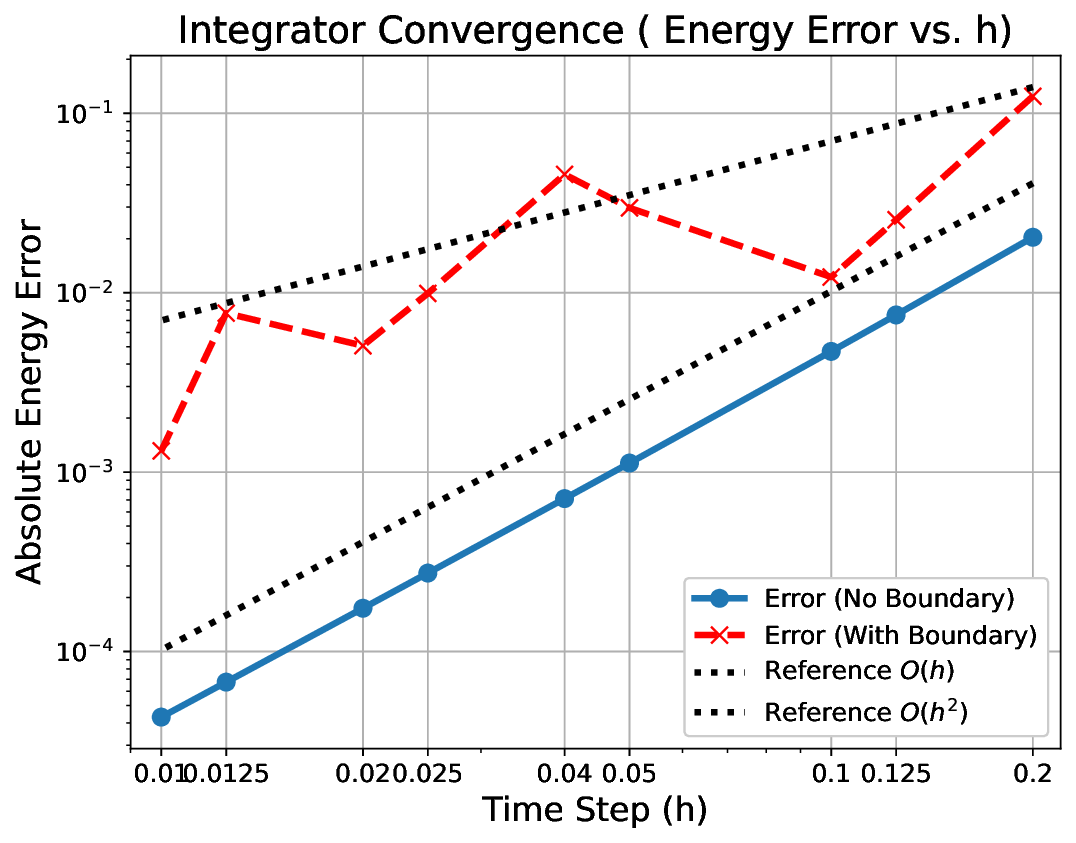}
 \captionof{figure}{The energy error for the Störmer-Verlet scheme and its collisional counterpart in the case of  quadratic potential $U(q) = |q|^2/2$,  $T = 5$ and the initial conditions $Q_0 =(0.1, 0.5)$ and $P_0 = (1.5, 1.5)$. For collisional dynamics, domain $G$ is the square $(0, 1)^2$.}\label{collis_Hamil_compare_append}

\end{figure}

We also recall that the stochastic version of the Störmer-Verlet scheme, e.g. [OBA$_c$BO] from Section~\ref{sec:2nd}, remains second-order even in the collisional setting due to the presence of noise (see Section~\ref{sec:2ndorder}). This  raises the natural question:  if we randomize the initial conditions of Hamiltonian dynamics, would we observe second-order convergence of the energy error in the collisional setting?
We provide a numerical illustration to answer this question. 

Consider the Hamiltonian dynamics with quadratic potential $U = |q|^2/2$ and with the fixed $Q_0 = (0.1, 0.5)\in G$ and i.i.d. $P_0^i \sim \mathcal{N}(0 , 1)$, $i =1,2$. We look at the two situations: $G = \mathbb{R}^2$ and $G = (0,1)^2$. In both situations, we compute the following error:
\begin{align}
    \text{Error} = | \mathbb{E}H(Q_N, P_N) - \mathbb{E}H(Q_0, P_0)|.
\end{align}
We use the Monte Carlo method with number of independent trajectories being $10^6$ to estimate the expectation.
We observe that the collisional Störmer-Verlet scheme has second-order convergence in the case of random initial condition, see Figure~\ref{random_collis_Hamil_append}. 
It is of interest to continue studying numerical approximation of random and chaotic Hamiltonian dynamics but this is not within the scope of this paper. 

\begin{figure}[H]
  \centering
   
        \includegraphics[width=7cm]{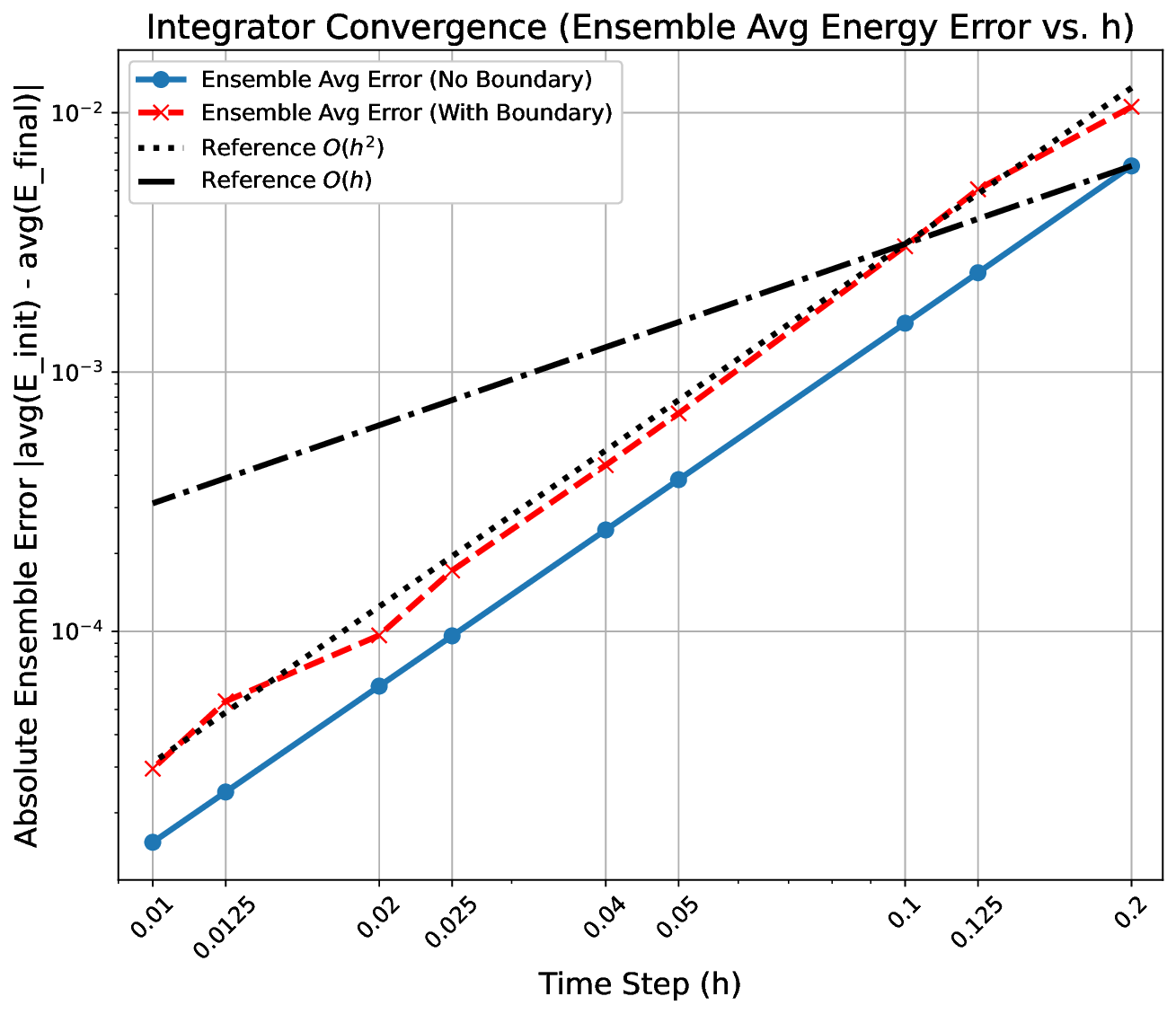}
 \captionof{figure}{
 The energy error for the Störmer-Verlet scheme and its collisional counterpart in the case of  quadratic potential $U(q) = |q|^2/2$,  $T = 1$ and initial points $Q_0 =(0.1, 0.5)$ and \textit{random} $P_0^i \sim \mathcal{N}(0 , 1)$, $i =1,2$. For collisional dynamics, domain $G$ is the square $(0, 1)^2$. }\label{random_collis_Hamil_append}

\end{figure}

\section{Existence of the density for [OBA$_c$BO]}\label{sec:density}

This appendix provides support to Lemma~\ref{lem:just2nd}, where it is assumed that at every step the Markov chain generated by [OBA$_c$BO] has a density. Here, for simplicity of the exposition we consider the case of $d=1$ for which we show that there is indeed a density.

Let $\Gamma = \Gamma(\beta,\gamma,h)= \sqrt{\frac{2\gamma}{\beta}(1 - e^{-\gamma h})}$. 
For $d=1$, [OBA$_c$BO] from (\ref{eq:halfspace}) can be written as a map $(Q, P) = T(\xi_1, \xi_2)$ taking the form  
\begin{align*}
    {\cal P}_{2}(\xi_1) &= e^{-h/2}p + \Gamma \xi_1 - \frac{h}{2} \frac{\partial U}{\partial q}(q) , \quad
    {\cal Q}(\xi_1) = q +  {\cal P}_{2}(\xi_1)h,
    \\ 
    s(\xi_1) &= \text{sign}({\cal Q}(\xi_1)), \quad  
    Q(\xi_1) = s(\xi_1){\cal Q}(\xi_1)=|{\cal Q}(\xi_1)|,   \\ 
    {\cal P}_{4}(\xi_1) &= s(\xi_1){\cal P}_{2}(\xi_1) - \frac{h}{2} \frac{\partial U}{\partial q}(Q(\xi_1)), \quad
     P(\xi_1, \xi_2) ={\cal P}_{4}(\xi_1) e^{-h/2} + \Gamma \xi_2.
\end{align*}
The distribution of $(Q, P)$ is the pushforward of the two dimensional standard Gaussian measure (generated by $\xi_1$ and $\xi_2$) under the map $T$. This distribution has a density if the Jacobian matrix of $T(z_1,z_2)$, 
$$
J_T = \begin{bmatrix}
\frac{\partial Q}{\partial z_1} & \frac{\partial Q}{\partial z_2} \\
\frac{\partial P}{\partial z_1} & \frac{\partial P}{\partial z_2}
\end{bmatrix},
$$
is non-singular (i.e., its determinant is non-zero) almost everywhere. Note that $s(z_1) = 0$ corresponds to 
$z_1=-\frac{1}{\Gamma}\left(q+e^{-h/2}p  - \frac{h}{2} \frac{\partial U}{\partial q}(q)\right)$, which is of measure zero with respect to the two dimensional standard Gaussian measure.  Consequently, if the Jacobian is non-zero for both $s= 1$ and $s=-1$, there is a density. 


We have
\begin{align}
    \frac{\partial Q}{\partial z_2} = 0,\quad \quad \quad 
    \frac{\partial P}{\partial z_2} = \Gamma  .
\end{align}
In the region $s(z_1) = 1$, we have
\begin{align*}
     \frac{\partial Q}{\partial z_1} &= \frac{\partial}{\partial z_1} ( q + h {\cal P}_2) =  h \frac{\partial{\cal P}_2}{\partial z_1} =  h \Gamma, \\
     \frac{\partial P}{\partial z_1}&=e^{-h/2}\frac{\partial{\cal P}_{4}(\xi_1) }{\partial z_1}
     =e^{-h/2}\Gamma\left[1 -\frac{h^2}{2} \frac{\partial^2 U}{\partial q^2}(Q) \right] .
\end{align*}
In the region $s(z_1) = -1$, we have
\begin{align}
\frac{\partial Q}{\partial z_1} &= -  h \Gamma, \quad  
\frac{\partial P}{\partial z_1}=-e^{-h/2}\Gamma\left[1 -\frac{h^2}{2} \frac{\partial^2 U}{\partial q^2}(Q) \right] .
\end{align}
Therefore, the Jacobian is
\begin{align}
    \text{det}(J_T) = s(z_1)h\Gamma^2 ,
\end{align}
 and hence the density exist. Since the potential $U(x)$ is assumed to be smooth, the map $T(z_1,z_2)$ is smooth, and it is not difficult to show that the density is smooth too.
 We note that in the appendix we have used arguments analogous to the ones appearing in the discrete Malliavin Calculus (see e.g. \cite[Chap. 2]{Nualart2006} and \cite[Chap. 5]{Nualart2009CBMS}).
\end{document}